\documentclass[11pt, reqno]{amsart}

\usepackage{amsmath, amsthm, amssymb}
\usepackage{mathtools,enumerate,comment,bm}
\usepackage[margin=1.5in]{geometry}
\usepackage{mathrsfs}
\usepackage{tikz}
\usepackage{tikz-cd}
\usepackage{ytableau}
\usepackage{arydshln}
\usepackage{graphicx}
\usepackage{mathdots}
\usepackage[colorlinks=true,urlcolor=blue,linkcolor=blue,citecolor=blue]{hyperref}
\usepackage{cleveref}
\usepackage{comment}

\newtheorem{thm}{Theorem}[section]
\newtheorem{prop}[thm]{Proposition}
\newtheorem{lem}[thm]{Lemma}

\newtheorem{cor}[thm]{Corollary}

\theoremstyle{remark}
\newtheorem{rem}[thm]{Remark}

\newtheorem{example}[thm]{Example}

\theoremstyle{definition}
\newtheorem{definition}[thm]{Definition}

\newcommand{\Z}{\mathbb{Z}}
\newcommand{\Q}{\mathbb{Q}}
\newcommand{\R}{\mathbb{R}}
\newcommand{\C}{\mathbb{C}}

\newcommand{\SL}{\mathbf{SL}}
\newcommand{\GL}{\mathbf{GL}}

\newcommand{\PGL}{\mathbf{PGL}}
\newcommand{\Ad}{\mathrm{Ad}}

\newcommand{\id}{\mathrm{id}}

\newcommand{\bfG}{\mathbf{G}}
\newcommand{\bfT}{\mathbf{T}}
\newcommand{\bfN}{\mathbf{N}}
\newcommand{\bfA}{\mathbf{A}}
\newcommand{\bfE}{\mathbf{E}}
\newcommand{\bfK}{\mathbf{K}}
\newcommand{\bfC}{\mathbf{C}}
\newcommand{\bfM}{\mathbf{M}}
\newcommand{\bfR}{\mathbf{R}}
\newcommand{\bfH}{\mathbf{H}}

\newcommand{\bfr}{\mathbf{r}}
\newcommand{\bfS}{\mathbf{S}}
\newcommand{\bfZ}{\mathbf{Z}}

\newcommand{\bbA}{\mathbb{A}}

\newcommand{\calA}{\mathcal{A}}
\newcommand{\calF}{\mathcal{F}}
\newcommand{\calG}{\mathcal{G}}

\newcommand{\scrC}{\mathscr{C}}

\newcommand{\Gm}{\mathbb{G}_m}

\newcommand{\Haar}{\mathrm{Haar}}
\newcommand{\vol}{\mathrm{vol}}
\newcommand{\de}{\,\mathrm{d}}
\newcommand{\disc}{\mathrm{disc}}
\newcommand{\absdisc}{D}
\newcommand{\intdisc}[2]{\disc_{\mathrm{int}}(#1,#2)}
\newcommand{\reldisc}[2]{\disc_{\mathrm{rel}}(#1,#2)}

\newcommand{\diag}{\mathrm{diag}}
\newcommand{\Res}{\mathrm{Res}}
\DeclareMathOperator{\Spec}{Spec}
\newcommand{\VK}{{\mathcal{V}_K}}
\newcommand{\VF}{{\mathcal{V}_F}}
\newcommand{\inv}[1]{\psi^{#1}}

\newcommand{\hatu}{{\hat{u}}}
\newcommand{\hatw}{{\hat{w}}}
\newcommand{\Quhat}{\Q_{\hatu}}
\newcommand{\Lwhat}{L_{\hatw}}

\newcommand{\Ell}{L}
\newcommand{\Ebar}{{\overline{E}}}

\newcommand{\Lbar}{{\overline{L}}}
\newcommand{\htheta}{h_{\theta}}
\newcommand{\hAr}{h_{\mathrm{Ar}}}

\newcommand{\OF}{\mathcal{O}_F}
\DeclareMathOperator{\Adeg}{\widehat{\deg}}
\newcommand{\frakp}{\mathfrak{p}}
\DeclareMathOperator{\covol}{covol}
\DeclareMathOperator{\rk}{rk}
\DeclareMathOperator{\Hom}{Hom}
\DeclareMathOperator{\tr}{tr}
\DeclareMathOperator{\fin}{{fin}}

\newcommand{\muYtilde}{\widetilde{\mu_Y}}
\newcommand{\muTtilde}{\widetilde{\mu_{[\bfT]}}}

\newcommand{\algS}{{\mathbf{F}_\infty}}
\newcommand{\algT}{{\mathbf{K}_\infty}}

\DeclareMathOperator{\Nr}{Nr} 
\DeclareMathOperator{\Tr}{Tr} 
\DeclareMathOperator{\Gal}{Gal} 

\DeclarePairedDelimiter{\norm}{\lVert}{\rVert}
\DeclarePairedDelimiter{\abs}{\lvert}{\rvert}

\definecolor{darkred}{rgb}{0.4,0,0}
\definecolor{darkgreen}{rgb}{0,0.5,0}
\definecolor{darkblue}{rgb}{0,0,0.4}

\newcommand{\pyadd}[1]{{\color{darkblue}{\tiny [PY]} #1}}

\definecolor{darkolivegreen}{rgb}{0.33, 0.42, 0.18}
\newcommand{\awnote}[1]{\marginpar{\color{darkolivegreen}\tiny [ALW] #1}}
\newcommand{\awadd}[1]{{\color{darkolivegreen}{\tiny [ALW]} #1}}

\newcommand\rquot[2]{
        \mathchoice
            {
                \text{\raise0.5ex\hbox{$#1$}\!\big/\!\lower0.5ex\hbox{$#2$}}%
            }
            {
                \raisebox{.4ex}{\small \newline ${#1}$}\!/\!\raisebox{-.4ex}{\small ${#2}$}
            }
            {
                \raisebox{.4ex}{\tiny \newline ${#1}$}\!/\!\raisebox{-.4ex}{\tiny ${#2}$}
            }
            {
                \raisebox{.4ex}{\tiny \newline ${#1}$}\!/\!\raisebox{-.4ex}{\tiny ${#2}$}
            }
    }
        
\newcommand\lquot[2]{
        \mathchoice
            {
                \text{\lower0.5ex\hbox{$#1$}\!\big\backslash\!\raise0.5ex\hbox{$#2$}}%
            }
            {
                \raisebox{-.4ex}{\small \newline ${#1}$}\!\backslash\!\raisebox{.4ex}{\small ${#2}$}
            }
            {
                \raisebox{-.4ex}{\tiny \newline ${#1}$}\!\backslash\!\raisebox{.4ex}{\tiny ${#2}$}
            }
            {
                \raisebox{-.3ex}{\tiny \newline ${#1}$}\!\backslash\!\raisebox{.3ex}{\tiny ${#2}$}
            }
    }

\newcommand\lrquot[3]{
        \mathchoice
            {
                \text{\lower0.5ex\hbox{$#1$}\!\big\backslash\!\raise0.5ex\hbox{$#2$}\!\big/\!
                \lower0.5ex\hbox{$#3$}}%
            }
            {
                \raisebox{-.4ex}{\small \newline ${#1}$}\!\backslash\!\raisebox{.4ex}{\small ${#2}$}\!/\!\raisebox{-.4ex}{\small \newline ${#3}$}
            }
            {
                \raisebox{-.4ex}{\tiny \newline ${#1}$}\!\backslash\!\raisebox{.4ex}{\tiny ${#2}$}\!/\!\raisebox{-.4ex}{\tiny \newline ${#3}$}
            }
            {
                \raisebox{-.3ex}{\tiny \newline ${#1}$}\!\backslash\!\raisebox{.3ex}{\tiny ${#2}$}\!/\!\raisebox{-.3ex}{\tiny \newline ${#3}$}
            }
    }

\newcommand{\rdslash}{\mathbin{/\mkern-6mu/}}
\newcommand{\ldslash}{\mathbin{\backslash\hspace*{-3pt}\backslash}}

\newcommand\lrdoublequot[3]{
        \mathchoice
            {
                \text{\lower0.5ex\hbox{$#1$}
                \hspace*{-0.5em}
                $\big\backslash\hspace*{-3pt}\big\backslash$
                \raise0.5ex\hbox{\hspace*{-0.3em}$#2$\hspace*{-0.3em}}
                $\big/\hspace*{-3pt}\big/$
                \lower0.5ex\hbox{\hspace*{-0.2em}$#3$}}%
            }
            {
                \raisebox{-.4ex}{\small \newline ${#1}$}\!\ldslash\!\raisebox{.4ex}{\small ${#2}$}\!\rdslash\!\raisebox{-.4ex}{\small \newline ${#3}$}
            }
            {
                \raisebox{-.4ex}{\tiny \newline ${#1}$}\!\ldslash\!\raisebox{.4ex}{\tiny ${#2}$}\!rdslash\!\raisebox{-.4ex}{\tiny \newline ${#3}$}
            }
            {
                \raisebox{-.3ex}{\tiny \newline ${#1}$}\!\ldslash\!\raisebox{.3ex}{\tiny ${#2}$}\!rdslash\!\raisebox{-.3ex}{\tiny \newline ${#3}$}
            }
    }

\newcommand{\cspl}{\mathsf{g}}
\newcommand{\intgroup}{\bfR}
\newcommand{\inttorus}{\bfS}
\newcommand{\perm}{\mathcal{S}_{\mathrm{sp}}}
\newcommand{\decayone}{\eta}
\newcommand{\entr}{\mathrm{h}}
\newcommand{\hint}{\mathrm{h}_{\mathrm{int}}}

\numberwithin{equation}{section}
\title{A uniform Linnik basic lemma and entropy bounds}
\author{Andreas Wieser}
\address{Einstein Institute of Mathematics, Edmund J.~Safra Campus, Givat Ram, Jerusalem, Israel}
\email{andreas.wieser@mail.huji.ac.il}
\author{Pengyu Yang}
\address{Morningside Center of Mathematics, Chinese Academy of Sciences,
Beijing 100190}
\email{yangpengyu@amss.ac.cn}

\thanks{A.W.~was supported by SNF grant 178958, SNF Doc.~Mobility grant 195737 as well as the ERC 2020 grant no.~833423.}

\begin{document}

\begin{abstract}
We prove a version of Linnik's basic lemma uniformly over the base field using $\theta$-series and geometric invariant theory in the spirit of Khayutin's approach (Duke Math.~J., 168(12), 2019).
As an application, we establish entropy bounds for weak${}^\ast$-limits of invariant measures on homogeneous toral sets in $\GL_4$ of biquadratic, cyclic, or dihedral type.
\end{abstract}

\maketitle

\section{Introduction}

This article studies homogeneous toral sets in the finite volume homogeneous space
\begin{align*}
[\GL_n] = \lquot{\GL_n(\Q)}{\GL_n(\bbA)^1}
\end{align*}
for $n\geq 1$ and entropy of weak${}^\ast$-limits of their homogeneous measures.
Here, the group $\GL_n(\bbA)^1$ consists of the \emph{content one} elements of $\GL_n(\bbA)$ i.e.~elements $g$ with $\prod_v|\det(g_v)|_v = 1$ where $v$ runs over all places of $\Q$.
A homogeneous toral set is a subset of the form $[\bfT g] = \GL_n(\Q)\bfT(\bbA)^1 g$, where $\bfT < \GL_2$ is a $\Q$-torus of rank $1$, $\bfT(\bbA)^1 = \GL_n(\bbA)^1 \cap \bfT(\bbA)$, and $g \in \GL_n(\bbA)^1$ is a `shift'.
Note that any such $\Q$-torus $\bfT$ is isomorphic to the restriction of scalars $\Res_{K/\Q}(\mathbb{G}_{m,K})$ of the multiplicative group $\mathbb{G}_{m,K}$ over some number field $K$ of degree $n$.

\subsection{Some historical context}
Let us begin with some of the history and known results towards the equidistribution problem for homogeneous toral sets.

Consider first $n=2$.
The problem of equidistribution of homogeneous toral sets in $[\GL_2]$ 
is strongly related to Linnik-type problems. These Linnik-type problems include:
\begin{itemize}
\item Equidistribution of complex multiplication (CM) points on the complex modular curve $Y_0(1)$.
\item Equidistribution of long periodic geodesics on the unit tangent bundle of $Y_0(1)$.
\end{itemize}
We refer to \cite{DukeSurvey}, \cite{MV-ICM}, and \cite{ELMV12} for an introduction to these as well as other Linnik-type problems (such as equidistribution of integral points on spheres).
Duke \cite{Duk88} in the late 80's resolved the above equidistribution problems building on a breakthrough of Iwaniec \cite{Iwa87}.
Prior to these works, Linnik~\cite{Lin68} and Skubenko~\cite{skubenko} proved the same result under an additional congruence condition, which is often called a Linnik-type condition.

The above problems can be presented in a unified setting by working over the ring of adeles and considering homogeneous toral sets in $[\GL_2]$.
In this unified setting, the problem is to establish the following statement:
for any compactly supported continuous function $f$ on $[\GL_2]$ orthogonal to the character spectrum\footnote{The character spectrum consists of the functions on $[\GL_2]$ arising from the abelian quotient $[\GL_2]/\SL_2(\bbA)$. Equidistribution as in \eqref{eq:equicharspec} for functions in the character spectrum can also be understood but is not guaranteed in general.} of $[\GL_2]$,
\begin{align}\label{eq:equicharspec}
\int_{[\bfT g]} f \de \mu_{[\bfT g]} \to \int_{[\GL_2]} f\de\mu_{[\GL_2]}
\end{align}
when $\vol([\bfT g]) \to \infty$.
Here, $\mu_{[\bfT g]}$ denotes the $g^{-1}\bfT(\bbA)^1g$-invariant probability measure on the orbit $[\bfT g]$ and, similarly, $\mu_{[\GL_2]}$ is the $\GL_2(\bbA)^1$-invariant probability measure on $[\GL_2]$.
We will define the volume $\vol([\bfT g])$ of a homogeneous toral set $[\bfT g]$ later but note here that it is closely related to an arithmetic quantity one can also attach to $[\bfT g]$ (a discriminant); both quantities should be seen as a measure of `complexity'.

One can show that \eqref{eq:equicharspec} solves both Linnik-type problems mentioned above (equidistribution of CM points and periodic geodesics).
From a dynamicist's perspective, the above adelic viewpoint is particularly appealing as it captures more invariance than is apriori given.
(For instance, Duke's theorem for equidistribution of CM points on the complex modular curve $Y_0(1)$ describes the behaviour of a \emph{finite} set of points.)

We also note that \eqref{eq:equicharspec} is also known when $\GL_2$ is replaced by an inner form $\bfG$ of $\GL_2$ i.e.~$\bfG$ is the group of invertible elements in a quaternion algebra over $\Q$.
This, yet again broader, setting includes many interesting applications as well, for example to equidistribution problems for integer points on spheres or for reductions of CM elliptic curves \cite{Michel-subconvex}.
Equidistribution of homogeneous toral sets as in \eqref{eq:equicharspec} is due to a long list of authors depending on different ways of varying the homogeneous toral sets -- see \cite[\S4]{ELMV11Annals} for a formulation and references.

In the context of this article, let us also mention the dynamical approach of Einsiedler, Lindenstrauss, Michel, and Venkatesh \cite{ELMV12} for $[\GL_2]$ (see also \cite{W-linnik}) which reinterprets Linnik's work.
Whenever $[\bfT_i g_i]$ is a sequence of homogeneous toral sets with common invariance under a split torus $A$ at some place of $\Q$ (this is the Linnik-type condition), one shows that any weak${}^\ast$-limit has maximal entropy for any non-trivial element in $A$.
As all measures of maximal entropy are invariant under $\SL_2(\bbA)$ (see e.g.~\cite[Thm.~7.9]{pisaEL}), this proves equidistribution as in \eqref{eq:equicharspec}.

Equidistribution of homogeneous toral sets in $[\GL_3]$ with common invariance has been established by Einsiedler, Lindenstrauss, Michel, and Venkatesh in \cite{ELMV11Annals}.
Roughly speaking, the strategy they employ is to show that \emph{any ergodic component} of a weak${}^\ast$-limit of homogeneous toral measures has positive entropy.
This in turn implies that each of these ergodic components is the Haar probability measure on $[\GL_3]$ by deep results of Einsiedler, Katok, and Lindenstrauss \cite{EKL06} and Einsiedler, Lindenstrauss \cite{EL1,EL2}. 
The positive entropy follows from subconvex estimates for Dedekind $\zeta$-functions for cubic fields -- see the references in \cite[App.~A]{ELMV11Annals}. (Of course, positive entropy for all ergodic components would also follow from maximal entropy for the weak${}^\ast$-limit.)
The results of \cite{ELMV11Annals} are in fact more general and apply to certain homogeneous toral sets in $\GL_n$ for $n>3$, $n$ prime.

In the spirit of Linnik's work as reinterpreted by \cite{ELMV12},
one could ask whether it is possible to show that all weak${}^\ast$-limits have maximal entropy.
In this sense, lower bounds on the entropy of weak${}^\ast$-limits can be seen as progress towards establishing \eqref{eq:equicharspec} for general $\GL_n$; this is one of the objectives of the current article.

\subsection{Entropy bounds}
Consider a sequence of homogeneous toral sets 
\begin{align*}
Y_i = \GL_n(\Q)\bfT_i(\bbA)^1 g_i \subset [\GL_n]
\end{align*}
where $\bfT_i \simeq \Res_{K_i/\Q}(\mathbb{G}_{m,K_i})$ for some number field $K_i$ of degree $n$ and $g_i \in \GL_n(\bbA)^1$.
To each $Y_i$ one can associate an order in $K_i$ and define a discriminant $\disc(Y_i)$ which is the product of the discriminant of this order and a measure of distortion at the Archimedean place. We refer to \cite[\S4]{ELMV11Annals} and \S\ref{sec:discvol} below for an exact definition.
We say that $Y_i$ is of maximal type if the associated order is the ring of integers in $K_i$.
The homogeneous toral set $Y_i$ also comes with a notion of volume $\vol(Y_i)$ related to the discriminant through $\vol(Y_i) = \disc(Y_i)^{\frac{1}{2}+o(1)}$ (cf.~\cite{ELMV11Annals}).

Assume that the discriminants $\disc(Y_i)$ go to infinity for $i \to \infty$.
As before, we write $\mu_{Y_i}$ for the $g_i^{-1}\bfT_i(\bbA)^1 g_i$-invariant probability measure on $Y_i$. 
Suppose that the measures $\mu_{Y_i}$ converge in the weak${}^\ast$-topology to some probability\footnote{In some cases, non-escape of mass is known. For instance when $K_i/\Q$ is abelian, subconvexity of the Dedekind $\zeta$-function $\zeta_{K_i}(s)$ is known by Burgess' bound and non-escape of mass holds by \cite{ELMV11Annals}.} measure $\mu$.
Furthermore, suppose that there exists a place $u$ of $\Q$ and a split $\Q_u$-torus $\bfA < \GL_n$ so that every $\mu_{Y_i}$ is invariant under $A = \bfA(\Q_u)$.
If $u \neq \infty$, we also assume that we have a fixed invariance subgroup at the Archimedean place i.e.~that $g_{i,\infty}^{-1}\bfT_ig_{i,\infty}$ is independent of $i$.
For the purposes of this article, there is no difference between $u$ Archimedean and $u$ non-Archimedean.

For any $A$-invariant probability measure $\nu$ on $[\GL_n]$ we write $h_\nu(a)$ for the Kolmogorov-Sinai entropy of $a \in A$ with respect to $\nu$. 
We write $h_{[\GL_n]}(a)$ instead of $ h_{\mu_{[\GL_n]}}(a)$ for simplicity where $\mu_{[\GL_n]}$ is the $\GL_n(\bbA)^1$-invariant probability measure on $[\GL_n]$.
In an ideal situation, one would be able to show that the weak${}^\ast$-limit $\mu$ satisfies $h_\mu(a) = h_{[\GL_n]}(a)$ as is the case for $n=2$.
Indeed, this would imply that $\mu$ is invariant under $\SL_n(\bbA)$ and thus establish equidistribution for functions orthogonal to character spectrum as in \eqref{eq:equicharspec}.
The objective in the following is to obtain progress by establishing bounds of the kind 
\begin{align*}
h_\mu(a) \geq \eta
\end{align*}
for $\eta>0$ (where the amount of progress could be quantified by the size of $\eta$).

The most general such bound to date has been established in \cite[Thm.~3.1]{ELMV09Duke} showing that
\begin{align*}
h_\mu(a) \geq \tfrac{1}{2} \min_{\phi \in \Phi} |\phi(a)|
\end{align*}
where $\Phi$ is the set of roots.
In \cite{Kha15}, Khayutin used novel techniques originating in geometric invariant theory (GIT) to prove the lower bound
\begin{align*}
h_\mu(a) \geq \frac{ h_{[\GL_n]}(a)}{2(n-1)}
\end{align*}
under the additional assumption that the Galois groups $\mathcal{G}_i$ of the Galois closure $L_i$ of $K_i$ over $\Q$ act two-transitively on the roots of the characteristic polynomial of a generator of $L_i$.
Khayutin's bound is in general drastically better than the bound in \cite{ELMV09Duke}. For instance, if we take $u=\infty$, $A$ the full diagonal subgroup, and $a=\diag(\mathrm{e}^{(n-1)/2},\mathrm{e}^{(n-3)/2},\ldots, \mathrm{e}^{-(n-1)/2})$ then Khayutin's bound loses a linear factor in $n$ while the loss in \cite{ELMV09Duke} is cubic in $n$.

One of the aims of this article is to extend \cite{Kha15} to all homogeneous toral sets in $[\GL_4]$ for which the Galois group does not act two-transitively. 

\begin{rem}[Galois types]\label{rem:galoistypes}
Let $K$ be a quartic field and let $\mathcal{G}$ be the Galois group of its Galois closure over $\Q$.
Then one of the following is true:
\begin{itemize}
\item $\mathcal{G} \simeq \rquot{\Z}{4\Z}$ ($K$ is \emph{cyclic}).
\item $\mathcal{G}\simeq \rquot{\Z}{2\Z}\times \rquot{\Z}{2\Z}$ ($K$ is \emph{biquadratic}).
\item $\mathcal{G} \simeq D_4$ ($K$ is \emph{dihedral}).
\item $\mathcal{G} \simeq A_4$.
\item $\mathcal{G} \simeq S_4$.
\end{itemize} 
In particular, the two-transitive Galois types are exactly $A_4$ and $S_4$. In all other cases, $K$ contains a quadratic subfield.
We also remark that in terms of density $\approx 17.11\%$ of quartic fields are dihedral (when ordered by discriminant) and $\approx 82.89\%$ are of type $S_4$ by a result of Bhargava \cite{BhargavaQuartic}; the other cases have density zero.
\end{rem}

In the following, we suppose that all of the quartic fields $K_i$ contain a quadratic subfield $F_i$.
Then $K_i^\times$ contains $F_i^\times$ and thus $\bfT_i$ has a subtorus
\begin{align*}
\Res_{F_i/\Q}(\mathbb{G}_{m,F_i}) \simeq \inttorus_i \subset \bfT_i.
\end{align*}
In particular, $\bfT_i$ is contained in the centralizer $\intgroup_i$ of $\inttorus_i$ in $\GL_4$; one can verify that
\begin{align}\label{eq:intermediategroupdef}
\intgroup_i \simeq \Res_{F_i/\Q}(\GL_2).
\end{align} 
(Note that this is a phenomenon unseen in dimensions $n=2,3$.)
We remark that in the dihedral or in the cyclic case, the subfield $F_i$ is unique. In the biquadratic case, there are three subfields and we have fixed one of them.
For the readers' convenience, we work out an explicit example.

\begin{example}\label{exp:biquadratic}
Suppose that $K = \Q(\sqrt{d_1},\sqrt{d_2})$ for two non-square integers $d_1,d_2$ with $d_1 \not\in \{d_2^n:n \geq 0\}$.
Representing multiplication in the basis $1,\sqrt{d_1},\sqrt{d_2},\sqrt{d_1d_2}$ embeds $K$ into $4\times4$-matrices over $\Q$ as
\begin{align*}
\left\{\begin{pmatrix}
x_1 & x_2 & x_3 & x_4 \\
d_1x_2 & x_1 & d_1x_4  & x_3 \\
d_2x_3 & d_2x_4 & x_1 & x_2 \\
d_1d_2x_4 & d_2x_3 &d_1x_2 & x_1 
\end{pmatrix}:x_1,x_2,x_3,x_4\in \Q \right\}.
\end{align*}
The subfield $F=\Q(\sqrt{d_1})$ is mapped to the above matrices with $x_3=x_4=0$. Denote by $\bfT>\bfS$ the corresponding tori in $\GL_4$. Then the centralizer of $\bfS$ is
\begin{align*}
\bfR=\left\{\begin{pmatrix}
x_1 & y_1 & x_2 & y_2 \\
d_1y_1 & x_1 & d_1y_2  & x_2 \\
x_3 & y_3 & x_4 & y_4 \\
d_1y_3 & x_3 &d_1y_4 & x_4 
\end{pmatrix}\right\}
\end{align*}
which clearly contains $\bfT$ and is isomorphic to $\mathrm{Res}_{F/\Q}(\mathbb{G}_{m,F})$.
Note that the other quadratic subfields of $K$ are $\Q(\sqrt{d_2})$ and $\Q(\sqrt{d_1d_2})$; each of these yields another subgroup of $\GL_4$ containing $\bfT$ in the same manner as for $\Q(\sqrt{d_1})$.

From the above embedding of $K$ into $\bfM_4(\Q)$ one can readily construct a homogeneous toral set $Y=\GL_4(\Q)\bfT(\bbA)^1 g$ with $g \in \SL_n(\R)$, though we caution readers that $Y$ will not be of maximal type in general.
Indeed, for instance if $d_1,d_2$ are squarefree and $d_1\equiv d_2 \equiv 1\mod 4$ the ring of integers of $K$ can be shown to be $\Z[\frac{1+\sqrt{d_1}}{2},\frac{1+\sqrt{d_2}}{2}] \supsetneq \Z[\sqrt{d_1},\sqrt{d_2}]$; the above example is easily adapted to treat this case though.
\end{example}

We return to our general setup.
For $a \in A$ the $g_i^{-1}\intgroup_i(\bbA)^1g_i$-invariant probability measure $\mu_{[\intgroup_i g_i]}$ on
\begin{align*}
[\intgroup_i g_i] = \GL_4(\Q)\intgroup_i(\bbA)^1g_i
\end{align*}
has positive entropy $h_{[\intgroup_i g_i]}(a)$ whenever $a \not\in g_{i,u}^{-1}\inttorus_i(\Q_u)g_{i,u}$.
Indeed, $a$ acts non-trivially on the Lie algebra $\mathrm{Lie}(g_{i,u}^{-1}\intgroup_i(\Q_u)g_{i,u})$ and expands/contracts certain directions.

As $A = \bfA (\Q_u) = g_i^{-1}\bfT_i(\Q_u)g_i$ is fixed (i.e.~independent of $i$), the subtorus $g_{i,u}^{-1}\inttorus_ig_{i,u}<A$ has finitely many options (it is characterized by triviality of some of the roots).
By restricting to a subsequence, we shall assume without loss of generality that the $\Q_u$-torus $g_{i,u}^{-1}\inttorus_ig_{i,u}$ and hence also the group $g_{i,u}^{-1}\intgroup_i(\Q_u)g_{i,u}$ are independent of $i$.
In this case, the entropy $h_{[\intgroup_i g_i]}(a)$ is independent of $i$; we shall denote it in the following by $\hint(a)$.

Let $A' \subset A$ be the set of $a \in A$ with $\hint(a) \leq \frac{1}{3}h_{[\GL_4]}(a)$.
If $u$ is the real place, we will see that $\log(A')$ is a union of closed Weyl chambers.
For any number field $F$, let $D_F$ denote the absolute value of its discriminant.

The following theorem is one of the main results of this article.

\begin{thm}\label{thm:main}
Let $Y_i$ be a sequence of $A$-invariant homogeneous toral sets in $[\GL_4]$ of maximal type that satisfies the above assumptions and choices.
In particular, for each $i$ the quartic field $K_i$ associated to the homogeneous toral set $Y_i$ contains a (chosen) quadratic subfield $F_i$ and the probability measures $\mu_{Y_i}$ converge in the weak${}^\ast$ topology to a probability measure $\mu$.

Suppose additionally that there exists $c>0$ with $\absdisc_{K_i}\geq c\absdisc_{F_i}^6$ for all $i$.
Then
\begin{align*}
\entr_\mu(a) \geq \hint(a)
\end{align*}
for all $a \in A'$.
\end{thm}

Let us illustrate the condition $\absdisc_{K_i}\geq c\absdisc_{F_i}^6$ in the context of Example~\ref{exp:biquadratic}: 
Suppose $K=\Q(\sqrt{d_1},\sqrt{d_2})$ for two distinct squarefree coprime integers $d_1,d_2$, then up to factors of $2$ the discriminant of $K$ is $d_1^2d_2^2$ (by the assumptions on $d_1,d_2$) and the discriminant of $F=\Q(\sqrt{d_1})$ is $d_1$. Thus, $\absdisc_{K}\gg \absdisc_{F}^6$ is equivalent to $d_2 \gg d_1^2$.

More remarks are in order:

\begin{rem}[Optimality]
With the imposed assumptions, the entropy bound in Theorem~\ref{thm:main} is best possible.
Indeed, for any subgroup $\Res_{F/\Q}(\GL_2) = \intgroup < \GL_4$ the homogeneous set $[\intgroup]$ contains sequences of homogeneous toral sets with growing discriminant.
This situation does not occur under the assumptions of Theorem~\ref{thm:main} when the (minimal) discriminant of the quadratic subfields tends to $\infty$.
In this case, the conjectured entropy bound is $\entr_\mu(a) \geq \entr_{[\GL_4]}(a)$.
Using a uniform version of Linnik's theorem with respect to the (in the application quadratic) base field we can improve upon Theorem~\ref{thm:main} when the discriminant of the intermediate field grows polynomially in the discriminant of the quartic field -- see \Cref{thm:refinement of main} below.
\end{rem}

\begin{rem}[Abelian quartic fields]
For abelian quartic fields (i.e.~biquadratic or cyclic ones) subconvexity of the associated Dedekind $\zeta$-functions is known.
In particular, \cite{ELMV11Annals} also yields entropy bounds for such homogeneous toral sets. 
These bounds in fact match the bounds obtained here.
While our focus in the current article lies on dihedral quartic fields, we also carry out the argument in the abelian case for completeness.
Note that amongst the fields considered here, dihedral fields are generic (by Bhargava's result mentioned earlier). For dihedral quartic fields, subconvexity is not known to our knowledge.
\end{rem}

\begin{rem}[First generalizations]
There is likely no major obstruction to generalizing Theorem~\ref{thm:main} to homogeneous toral sets in $[\bfG]$ where $\bfG$ is the group of units in a central simple algebra over $\Q$ of degree $4$ (i.e.~$\bfG$ is an inner form of $\GL_4$).
In principle, the arguments of this article should also apply to more general towers of extensions $\Q \subset F \subset K$ when $[K:F] = 2$.
\end{rem}

\begin{rem}
In an unpublished preprint of Ilya Khayutin and A.W., the authors improve upon the entropy bound in Theorem~\ref{thm:main} under stronger assumptions on the growth of the discriminant relative to the discriminant of the quadratic subfield.
\end{rem}

\subsection{Linnik's theorem and Bowen decay uniformity}\label{sec:unifLinnikIntro}

In this section, we explain the aforementioned version of Linnik's theorem with uniformity over the base field.
Let $F$ be a number field of degree $n$ over $\Q$ and consider the $F$-group $\GL_{2,F}$.
As before, we write 
\begin{align*}
[\GL_{2,F}] = \lquot{\GL_2(F)}{\GL_2(\bbA_F)^1}
\end{align*}
and equip it with its invariant probability measure $\mu_{[\GL_{2,F}]}$.
By a homogeneous toral set in $[\GL_{2,F}]$ we mean a set of the form 
\begin{align*}
Y = [\bfT g]= \GL_2(F)\bfT(\bbA_F)^1g
\end{align*}
where $g \in \GL_2(\bbA_F)^1$ and $\GL_{2,F}>\bfT \simeq \Res_{K/F}(\mathbb{G}_{m,K})$ for a quadratic extension $K/F$.
As before, we refer to \S\ref{sec:discvol} below for the notions of discriminant $\disc(Y)$, volume $\vol(Y)$, and maximal type for $Y$.
Write $\mu_Y$ for the invariant probability measure on $Y$.

Let $S$ be a finite set of places of $F$. For each $v\in S$, let $\bfA_v$ denote a split $F_v$-torus of $\GL_2(F_v)$, and let $A_v=\bfA_v(F_v)$ be its group of $F_v$-points. 
Let $A_S:=\prod_{v\in S}A_v$ and consider $A_S$ as a subgroup of $\GL_2(\bbA)$. 
Suppose that for all $i$ and all $v\in S$, we have $g_{i,v}^{-1}\bfT_ig_{i,v}=\bfA_v$. 
We fix $$a=(a_v)_{v\in S}\in A_S.$$ 
Let $B \subset \GL_2(\bbA_F)$ be an open set of the form $B = \prod_u B_u$ where $u$ runs over all places of $F$ and $B_u \subset \GL_2(F_u)$ is a bounded open neighborhood of the identity.
For any integer $\tau \geq 1$ we define the $\tau$-Bowen ball
\begin{align*}
B_\tau = \bigcap_{-\tau \leq t \leq \tau} a^{-t}B a^t.
\end{align*}

\begin{thm}[Uniform Linnik's basic lemma]\label{thm:linnikbasiclemma}
Let $c>0$ and let $Y\subset [\GL_{2,F}]$ be a homogeneous toral set of maximal type which is invariant under $A_S$. 
Suppose that $B_u = \GL_2(\mathcal{O}_{F,u})$ for any finite place $u$ of $F$ and define $B_\infty$ to be the product over all $B_v$ with $v$ Archimedean.

Then for any $0 < \tau \leq \frac{\log \disc(Y)-\log c\absdisc_F}{2h_{[\GL_{2,F}]}(a)}$,
\begin{align*}
\mu_Y \times \mu_Y\big(\{(x,y) \in &[\GL_{2,F}]^2:y \in x B_\tau\}\big)\\
&\ll_{n,B_\infty,c,\varepsilon} \frac{1}{\vol(Y)} + \frac{\disc(Y)^{1+\varepsilon}}{\vol(Y)^2} 
\,\mathrm{e}^{-2\tau h_{[\GL_{2,F}]}(a)}
\end{align*}
where the implicit constant depends only on $n$, $c$, $\varepsilon$, and polynomially on the diameter of $B_\infty$.
\end{thm}

The diameter of $B_\infty$ is defined in \eqref{eq:definition of r}. 
Roughly speaking, Theorem~\ref{thm:linnikbasiclemma} asserts that at certain time scales (the range growing with the discriminant) the measure $\mu_Y\times \mu_Y$ gives the same amount of mass to pairs of points with displacement in the $\tau$-Bowen ball as the Haar measure does.

We would like to emphasize here that the implicit constant is independent of the base field $F$ once its degree is fixed.
In this sense we consider Theorem~\ref{thm:linnikbasiclemma} to be uniform in the base field $F$.
In applications, the following special case is usually the most relevant.

\begin{example}[A special case]\label{exp:simplepackets}
Suppose that $Y=[\bfT g]$ is a homogeneous toral set where the invariance subgroup $T_\infty = g^{-1}\bfT(\prod_{v\mid \infty}F_v)g$ at the Archimedean places is considered fixed (more precisely, one ought to consider a family of homogeneous toral sets with this invariance).
Let $K/F$ be the quadratic extension associated to $Y$.
Using the definition of discriminant and volume in \S\ref{sec:variousdiscs} resp.~\S\ref{sec:volumes} we have
\begin{align*}
\disc(Y)\asymp_{T_\infty} \absdisc_{K}\absdisc_F^{-2},\quad \vol(Y) = \absdisc_K^{\frac{1}{2}+o(1)}
\end{align*}
so that Theorem~\ref{thm:linnikbasiclemma} takes the form
\begin{align*}
\mu_Y \times \mu_Y\big(\{(x,y) \in [\GL_{2,F}]^2:y \in x B_\tau\}\big)
\ll_{n,B_\infty,c,\varepsilon}\absdisc_K^{-\frac{1}{2}+\varepsilon} + \absdisc_F^{-2}\absdisc_K^{\varepsilon}\mathrm{e}^{-2\tau h_{[\GL_{2,F}]}(a)}
\end{align*}
where the implicit constant depends on $T_\infty$ as well (or more precisely on its Archimedean discriminant).
The factor $\absdisc_F^2$ can be thought of as the volume of the ambient space $[\GL_{2,F}]$ -- see for example \cite[Thm.~29.10.1(c)]{voight}.
\end{example}

From Theorem~\ref{thm:linnikbasiclemma} one can deduce by relatively standard arguments (see for instance \cite[\S4.2]{ELMV12}) that any weak${}^\ast$-probability limit of a sequence of measures $\mu_{Y_i}$ as in the theorem has maximal entropy.
This proves \eqref{eq:equicharspec} in this setup:

\begin{cor}[Linnik's theorem]\label{thm:linnik}
Let $Y_i$ be a sequence of toral packets of maximal type which are invariant under $A_S$. Suppose that the measures $\mu_{Y_i}$ converge in weak${}^\ast$-topology to a probability measure $\mu$ on $[\GL_{2,F}]$.
Then $\mu$ is invariant under $\SL_2(\bbA_F)$.
\end{cor}

\begin{rem}
In some cases (e.g.~in the situation of \Cref{exp:simplepackets}), the invariance statement in \Cref{thm:linnik} can be coupled with an analysis on the character spectrum of $L^2_0([\bfG])$ to yield equidistribution.
See for example the discussions in \cite[\S3.3]{Khayutin-mixing} and \cite[\S9]{ALMW}.
\end{rem}

\begin{rem}[About the proof of Theorem~\ref{thm:linnikbasiclemma}]
The proof of Theorem~\ref{thm:linnikbasiclemma} does exhibit some similarities to other ergodic-theoretic approaches to Linnik's theorem while showing some simplifications.
In \cite{ELMV12,W-linnik} an estimate for representations of binary by ternary quadratic forms (proved in \cite[App.~A]{ELMV12}) was used to establish Linnik's basic lemma.
Here, Linnik's basic lemma is established using geometric invariant theory as well as an elementary estimate for certain adelic orbital integrals. 
The latter substitutes the local analysis on the Bruhat-Tits tree of $\PGL_2(k)$ in \cite[App.~A]{ELMV12} when $k$ is a non-Archimedean local field.
The application of geometric invariant theory involves counting $F$-rational points on the affine line over $F$ satisfying certain local restrictions. The number of such points is related to the $\theta$-invariants of an Hermitian line bundle over $\Spec \OF$ associated to an Arakelov divisor.
This treatment allows us to get an estimate working uniformly for all number fields $F$ of the same degree, which is crucial in our application.
\end{rem}

\begin{rem}[Maximal type]
The maximal type assumption in Theorem~\ref{thm:linnikbasiclemma} should not be considered essential. It is only used in Lemma~\ref{lem:locint bound I}, and we do not assume maximal type elsewhere.  
It is conceivable that the estimate for local orbital integrals in Lemma~\ref{lem:locint bound I} extends to non-maximal type when the counting results in \S\ref{lem:locint bound I} are appropriately refined for exact denominators.
\end{rem}

\begin{thm}[A refinement of Theorem~\ref{thm:main}]\label{thm:refinement of main}
Assume the notations and conditions in Theorem~\ref{thm:main} except for the condition $\absdisc_{K_i} \geq c \absdisc_{F_i}^6$ for some constant $c>0$.
Furthermore, assume that there exists $\alpha >0$ with
\begin{align*}
\absdisc_{F_i} \geq  \absdisc_{K_i}^\alpha
\end{align*}
for every $i$.
Let $\beta \leq 2$ be a non-negative number.
If there is some $\delta>0$ with $ \absdisc_{F_i} \leq \absdisc_{K_i}^{\frac{1}{2\beta}-\delta}$ for all $i$, then there exists a closed subset $A'(\delta,\beta) \subset A'$ with non-empty interior and with 
\begin{align*}
h_\mu(a) \geq (1+2\alpha\beta) \hint(a)
\end{align*}
for all $a \in A'(\delta,\beta)$.
\end{thm}

While Theorem~\ref{thm:refinement of main} does present an improvement upon Theorem~\ref{thm:main}, it is insufficient to prove density of these homogeneous toral sets.

\subsection{Strategy of proof of Theorem~\ref{thm:main}}
To get an entropy bound, we consider the set of pairs of points in a homogeneous toral set $Y = [\bfT g]=[\bfT_i g_i]$ (satisfying the assumptions of the theorem) which differ by an element in the Bowen ball 
\begin{align*}
B_\tau = \bigcap_{-\tau \leq t \leq \tau} a^{-t}B a^t
\end{align*}
for a neighborhood $B\subset \GL_4(\bbA)^1$ of the identity.
It suffices to get an exponential decay rate of the measure of this set with respect to a suitable time parameter $\tau$ (see \Cref{prop:Bowentoentropy} and Theorem~\ref{thm:main}). This strategy has been used in \cite{ELMV09Duke} and \cite{Kha15} among others and originates in the proof of the variational principle.

We first choose $\tau$ large enough to ensure that we "return within the intermediate group". 
More precisely, we choose $\tau$ large enough so that if $t_1,t_2\in\bfT(\bbA)^1$ and $t_1g=\gamma t_2g b$ for some $\gamma\in\GL_4(\Q)$ and $b\in B_\tau$, then we must have $\gamma\in\bfR(\Q)$.
Here, $\bfR\cong\Res_{F/\Q}\GL_2<\GL_4$ defined is as in \eqref{eq:intermediategroupdef}.
This can be achieved by considering the $\bfT\times\bfT$-invariant polynomials on $\GL_4$. They are parametrized by the absolute Weyl group. If $\tau$ is sufficiently large compared to the global discriminant of the homogeneous toral set, then certain invariant polynomials will vanish. Using GIT as in Khayutin's pioneering work \cite{Kha15}, we conclude that $\gamma$ has to lie in $\bfR(\Q)$.

Once the return is in $\bfR(\Q)$, we can change the ambient group from $\GL_4$ to $\bfR$, and study the above mentioned decay in $\GL_{2,F}$. Here is the main difference between our approach and \cite{Kha15}: instead of taking $\tau$ large enough to ensure that the return is in $\bfT(\Q)$, we take a slightly smaller $\tau$ and apply the uniform Linnik's basic lemma (\Cref{thm:linnikbasiclemma}). 
This allows us to get an improved decay rate, and thus an improved entropy bound.

\textbf{Acknowledgments}
This work was initiated when the authors were visiting Hausdorff Research Institute for Mathematics (HIM), Bonn, Germany for the tri\-mes\-ter program ``Dynamics: Topology and Numbers'' in 2020. We would like to thank HIM for hospitality.

The authors would like to thank Ilya Khayutin for his encouragement and various fruitful discussions around the topic. The authors also would like to thank Menny Aka, Manfred Einsiedler, Elon Lindenstrauss and Philippe Michel for helpful conversions.
Lastly, we are very grateful toward the referees for their insightful comments regarding earlier versions of this article.
\section{Discriminant and volume}\label{sec:discvol}

Let $F$ be a number field of degree $n$.

\subsection{Notation for number fields}\label{sec:number fields notations}
For any number field $K$, denote by $\mathcal{O}_K$ the ring of integers of $K$ and by $\absdisc_{K}$ the absolute value of its discriminant.
When $K \supset F$ we write $\absdisc_{K/F}$ for the norm of the relative discriminant for $K/F$.
Note here that the relative discriminant is an ideal in $\mathcal{O}_F$ and that the norm of such an ideal is its index in $\mathcal{O}_F$.
With these notations, the product formula for discriminants states that
\begin{align*}
\absdisc_{K} = \absdisc_{K/F}\absdisc_{F}^{[K:F]},
\end{align*}
see e.g. \cite[\S 14]{Shi10}.

Let $\bbA_F$ denote the ring of adeles of $F$. We simply write $\bbA$ for $\bbA_F$ if the field $F$ is clear from context. Let $\VF$ denote the set of places of $F$, i.e. the equivalence classes of absolute values. 
Write $\mathcal{V}_{F,\infty}$ for the set of infinite places and $\mathcal{V}_{F,\mathrm{f}}$ for the set of finite places.
Let $F_v$ be the completion of $F$ at $v$ for any $v \in \VF$ and identify $v$ with a valuation on $F_v$ when $v$ is finite.
In this case, we denote by $\mathcal{O}_{F,v}$ the discrete valuation ring given by $v$.

For each $v\in\VF$, let $\abs{\cdot}_v$ be the normalized absolute value on $F_v$ so that $m(xA)=\abs{x}_vm(A)$ for every $x\in F_v$, $A\subset F_v$ measurable, and some/any Haar measure $m$ on $F_v$.
For real places $v$ of $F$ this is the standard absolute value while for complex places it is the square of the standard absolute value.
If $v$ is a finite place and $q_v$ is the cardinality of the residue field of $F_v$, then $|x|_v = q_v^{-v(x)}$ for any $x \in F_v$.
Alternatively, if $p$ is the rational prime below $v$ then $|x|_v = |\Nr(x)|_p$ where $\Nr = \Nr_{F_v/\Q_p}: F_v \to \Q_p$ is the norm map.

For any $a\in \bbA_F$, the \emph{content} of $a$ is defined to be 
\begin{align*}
\abs{a}_{\bbA_F}:=\prod_{v\in \VF}\abs{a_v}_v.
\end{align*}
If $S$ is a finite set of places of $F$, we define $F_S=\prod_{v\in S}F_v$, and the $S$-adic content $\abs{a}_S:=\prod_{v\in S}\abs{a_v}_v$ for $a\in F_S$.
The product formula states that $|a|_{\bbA_F} = 1$ for any $a \in F^\times$ (where $F^\times$ is embedded diagonally).

Given number fields $F\subset K$ we write $\Nr_{K/F}$ (resp.~$\Tr_{K/F}$) for the norm (resp.~trace) map $K \to F$, and sometimes also for its extensions e.g.~to the norm (resp.~trace) map $\bbA_K \to \bbA_F$.
For any $x \in \bbA_K$ we have
\begin{align*}
|x|_{\bbA_K} = |\Nr_{K/F}(x)|_{\bbA_F}.
\end{align*}
For any place $u$ of $F$ we define the \'etale algebra $K_u:=K\otimes F_u\simeq \prod_{w\mid u}K_w$. If $x_u\in K_u$ and $w\mid u$, let $x_w\in K_w$ denote the $w$-component of $x_u$.
The norm and the trace extend to maps $K_u \to F_u$.
We also write $\mathcal{O}_{K,u}$ for the maximal order in $K_u$ which under the above identification corresponds to $\prod_{w\mid u}\mathcal{O}_{K,w}$.

\subsection{Homogeneous spaces}\label{sec:homogeneous spaces}
Let $F$ be a number field and consider the homogeneous space
\begin{align*}
\lquot{\GL_n(F)}{\GL_n(\bbA_F)^1}=: [\GL_{n,F}]
\end{align*}
where $\GL_n(\bbA_F)^1 = \{g \in \GL_n(\bbA_F): |\det(g)|_{\bbA_F} = 1\}$.
By a theorem of Borel and Harish-Chandra, $\GL_n(F)$ is a lattice in $\GL_n(\bbA_F)^1$ and we write $\mu_{[\GL_{n,F}]}$ for the $\GL_n(\bbA_F)^1$-invariant probability measure.

More generally, for a reductive $F$-group $\mathbf{G}$, we write
\begin{equation*}
    \mathbf{G}(\bbA)^1 = \big\{ g\in\mathbf{G}(\bbA) : |\chi(g)|_{\bbA_F}=1, \forall\chi\in X^*(\mathbf{G}) \big\},
\end{equation*}
where $X^*(\mathbf{G})$ denotes the group of characters of $\mathbf{G}$.

Denote by $\bfZ < \GL_n$ the center. Letting $[\PGL_{n,F}] = \PGL_n(F)\backslash \PGL_n(\bbA_F)$ we have a fiber bundle
\begin{align*}
\{1\} \to [\bfZ] = \lquot{\bfZ(F)}{\bfZ(\bbA_F)^1} \to [\GL_{n,F}]
\to [\PGL_{n,F}] \to \{1\}.
\end{align*}
so that $[\GL_{n,F}]$ is a compact extension $[\PGL_{n,F}]$ where the fibers are copies of the compact abelian group $[\bfZ]$.

\subsubsection{Homogeneous toral sets}\label{sec:homogeneoustoralsets}
A \emph{homogeneous toral set} in $[\GL_{n,F}]$ is a subset of the form
\begin{align*}
Y := [\bfT g] := \GL_n(F) \bfT(\bbA_F)^1g \subset [\GL_{n,F}],
\end{align*}
where $\bfT$ is a maximal $F$-torus of $\GL_n$, and $g\in\GL_n(\bbA)^1$.
Unless specified otherwise, any homogeneous toral set $Y = [\bfT g]$ in this article satisfies that the group of $F$-characters of $\bfT$ has rank one.
In this case, there is a maximal subfield $K \subset \bfM_n(F)$ so that $\bfT$ is given by the equations $g k = kg$ for all $k\in K$.
Strictly speaking, the field $K$ depends not only on $Y$ but also on $\bfT$ i.e.~it is determined up to conjugation with $F$-rational elements (because $\bfT$ is). We shall put no emphasis on this subtlety here and simply refer to $K$ as the \emph{associated field} (to $Y$ or $\bfT$).
Furthermore, $Y$ has finite volume; we write $\mu_Y$ for the Haar probability measure on $Y$.
Let $\muYtilde$ be the Haar measure on $g^{-1}\bfT(\bbA_F)^1g$ that descends to $\mu_Y$.

\subsubsection{Intermediate homogeneous sets}\label{sec:inthomog}
Given a homogeneous toral set $[\bfT g]$ with associated field $K$ as well as a subfield $K' \subset K$, we write $\bfS_{K'}$ for the subtorus of $\bfT$ isomorphic to $\Res_{K'/\Q}(\mathbb{G}_{m,K'})$. 
The centralizer subgroup $\intgroup_{K'} < \GL_n$ of $K'$ (or equivalently of $\bfS_{K'}$) is isomorphic to $\Res_{K'/\Q}(\GL_{[K:K']})$. The homogeneous space
\begin{align*}
[\intgroup_{K'} g] =  \GL_n(F)\intgroup_{K'}(\bbA_F)^1g
\end{align*}
is of finite volume and contains $[\bfT g]$. We write $\mu_{[\intgroup_{K'} g]}$ for the $g^{-1}\intgroup_{K'}(\bbA_F)^1g$-invariant probability measure on $[\intgroup_{K'} g]$.

\subsection{Discriminants}\label{sec:variousdiscs}
Fix a homogeneous toral set $Y = [\bfT g]$ with associated field $K$.
Any such homogeneous toral set comes with a notion of \emph{discriminant} $\disc(Y)$ and volume $\vol(Y)$ which we recall in this section.

\subsubsection{Non-Archimedean discriminants}\label{sec:nonarchdisc}
Define for a finite place $u$ the \emph{associated local order} at $u$
\begin{align}\label{eq:deflocalorder}
\mathcal{O}_u = g_u\bfM_n(\mathcal{O}_{F_u}) g_u^{-1} \cap K_u.
\end{align}
This is an order in $K_u$ with the property that $\mathcal{O}_u = \mathcal{O}_{K_u}$ for all but finitely many places $u$.
Indeed, $g_u \in \GL_n(\mathcal{O}_{F_u})$ for almost every $u$.
Note also that $\mathcal{O}_u \cap F_u = \mathcal{O}_{F,u}$.
We will say that $Y$ is of \emph{maximal type} if every associated local order $\mathcal{O}_{u}$ is the maximal order in $K_u$.

The local order $\mathcal{O}_u$ has a local discriminant $\disc(\mathcal{O}_u)$ which is obtained by taking an $\mathcal{O}_{F,u}$-basis of $\mathcal{O}_u$ and taking the discriminant of the bilinear form given by the reduced trace of $K_u/F_u$. As such, the discriminant is a well-defined element of $\rquot{\mathcal{O}_{F,u}}{(\mathcal{O}_{F,u}^\times)^2}$ where $(\mathcal{O}_{F,u}^\times)^2$ denotes the subgroup of squares in $\mathcal{O}_{F,u}^\times$.
If $\mathcal{O}_u$ is maximal, the ideal generated by $\disc(\mathcal{O}_u)$ is exactly the relative discriminant of the local extension $K_u/F_u$.
The \emph{local absolute discriminant} of $Y$ at $u$ is given by
\begin{align*}
\disc_u(Y):= |\disc(\mathcal{O}_u)|_u^{-1}.
\end{align*}

By the local-to-global principle for orders,
\begin{align*}
\mathcal{O} = \bigcap_{u\text{ finite}} (\mathcal{O}_u \cap K)
\end{align*}
is an order in $K$. 
As in the local case, we have a discriminant $\disc(\mathcal{O}) \in \rquot{\mathcal{O}_{F}}{(\mathcal{O}_{F}^\times)^2}$ and the \emph{absolute non-Archimedean discriminant} of $Y$ is
\begin{align*}
\disc_{\fin}(Y) = |\Nr_{F/\Q}(\disc(\mathcal{O}))| = \prod_{u\in \mathcal{V}_{F,\mathrm{f}}} \disc_u(Y),
\end{align*}
If $Y$ is of maximal type, then $\mathcal{O} = \mathcal{O}_K$ and $\disc_{\fin}(Y)=\absdisc_{K/F}$.

\subsubsection{Archimedean discriminants}\label{sec:archdisc}

We give an adhoc definition of the Archimedean discriminant of the homogeneous toral set $Y = \GL_n(F)\bfT(\bbA_F)g$ which is most useful in our context.
See \cite[\S4,\S6]{ELMV11Annals} and \cite[\S7.4.2]{Kha15} for a more thorough discussion.

Let $u$ be an Archimedean place of $F$ so that $F_u \simeq \R$ or $F_u\simeq \C$. In either case, the algebraic closure $\overline{F_u}$ is $\C$; we identify $\overline{F_u} = \C$.
Define a norm $\| \cdot \|$ on $\bfM_n(\C)$ through the Hermitian inner product
\begin{align*}
\langle A,B \rangle = \Tr(AB^*).
\end{align*}
While this is clearly a choice of a good norm in the sense of \cite[\S7]{ELMV11Annals}, it also constitutes a consistent choice for varying $n$ where $n$ is the degree of $F$.
Under the identification $\C = \overline{F_u}$, $\| \cdot \|$ (resp.~$\langle
\cdot,\cdot\rangle$) restricts to a norm on $\bfM_n(F_u)$ which we also denote by $\| \cdot \|$ (resp.~$\langle\cdot,\cdot\rangle$).

Let $\bfK'_u \subset \bfM_n$ be the centralizer of $g_u^{-1}\bfT(F_u)g_u$ and let $k_1,\ldots,k_n$ be a $\C$-basis of $\bfK'_u(\C)$.
The Archimedean discriminant of $Y$ at $u$ is then given by
\begin{align*}
\disc_u(Y) = \frac{\det(\langle k_i,k_j\rangle)_{ij}}{|\det(\Tr(k_ik_j))_{ij}|}
\end{align*}
and is independent of the choice of basis.

\begin{example}\label{exp:archdiscdim2}
Suppose that $n=2$. Let $k \in \bfK'_u(F_u)$ be non-zero with $\Tr(k) = 0$.
The choice of basis $k_1 = 1, k_2 = k$ yields
\begin{align*}
\disc_u(Y) = \frac{\norm{k}^2}{4|\det(k)|} \asymp \big|\det\big(\tfrac{k}{\norm{k}}\big)\big|^{-1}.
\end{align*}
Roughly speaking, the discriminant $\disc_u(Y)$ is large whenever the normalized $k$ is close to being nilpotent.
\end{example}

Now we can define the \emph{global absolute discriminant} of $Y$ to be
\begin{equation}\label{eq:definition of disc}
	\disc(Y)=\prod_{u\in\VF}\disc_u(Y).
\end{equation}


\subsubsection{Orders in quadratic extensions}\label{sec:quadorders}
In view of Theorem~\ref{thm:linnikbasiclemma} we record a special property of local quadratic extensions. To that end, let $K/F$ be a quadratic extension, and let $\mathcal{O} \subset K$ be an $\mathcal{O}_F$-order (i.e.~$\mathcal{O}\cap F = \mathcal{O}_F$).
For a finite place $u$ of $F$ we write (as before) $\mathcal{O}_{u} = \mathcal{O} \otimes_{\mathcal{O}_F} \mathcal{O}_{F,u}$ for the completion at $u$.

Recall that a fractional ideal $\mathfrak{a} \subset K$ is said to be $\mathcal{O}$-proper if
\begin{align*}
\mathcal{O} = \{\lambda \in K: \lambda \mathfrak{a} \subset \mathfrak{a}\}.
\end{align*}
Proper fractional ideals for $\mathcal{O}_{u}$ are similarly defined.

\begin{prop}[Local orders in quadratic extensions]\label{prop:locordersmono}
For $K$ and $\mathcal{O}$ as above and any finite place $u$ of $F$ the local order $\mathcal{O}_{u}$ is a monogenic i.e.~there exists $\alpha \in \mathcal{O}_{u}$ with $\mathcal{O}_{u} = \mathcal{O}_{F,u}[\alpha]$.
Moreover, any proper $\mathcal{O}_{u}$-ideal is principal (and vice versa).
\end{prop}

\begin{proof}
To see that $\mathcal{O}_{u}$ is monogenic observe first that $\mathcal{O}_{F,u}$ is PID (it is a discrete valuation ring).
The order $\mathcal{O}_{u}$ is a submodule of the free $\mathcal{O}_{F,u}$-module $\mathcal{O}_{K,u}$ and in particular also free. Furthermore, $1 \in\mathcal{O}_{u}$ is primitive (i.e.~$\frac{1}{\varpi} \not \in \mathcal{O}_{u}$ for a uniformizer $\varpi$ of $\mathcal{O}_{F,u}$) and there exists $\alpha \in \mathcal{O}_{u}$ so that $1,\alpha$ is a  $\mathcal{O}_{F,u}$-basis of $\mathcal{O}_{u}$. In particular, $\mathcal{O}_{u}$ is monogenic.
The second statement can be proven as in \cite[Prop.~2.1]{ELMV12} using monogenicity.
\end{proof}

Recall that the \emph{inverse different} for the order $\mathcal{O}_u$ is
\begin{align*}
\{x \in K_u: \Tr(xy) \in \mathcal{O}_{F,u} \text{ for all } y \in \mathcal{O}_u\}.
\end{align*}

\begin{lem}[Different ideals]\label{lem:different}
For any $\mathcal{O}_F$-order $\mathcal{O} \subset K$ and any finite place $u$ of $F$, the inverse different for $\mathcal{O}_u$ is a proper and principal ideal.
If $\alpha \in \mathcal{O}_u$ is such that $\mathcal{O}_{u} = \mathcal{O}_{F,u}[\alpha]$, the different ideal is generated by $\alpha-\sigma(\alpha)$ where $\sigma \in \Gal(K_u/F_u)$ is the non-trivial Galois element.
\end{lem}

Note that $\mathcal{O}_u$ is Galois-invariant as $\sigma(\alpha) = \Tr(\alpha)-\alpha \in \mathcal{O}_u$.

\begin{proof}
We omit the explicit calculation. One shows that $(\alpha-\sigma(\alpha))^{-1}$ generates the inverse different.
\end{proof}

For an order $\mathcal{O}$ as above we write $\Delta_{\mathcal{O},u}$ for a choice of generator for the different ideal of $\mathcal{O}_{u}$ (unique up to $\mathcal{O}_{u}^\times$-multiples).
In particular,
\begin{align*}
\disc(\mathcal{O}_u) = \Nr(\Delta_{\mathcal{O},u}).
\end{align*}
When $\mathcal{O} = \mathcal{O}_K$, we write $\Delta_{K/F,u} = \Delta_{\mathcal{O}_K,u}$.

\begin{rem}[Parametrizing suborders]
If $\alpha \in \mathcal{O}_{K,u}$ is a generator for the maximal order, one can show that any order $\mathcal{O}_u \subset \mathcal{O}_{K,u}$ is of the form $\mathcal{O}_u =\mathcal{O}_{F,u}[f \alpha]$ for some $f \in \mathcal{O}_{F,u}$. The element $f \in \mathcal{O}_{F,u}$ is uniquely determined up to $\mathcal{O}_{F,u}^\times$-multiples; its class in $\mathcal{O}_{F,u}/\mathcal{O}_{F,u}^\times$ can be called the \emph{conductor} of the order $\mathcal{O}_u$.
\end{rem}

\begin{lem}[Image of the norm]\label{lem:normimage}
When $u$ is non-dyadic (i.e.~$u \nmid 2$), the norm map $\mathcal{O}_u^\times \to \mathcal{O}_{F,u}^\times$ is surjective if $\disc(\mathcal{O}_u) \in \mathcal{O}_{F,u}^\times$ and otherwise the image has index two.
If $u$ is dyadic, the index is bounded by a constant depending on the degree $n=[F:\Q]$ only.
\end{lem}

We remark that the bound on the index for dyadic places can be drastically optimized (e.g.~if $u$ splits in $K$). For our application, the above lemma is sufficient.

\begin{proof}
Let $\alpha \in \mathcal{O}_u$ be as in \Cref{prop:locordersmono}. Any element in the image of the norm map $\mathcal{O}_u^\times \to \mathcal{O}_{F,u}$ takes the form
\begin{align*}
Q(x,y) = \Nr(x+y\alpha) = x^2 + \Tr(\alpha) xy + \Nr(\alpha) y^2
\end{align*}
where $x,y \in \mathcal{O}_{F,u}$.  Assume first that $u$ is not dyadic. In this case, we can choose $\alpha$ to be traceless. Whenever $\Nr(\alpha)$ is a unit, the binary form $Q$ is universal (i.e. represents all units) by an application of the pigeonhole principle and Hensel's lemma.
As $\alpha$ is equal to $\Delta_{\mathcal{O},u}$ up to a unit, this shows the claim. If $\Nr(\alpha)$ is not a unit, $Q$ only represents the squares when taken modulo the prime underlying $u$ and in particular, $\Nr$ is not surjective. On the other hand, the squares are represented by $Q$ and have index $2$ in $\mathcal{O}_{F,u}^\times$ so that this finishes the non-dyadic case.

Suppose now that $2 \mid u$. It suffices to estimate the index of the subgroup of squares in $\mathcal{O}_{F,u}^\times$. However, observe that $[F_u:\Q_2] \mid [F:\Q]=n$ and that for each $m\mid n$ there are finitely many degree $m$ extensions over $\Q_2$. This implies the lemma.
\end{proof}

\subsection{Volume of homogeneous sets}\label{sec:volumes}

For any homogeneous toral set $Y = [\bfT g] \subset [\GL_{n,F}]$ and a bounded neighborhood of the identity $B \subset \GL_n(\bbA_F)$ we define the \emph{volume}
\begin{align*}
\vol_B(Y) = \muYtilde(B)^{-1}.
\end{align*}

Note that for any other bounded neighborhood of the identity $B' \subset \GL_n(\bbA_F)$ we have
\begin{align}\label{eq:different volumes}
\vol_B(Y) \ll_{B,B'} \vol_{B'}(Y) \ll_{B,B'} \vol_B(Y).
\end{align}
Disregarding the dependency on the base field, it was shown in \cite[Thm.~4.8]{ELMV11Annals} that
\begin{align*}
\vol_B(Y) = \disc(Y)^{\frac{1}{2}+o_F(1)}
\end{align*}
when $B$ is chosen appropriately. The method is applicable to yield a sharper relation. For instance, whenever $Y$ is a homogeneous toral set of maximal type with fixed invariance $A_\infty$ at the Archimedean place, then
\begin{align*}
\vol_B(Y) = \disc(Y)^{\frac{1}{2}+o(1)} \absdisc_F^{[K:F]/2 +o(1)}
\end{align*}
where the implicit constants depend on $A_\infty$.
In our case, the relative discriminant $\absdisc_{K/F} \asymp \disc(Y)$ is usually bigger than $\absdisc_F$ (see Lemma~\ref{lem:counting lemma} and its application). 
In particular, one can omit the $o(1)$ exponent in $\absdisc_F$.

In this article, the neighborhood $B$ will be typically fixed in which case we take the liberty of dropping the subscript $B$ in $\vol_B(Y)$.
This is additionally justified by~\eqref{eq:different volumes}.
Furthermore, the volume can be defined for any homogeneous set and in particular for the homogeneous sets introduced in \S\ref{sec:homogeneous spaces}.
Observe that
\begin{align}\label{eq:volume of Gm over F}
\vol([\bfZ g]) = \absdisc_F^{\frac{1}{2}+o(1)}
\end{align}
for any $g \in \GL_2(\bbA_F)^1$ (the implicit constants are independent of $F$).

\subsection{Discriminants of cyclic fields}
This short subsection addresses the non-necessity of the discriminant condition in Theorem~\ref{thm:main} in the cyclic case.
Let $K$ be a quartic number field, and $L$ be its normal closure. 
Recall from Remark~\ref{rem:galoistypes} that the Galois group $\calG:=\Gal(L/\Q)$ is one of the following five groups: $S_4$, $A_4$, $D_4$, $C_4=\rquot{\Z}{4\Z}$, $V_4=\rquot{\Z}{2\Z}\times\rquot{\Z}{2\Z}$.

As explained in the introduction, we are particularly interested in the cases where the Galois group is not 2-transitive, i.e. $\calG$ equals $D_4$ or $C_4$ or $V_4$. In all of these three cases, we can find a quadratic subfield $F$ of $K$. Furthermore, if $\calG$ is $D_4$ or $C_4$, then $F$ is unique.

\begin{lem}\label{lem:disc C_4}
Let $K$ be a cyclic quartic field and let $F$ be the unique quadratic subfield. Then
	\begin{equation*}
	\absdisc_{K/F}\geq \frac{1}{4}\absdisc_{F}.
	\end{equation*}
\end{lem}

Lemma~\ref{lem:disc C_4} implies that the discriminant condition in Theorem~\ref{thm:main} is automatic for cyclic fields.

\begin{proof}
	Suppose $\calG=C_4$. By \cite[Prop. 2]{EP80}, there exist integers $W>0,d>1$ with $d$ squarefree, such that $\absdisc_{K}=W^2d^3$, $\absdisc_{K/F}=W^2d$ or $(W/4)^2d$, and $\absdisc_{F}=d$ or $4d$, depending on whether $d\equiv 1(4)$ or not. The conclusion can be easily checked.
\end{proof}

We note that for $\calG=D_4$ or $V_4$, the discriminant $\absdisc_{K/F}$ could be arbitrarily large or small compared to $\absdisc_{F}$. See \cite[Lemmas 17, 20]{Bai80}  for the $\calG=D_4$ case.

\section{A counting lemma}\label{sec:countinglemma}

In this section, we introduce Hermitian vector bundles over arithmetic curves, their $\theta$-invariants, and the Poisson-Riemann-Roch formula. Our main reference is the first three sections of \cite{Bos20}. We then use this to prove a counting estimate of rational points with certain local bounds in a number field.



\subsection{Hermitian vector bundles over arithmetic curves}
Let $F$ be a number field and let $\mathcal{O}_F \subset F$ be the ring of integers. A \emph{Hermitian vector bundle} over $\Spec \OF$ is a pair
\begin{equation*}
	\Ebar = (E, (\norm{\cdot}_\sigma)_{\sigma:F\hookrightarrow\C})
\end{equation*}
consisting in a finitely generated projective $\OF$-module $E$ and in a family $(\norm{\cdot}_\sigma)_{\sigma:F\hookrightarrow\C}$ of Hermitian norms $\norm{\cdot}_\sigma$ on the complex vector spaces
\begin{equation*}
	E_\sigma=E\otimes_{\OF,\sigma}\C.
\end{equation*}
The family $(\norm{\cdot}_\sigma)_{\sigma:F\hookrightarrow\C}$ is moreover required to be invariant under complex conjugation in that sense that
$\norm{x}_{\sigma} = \norm{\bar{x}}_{\bar{\sigma}}$ for $x \in E_\sigma$.

A \emph{Euclidean lattice} $\Ebar=(E,\norm{\cdot})$ is a free $\Z$-module of finite rank together with a Euclidean norm $\norm{\cdot}$ on the real vector space $E\otimes_{\Z}\R$.
We note that by definition, a Hermitian vector bundle over $\Spec\Z$ is nothing but a Euclidean lattice. Indeed, to give a Hermitian norm on $E_\C=E\otimes_\Z\C$ invariant under complex conjugation is the same as to give a Euclidean norm on $E_\R$.

The \emph{rank} of some Hermitian vector bundle $\Ebar$ over $\Spec \OF$, denoted by $\rk E$, is the rank of the $\OF$-module $E$, or equivalently the dimension of the complex vector spaces $E_\sigma$. A \emph{Hermitian line bundle} is a Hermitian vector bundle of rank 1.

Give any Hermitian vector bundle $\Ebar = (E, (\norm{\cdot}_\sigma)_{\sigma:F\hookrightarrow\C})$, we shall define its dual $\Ebar^\vee$. Since $E$ is finitely generated and projective over $\OF$,  $E^\vee=\Hom(E, \OF)$ is also a finitely generated projective $\OF$-module. 
Moreover, for any $\sigma:F\hookrightarrow\C$, the Hermitian norm $\norm{\cdot}_\sigma$ induces a canonical isomorphism between $E_\sigma$ and $E^\vee_\sigma$, and thus $\norm{\cdot}_\sigma$ can be viewed as a Hermitian norm on $E^\vee_\sigma$ as well. We define $\Ebar^\vee:=(E^\vee,(\norm{\cdot}_\sigma)_{\sigma:F\hookrightarrow\C})$. Similarly we can define the exterior powers $\bigwedge^k\Ebar$ of a Hermitian vector bundle $\Ebar$, as well as the direct sum $\Ebar_1\oplus\Ebar_2$ and the tensor product $\Ebar_1\otimes\Ebar_2$ of two Hermitian vector bundles $\Ebar_1$ and $\Ebar_2$.

Let $\pi$ denote the unique morphism of schemes from $\operatorname{Spec}\mathcal O_F$ to $\operatorname{Spec}\mathbb Z$. Given any Hermitian vector bundle $\Ebar = (E, (\norm{\cdot}_\sigma)_{\sigma:F\hookrightarrow\C})$ over $\Spec \OF$, we shall define its \emph{direct image} $\pi_*\Ebar$ over $\Spec \Z$. Let $\pi_*E$ be the underlying $\Z$-module of $E$, and we observe that
\begin{equation*}
\begin{split}
	(\pi_*E)_\C :=&\pi_*E\otimes_\Z\C \\
	=&\bigoplus_{\sigma:F\hookrightarrow\C}(E\otimes_{\OF,\sigma}\C).
\end{split}
\end{equation*}
Now for any $v=(v_\sigma)_{\sigma:F\hookrightarrow\C}$ in $(\pi_*E)_\C$, let
\begin{equation*}
	\norm{v}^2_{\pi_*\Ebar}:=\sum_{\sigma:F\hookrightarrow\C}\norm{v_\sigma}^2_\sigma.
\end{equation*}
We define $\pi_*\Ebar:=(\pi_* E,\norm{\cdot}_{\pi_*E})$.

Let $\omega_{\OF/\Z}:=\Hom_\Z(\OF, \Z)$ which can be seen as the inverse different. The formula $af(b)=f(ab)$ defines an $\OF$-module structure on $\omega_{\OF/\Z}$. 
The \emph{canonical Hermitian line bundle} $\overline{\omega}_{\OF/\Z}$ over $\Spec\OF$ is the pair $(\omega_{\OF/\Z},(\norm{\cdot}_\sigma)_{\sigma:F\hookrightarrow\C})$, where the Hermitian norms are given by
\begin{equation*}
	\norm{\tr_{F/\Q}}_\sigma=1,\;\forall \sigma:F\hookrightarrow\C,
\end{equation*}
where $\tr_{F/\Q}$ is the trace map from $F$ to $\Q$, which is indeed a non-zero element in $\Hom_\Z(\OF,\Z)$. Since $\omega_{\OF/\Z}$ has rank 1 as an $\OF$-module, the above equation uniquely determines the Hermitian norms. Note that for every Hermitian vector bundle $\Ebar$ over $\Spec \OF$, we have a canonical isometric isomorphism of Hermitian vector bundles over $\Spec\Z$ (see e.g. \cite[Proposition 3.2.2]{BK10}):

\begin{equation}\label{eq:comparison_direct_image_dual}
	\pi_*(\Ebar^\vee\otimes\overline{\omega}_{\OF/\Z})\stackrel{\sim}{\longrightarrow}\pi_*(\Ebar)^\vee.
\end{equation}

\subsection{Arakelov degree}
The \emph{Arakelov degree} of a Hermitian line bundle $\Lbar$ is defined to be
\begin{equation}\label{eq:definition_Arakelov_degree}
\begin{split}
	\Adeg\Lbar&=\log\abs{L/\OF s} - \sum_{\sigma:F\hookrightarrow\C}\log\norm{s}_\sigma\\
	&=\sum_{0\neq\frakp\in\Spec\OF}v_\frakp(s)\log \Nr\frakp - \sum_{\sigma:F\hookrightarrow\C}\log\norm{s}_\sigma,
\end{split}
\end{equation}
where $s\in L\setminus 0$, $v_\frakp$ is the valuation associated to $\frakp$\footnote{$\mathcal{O}_{F,\frakp}$ is a local ring, and let $\mathfrak{m}_\frakp$ denote its unique maximal ideal. Then $s$ generates $\mathfrak{m}_\frakp^{v_\frakp(s)}L$ as an $\mathcal{O}_{F,\frakp}$-submodule of $L$.}, and $\Nr\frakp=\abs{\OF/\frakp}$ is the norm of $\frakp$. 
By product formula this is well-defined, i.e. independent of choice of $s$. We extend the definition to any Hermitian vector bundle $\Ebar$ by setting
\begin{equation*}
	\Adeg\Ebar:=\Adeg\wedge^{\rk E}\Ebar.
\end{equation*}

From this definition, when $\Ebar$ is an Euclidean lattice, we have
\begin{equation*}
	\Adeg(\Ebar)=-\log\covol(\Ebar).
\end{equation*}

For any two Hermitian line bundles $\Lbar_1$ and $\Lbar_2$, it follows from \eqref{eq:definition_Arakelov_degree} that
\begin{equation}\label{eq:degree_tensor_product_line_bundle}
	\Adeg \Lbar_1\otimes\Lbar_2 = \Adeg \Lbar_1 + \Adeg \Lbar_2.
\end{equation}

By reducing to the line bundle case, it can be shown that for any Hermitian vector bundle $\Ebar$,
\begin{equation}\label{eq:degree_of_dual}
	\Adeg \Ebar^\vee = -\Adeg \Ebar.
\end{equation}

We have the following formula relating the Arakelov degrees of a Hermitian vector bundle and its direct image (see e.g. \cite[(1.3.6)]{Bos20} or \cite[(2.1.13)]{BGS94}):
\begin{equation}\label{eq:degree_of_direct_image}
	\Adeg\pi_*\Ebar = \Adeg\Ebar-\frac{1}{2}\log\,\absdisc_F\cdot\rk E.\,
\end{equation}
where $\absdisc_F$ is the absolute value of the discriminant of the number field $F$.

Combining \eqref{eq:comparison_direct_image_dual} and \eqref{eq:degree_of_direct_image}, one can show that
\begin{equation}\label{eq:degree_of_canonical}
	\Adeg \overline\omega_{\OF/\Z}=\log\,\absdisc_F.
\end{equation}

\subsection{$\theta$-invariants and Poisson-Riemann-Roch formula}

Given an Euclidean lattice $\Ebar=(E,\norm{\cdot})$, we can define \emph{$\theta$-invariants} associated to it. The first one is analogous to the dimension of global sections of a vector bundle over a smooth projective curve over a base field $k$ and is defined to be
\begin{equation*}
	\htheta^0(\Ebar) := \log \sum_{v\in E}e^{-\pi\norm{v}^2}.
\end{equation*}
We also define
\begin{equation}\label{eq:definition_h^1_Euclidean_lattice}
\htheta^1(\Ebar) := \htheta^0(\Ebar^\vee).
\end{equation}

By the Poisson summation formula, for any Euclidean lattice $\Ebar$ one has
\begin{equation*}
	\sum_{w\in E^\vee}e^{-\pi\norm{w}^2_{\Ebar^\vee}}=\covol(\Ebar)\sum_{v\in E}e^{-\pi\norm{v}^2_{\Ebar}}.
\end{equation*}
Using the above notations, this formula may be rewritten as
\begin{equation}\label{eq:Poisson-Riemann-Roch_for_Euclidean_lattices}
	\htheta^0(\Ebar)-\htheta^1(\Ebar)=\Adeg(\Ebar).
\end{equation}


\bigskip

More generally, for any Hermitian vector bundle $\Ebar$ over $\Spec\OF$, we define
\begin{equation}\label{eq:definition_theta_invariants_Hermitian_vector_bundle}
	\htheta^i(\Ebar) := \htheta^i(\pi_*\Ebar),\;i=0,1.
\end{equation}
With this definition, \eqref{eq:comparison_direct_image_dual} gives the Hecke-Serre duality formula (which is analogous to the classical Serre duality formula):
\begin{equation}\label{eq:Hecke-Serre-duality}
	\htheta^1(\Ebar) = \htheta^0(\Ebar^\vee\otimes\overline\omega_{\OF/\Z}).
\end{equation}
Combining \eqref{eq:degree_of_direct_image}, \eqref{eq:Poisson-Riemann-Roch_for_Euclidean_lattices}, and \eqref{eq:definition_theta_invariants_Hermitian_vector_bundle}, we obtain the general version of Poisson-Riemann-Roch formula for any Hermitian vector bundle $\Ebar$ over $\OF$:
\begin{equation}\label{eq:Poisson-Riemann-Roch_for_Hermitian_vector_bundles}
	\htheta^0(\Ebar)-\htheta^1(\Ebar)=\Adeg\Ebar-\frac{1}{2}(\log\,\absdisc_F)\rk E.
\end{equation}

\subsection{Upper bounds of $\theta$-invariants}
We first recall Proposition 2.7.3 from \cite{Bos20}:

\begin{prop}\label{prop:upper_bound_h^0_negative_degree_line_bundle}
	Let $t\in\R$. For every Hermitian line bundle $\Lbar$ over $\Spec \OF$ such that $\Adeg\Lbar\leq t$, we have
	\begin{equation*}
		\htheta^0(\Lbar)\leq f(t),
	\end{equation*}
	where
	\begin{equation}\label{eq:f(t)}
		f(t)=\begin{cases}
		1+t & t\geq0\\e^{2\pi t} & t<0
		\end{cases}
	\end{equation}
\end{prop}

As a consequence, we obtain the following upper bound of $\htheta^1(\Lbar)$, for $\Lbar$ a Hermitian line bundle of sufficiently large degree.

\begin{cor}\label{cor:upper_bound_h^1_positive_degree_line_bundle}
	Let $t\in\R$. For every Hermitian line bundle $\Lbar$ over $\Spec \OF$ such that $\Adeg\Lbar\geq \log\,\absdisc_F+t$, we have
	\begin{equation*}
		\htheta^1(\Lbar)\leq f(-t).
	\end{equation*}
\end{cor}

\begin{proof}
	By \eqref{eq:Hecke-Serre-duality}, we have $\htheta^1(\Lbar) = \htheta^0(\Lbar^\vee\otimes\overline\omega_{\OF/\Z})$. By \eqref{eq:degree_tensor_product_line_bundle}\eqref{eq:degree_of_dual}\eqref{eq:degree_of_canonical} we have
	\[
	\Adeg(\Lbar^\vee\otimes\overline\omega_{\OF/\Z})=\log\,\absdisc_F-\Adeg\Lbar\leq -t.
	\]
	Hence we finish the proof by applying \Cref{prop:upper_bound_h^0_negative_degree_line_bundle} to $\Lbar^\vee\otimes\overline\omega_{\OF/\Z}$.
\end{proof}

We now state the main result of this subsection.
\begin{prop}\label{prop:upper_bound_h^0_large_degree_line_bundle}
	Let $t\in\R$. For every Hermitian line bundle $\Lbar$ over $\Spec \OF$ such that $\Adeg\Lbar\geq \log\,\absdisc_F+t$, we have
	\begin{equation*}
	\htheta^0(\Lbar)\leq \Adeg\Lbar-\frac{1}{2}\log\,\absdisc_F+ f(-t).
	\end{equation*}
\end{prop}

\begin{proof}
	This follows immediately from \eqref{eq:Poisson-Riemann-Roch_for_Hermitian_vector_bundles} and \Cref{cor:upper_bound_h^1_positive_degree_line_bundle}.
\end{proof}

\subsection{Comparing $\hAr^0$ and $\htheta^0$}
Let $\Ebar=(E,\norm{\cdot}_\sigma)$ be a Hermitian vector bundle over $\Spec\OF$. Following \cite[Section 2.4]{GS91}, we define the invariant
\begin{equation*}
	\hAr^0(\Ebar):=\log\#\{ s\in E : \norm{s}_\sigma\leq 1,\,\forall\sigma \}.
\end{equation*}

\begin{prop}\label{prop:comparison_hAr_htheta}
	Let $n=[F\colon\Q]$. For any Hermitian vector bundle $\Ebar$ over $\Spec\OF$, we have
	\begin{equation*}
		\hAr^0(\Ebar) \leq \htheta^0(\Ebar) + \pi n.
	\end{equation*}
\end{prop}

\begin{proof}
	Let $\norm{\cdot}:=\sqrt{\sum_\sigma\norm{\cdot}^2_\sigma}$ be the Euclidean norm in $\pi_\ast(\bar{E})$.
	By definition,
	\begin{equation}\label{eq:temp_comparison1}
		\hAr^0(\Ebar) \leq \log\#\left\{ s\in E : \sum_\sigma\norm{s}^2_\sigma\leq n \right\}=\log\#\{ s\in E : \norm{s}\leq \sqrt{n} \}.
	\end{equation}
	On the other hand, by \eqref{eq:definition_theta_invariants_Hermitian_vector_bundle}
	\begin{equation}\label{eq:temp_comparison2}
	\begin{split}
		\htheta^0(\Ebar) &\geq \log\sum_{\norm{s}\leq \sqrt{n}}e^{-\pi\norm{s}^2} \\
		&\geq \log\sum_{\norm{s}\leq \sqrt{n}}e^{-\pi n} \\
		&= \log\#\{ s\in E : \norm{s}\leq \sqrt{n} \} - \pi n.
	\end{split}
	\end{equation}
	The proposition follows from \eqref{eq:temp_comparison1} and \eqref{eq:temp_comparison2}.
\end{proof}

\begin{cor}\label{cor:upper_bound_hAr}
	Let $t\in\R$. Suppose $\Lbar$ is a Hermitian line bundle over $\Spec\OF$ such that $\Adeg\Lbar\geq \log\,\absdisc_F+t$. Then
	\begin{equation*}
	\hAr^0(\Lbar)\leq \Adeg\Lbar-\frac{1}{2}\log\,\absdisc_F+ f(-t)+\pi n.
	\end{equation*}
\end{cor}

\begin{proof}
	This follows from \Cref{prop:upper_bound_h^0_large_degree_line_bundle} and \Cref{prop:comparison_hAr_htheta}.
\end{proof}

\subsection{A counting estimate}
Let $F$ be a number field of degree $n$.

\begin{lem}\label{lem:counting lemma}
	 Let $c>0$. Let $\bfr=(r_u)_{u\in\VF}$ be a collection of positive real numbers satisfying the following:
	\begin{enumerate}
		\item $r_u=1$ for all but finitely many $u$'s. 
		\item $r_u$ is in the image of the absolute value\footnote{We take the normalized absolute value as defined in the beginning of \Cref{sec:number fields notations}.} map $\abs{\cdot}_u\colon F_u\to\R_{\geq0}$ for all $u$.
	\end{enumerate} 
	Let $\norm{\bfr}=\prod_u r_u$, and suppose that $\norm{\bfr}\geq c\absdisc_F$.
	Then there exists a constant $C=C(n,c)>0$ such that
	\begin{equation}\label{eq:non-effective_bound}
		\#\{x\in F : \abs{x}_u\leq r_u \text{ for all } u \}\leq \frac{C}{\absdisc_F^{1/2}}\norm{\bfr}.
	\end{equation}
\end{lem}

\begin{proof}
	We first associate a Hermitian line bundle $\Lbar(\bfr)$ to $\bfr$. Let $L$ be the fractional ideal of $F$ such that for every non-Archimedean $u$, $L_u:=L\otimes_{\OF} \mathcal{O}_{F_u}$ is generated by some $a_u\in F_u$ with $\abs{a_u}_u = r_u$. 
	We put Hermitian norms $\norm{\cdot}_\sigma$ on $L\otimes_{\OF,\sigma}\C \simeq \OF\otimes_{\OF,\sigma}\C$ such that $\norm{1_F\otimes1}_\sigma=r_\sigma^{-1}$, and define $\Lbar(\bfr):=(L,\norm{\cdot}_\sigma)$. 
	It follows from the definitions that 
	\begin{equation*}
		\hAr^0\left(\Lbar(\bfr)\right) = \log \#\{x\in F : \abs{x}_u\leq r_u \text{ for all } u\in \VF \}.
	\end{equation*}

	Take any $s\in L$. For any non-Archimedean $u\in\VF$, let $\frakp_u$ be the associated prime ideal in $\OF$. By definition we have $\mathcal{O}_{\frakp_u}s=\mathfrak{m}_{\frakp_u}^{v_{\frakp_u}(s)}L_{\frakp_u}=\mathfrak{m}_{\frakp_u}^{v_{\frakp_u}(s)} a_u$, where $\mathfrak{m}_{\frakp_u}$ denotes the unique maximal ideal in $\mathcal{O}_{\frakp_u}$. Taking $\abs{\cdot}_u$ on both sides, we get $v_{\frakp_u}(s)=\log_{\Nr{\frakp_u}}\abs{s^{-1}a_u}_u$. 
	By \eqref{eq:definition_Arakelov_degree} and the product formula,
	\begin{equation*}
	\begin{split}
		&\Adeg\left(\Lbar(\bfr)\right) 
		= \sum_{0\neq\frakp\in\Spec\OF}v_\frakp(s)\log \Nr\frakp - \sum_{\sigma:F\hookrightarrow\C}\log\norm{s}_\sigma\\
		&= \sum_{u<\infty}\log\,\abs{s^{-1}a_u}_u-\sum_{u\mid\infty}\log(r_u^{-1}\abs{s}_u)
		= \sum_{u\in\VF}-\log (r_u^{-1}\abs{s}_u)\\
		&= \sum_{u\in\VF}\log r_u-\sum_{u\in\VF}\log\,\abs{s}_u
		= \sum_{u\in\VF}\log r_u
		= \log\,\norm{\bfr}.
	\end{split}
	\end{equation*}
	Let $C=\exp(f(-\log c)+\pi n)$, where $f$ is defined in \eqref{eq:f(t)}. We apply \Cref{cor:upper_bound_hAr} to conclude the proof.
\end{proof}

\section{A Linnik-type result and uniformity over the base field}\label{sec:linnik}

The aim of this section is to prove Linnik's basic lemma in \Cref{thm:linnikbasiclemma}.
Therefore, we fix the notation of \Cref{thm:linnikbasiclemma} and recall it here for convenience of the reader.
Let $F$ be a number field of degree $n$ over $\Q$ and consider $\GL_2$ as an $F$-group.
For any finite place $u$ of $F$ we let $B_u = \GL_2(\mathcal{O}_{F,u})$.
For any Archimedean place $u$ we set
\begin{align}\label{eq:defBu}
B_u = \{A \in \GL_2(F_u):\max\{\norm{A},\norm{A^{-1}}\} \leq r_u\}
\end{align}
using the notation in \S\ref{sec:archdisc} where $r_u >1$ is fixed. Let $B = \prod_{u \in \VF}B_u$. 

Let $S$ be a finite set of places of $F$. 
For each $v\in S$, let $\bfA_v$ denote a split $F_v$-torus of $\GL_2$, and let $A_v=\bfA_v(F_v)$ be its group of $F_v$-points. Let $A_S:=\prod_{v\in S}A_v$.
We fix $a=(a_v)_{v\in S}\in A_S$, and let $\phi_v(a_v)\in F_v$ denote the unique eigenvalue of $\Ad\,a_v$ whose absolute value is $\geq1$. 
Let $\phi(a):=(\phi_v(a_v))_{v \in S}\in F_S$. 
For any integer $\tau \geq 1$ we define the $\tau$-Bowen ball
\begin{align*}
B_\tau = \bigcap_{-\tau \leq t \leq \tau} a^{-t}B a^t.
\end{align*}
For convenience, we also define $B_{u,\tau} =  \bigcap_{-\tau \leq t \leq \tau} a_u^{-t}B_u a_u^t$ for any place $u$ of $F$.

\subsection{Geometric invariant theory}\label{sec:GIT}

Let ${\bfT} < \GL_2$ be a torus defined over $F$. Let $\bfA < \GL_2$ be the diagonal subgroup. Consider the action of $\bfT\times\bfT$ on $\GL_2$ given by $(t_1,t_2).g=t_1gt_2^{-1}$. In geometric invariant theory, the universal categorical quotient for the ${\bfT} \times {\bfT}$ action on $\GL_2$ is representable by
\begin{align*}
\lrdoublequot{{\bfT}}{\GL_2}{{\bfT}} := \Spec F[\GL_2]^{{\bfT}\times{\bfT}}
\end{align*}
This is an affine variety defined over $F$ by a result of Hilbert, as ${\bfT}\times{\bfT}$ is reductive; c.f. \cite[Thm.3.5]{PV94}. Here, $F[\GL_2]^{{\bfT}\times{\bfT}}$ is the ring of bi-${\bfT}$-invariant regular functions on $\GL_2$.
We denote by $\pi_{{\bfT}}$ the canonical morphism
\begin{align*}
\pi_{{\bfT}}: \GL_2 \to \lrdoublequot{{\bfT}}{\GL_2}{{\bfT}}.
\end{align*}
The fiber of any point under $\pi_{{\bfT}}$ contains a unique ${\bfT} \times {\bfT}$-closed orbit.
If the image $\overline{\bfT}$ of ${\bfT}$ in $\PGL_2$ is anisotropic over $F$, work of Kempf \cite{Kem78} implies that the fiber of any point $\pi_{{\bfT}}(\gamma)$ with $\gamma \in \GL_2(F)$ is closed (see \cite[Prop.~3.6]{Kha15}).
This together with Hilbert's Theorem 90 implies that the induced map
\begin{align}\label{eq:inj on F-rat points}
\pi_{{\bfT}}: \lrquot{{\bfT}(F)}{\GL_2(F)}{{\bfT}(F)} \to \left(\lrdoublequot{{\bfT}}{\GL_2}{{\bfT}}\right)(F)
\end{align}
is injective -- see Lemma~\ref{lem:projectiononFpoints} for a proof.

\subsubsection{An explicit generator}
We give an explicit realization of $\lrdoublequot{{\bfT}}{\GL_2}{{\bfT}}$ as the affine line over $F$ by exhibiting a generator for $F[\GL_2]^{{\bfT}\times{\bfT}}$.

Consider first the diagonal subgroup ${\bfA} < \GL_2$ and write $x_{ij}$ with $i,j \in \{1,2\}$ for the coordinate functions on $\bfM_2$. The regular function
\begin{align*}
\psi = \frac{x_{12}x_{21}}{\det}
\end{align*}
is a generator for $F[\GL_2]^{{\bfA}\times\bfA}$ (cf.~\cite[\S 4]{Kha15}).

For an arbitrary maximal $F$-torus ${\bfT}< \GL_2$ let $K$ be the splitting field and let $g \in \GL_2(K)$ be such that $\Ad_g({\bfT}) = \bfA$. Then
\begin{align*}
\inv{\bfT} = \psi \circ \Ad_g
\end{align*}
is a generator for $K[\GL_2]^{{\bfT}\times{\bfT}}$. In fact, one can show that $\psi^{bfT}$ is defined over $F$ (cf.~\cite[\S 5]{Kha15}).
For illustration (in particular in view of later purposes), we shall give an example.

\begin{example}\label{exp:standardtori}
Let $D \in F^\times$ be such that $x^2-D$ is irreducible. In this case, the centralizer $\bfT$ of
\begin{align*}
X_D = \begin{pmatrix}
0 & 1 \\ D &0
\end{pmatrix}
\end{align*}
in $\GL_2$ is a non-split $F$ torus. Explicitly, any point in ${\bfT}(F)$ is of the form
\begin{align*}
\begin{pmatrix}
a & b \\ bD & a
\end{pmatrix}
\end{align*}
where $a,b \in F$. The splitting field of $\bfT$ is the splitting field of the characteristic polynomial of $X_D$, that is, $K =F[\sqrt{D}] \simeq \rquot{F[X]}{(X^2-D)}$. The matrix 
\begin{align*}
g = \begin{pmatrix}
1 &  1\\
\sqrt{D} & -\sqrt{D}
\end{pmatrix}^{-1} \in \GL_2(K)
\end{align*}
satisfies $\Ad_g({\bfT}) = \bfA$. As we explain more generally below, one can see that for $\gamma \in \bfM_2(F)$
\begin{align*}
g\gamma g^{-1} = \begin{pmatrix}
b_1 & b_2 \\ \sigma(b_2) & \sigma(b_1)
\end{pmatrix}
\end{align*}
where $\sigma$ is the non-trivial Galois automorphism of $K/F$ and $b_1,b_2 \in K$. In particular, $\inv{\bfT}(\gamma) = \frac{1}{\det(\gamma)} \Nr_{K/F}(b_2)$ is in $F$.
\end{example}

More generally, for an $F$-torus ${\bfT} < \GL_2$ denote by $\mathfrak{t} \subset \bfM_2$ its Lie-algebra.
For any non-zero traceless $X \in \mathfrak{t}(F)$ we have
\begin{align*}
{\bfT} = \{g \in \GL_2: \Ad_gX = X\}.
\end{align*}
Observe that $X^2 = -\det(X) \in F^\times$ and if $\bfT$ is non-split over $F$, $K = F(\sqrt{D})$ for $D = -\det(X)$ is a quadratic extension with $\bfT \simeq \Res_{K/F}(\mathbb{G}_{m,K})$.
If $v \in K^2$ is an eigenvector of $X$, so is $\sigma(v)$ when $\sigma \in \mathrm{Gal}(K/F) \setminus\{\mathrm{id}\}$. 
In particular, setting $g \in \GL_2(K)$ to be the inverse of the matrix with columns $v,\sigma(v)$ the traceless matrix $\Ad_g(X)$ is diagonal and $\Ad_g({\bfT}) = \bfA$.
By definition, $g$ satisfies that $\sigma(g) = w_{(12)} g $ where $w_{(12)}$ is the permutation matrix for the transposition $(1\,2)$.
In particular, for any $\gamma \in \GL_2(F)$ the matrix $g\gamma g^{-1}$ is of the form
\begin{align*}
g\gamma g^{-1} = \begin{pmatrix}
b_1 & b_2 \\ \sigma(b_2) & \sigma(b_1)
\end{pmatrix}
\end{align*}
for some $b_1,b_2 \in K$.
We remark that this property does not hold for an arbitrary $g \in \GL_2(K)$ with $\Ad_g({\bfT}) = \bfA$. The special choice of $g \in \GL_2(K)$ is unique up to left-multiples with diagonal matrices having entries in $F^\times$.

Let $\bfZ$ be the center of $\GL_2$. Let $\Delta:\bfT\hookrightarrow\bfT\times\bfT$ be the diagonal embedding.

\begin{lem}\label{lem:stabilizers}
Let $\gamma \in \GL_2(F)$ and define $\mathrm{Stab}_{\bfT\times\bfT}(\gamma) \subset \bfT\times \bfT$ to be the stabilizer of $\gamma$ under the action of $\bfT\times \bfT$ on $\GL_2$.
\begin{enumerate}[(i)]
\item If $\gamma \in \bfT(F)\setminus\bfZ(F)$, then $\mathrm{Stab}_{\bfT\times\bfT}(\gamma)=\Delta(\bfT)\simeq \bfT$.
\item If $\gamma \in \bfN_{\bfT}(F)\setminus\bfT(F)$, $\mathrm{Stab}_{\bfT\times\bfT}(\gamma)\simeq \bfT$ via the embedding $\bfT \hookrightarrow \bfT \times \bfT$ mapping a geometric point $t$ to $(t,\gamma^{-1}t\gamma)$.
Here, $\bfN_\bfT < \GL_2$ is the normalizer of $\bfT$.
\item If $\gamma \not\in \bfN_{\bfT}(F)$, the stabilizer $\mathrm{Stab}_{\bfT\times\bfT}(\gamma)$ is $\Delta(\bfZ)$.
\end{enumerate}
\end{lem}

\begin{proof}
As in the discussion preceding the lemma, let $K/F$ be the quadratic extension associated to $\bfT$ and let $g \in \GL_2(K)$ be such that $g\bfT g^{-1}$ is the diagonal subgroup and
 \begin{align*}
g\gamma g^{-1} = \begin{pmatrix}
b_1 & b_2 \\ \sigma(b_2) & \sigma(b_1)
\end{pmatrix}
\end{align*}
for some $b_1,b_2 \in K$ and the non-trivial Galois automorphism $\sigma$ of $K/F$. 
Case (i) corresponds to $b_2=0$, Case (ii) to $b_1= 0$, and Case (iii) to $b_1 \neq 0 \neq b_2$. Let $(t_1,t_2) \in \mathrm{Stab}_{\bfT\times\bfT}(\gamma)$ and write $\Ad_g(t_i)=\mathrm{diag}(\alpha_i,\sigma(\alpha_i))$ for $i=1,2$. By assumption on $(t_1,t_2)$ we have
\begin{align*}
b_1 \frac{\alpha_1}{\alpha_2} = b_1,\quad
b_2 \frac{\alpha_1}{\sigma(\alpha_2)} = b_2.
\end{align*}
If $b_1 \neq 0 \neq b_2$, $\alpha_1 = \alpha_2 = \sigma(\alpha_2) = \sigma(\alpha_1)$ so that $t_1=t_2$ and $t_1$ is a scalar matrix i.e.~in $\bfZ$. The remaining cases follow similarly.
\end{proof}

We remark that, as the proof shows, imposing the additional condition $\inv{\bfT}(\gamma) \neq -1$ rules out Case (ii) in Lemma~\ref{lem:stabilizers}.

\begin{lem}\label{lem:projectiononFpoints}
For any non-split $F$-torus $\bfT < \GL_{2,F}$ the map
\begin{align*}
\pi_{{\bfT}}: \lrquot{{\bfT}(F)}{\GL_2(F)}{{\bfT}(F)} \to \left(\lrdoublequot{{\bfT}}{\GL_2}{{\bfT}}\right)(F)
\end{align*}
is injective.
\end{lem}

\begin{proof}
Let $\gamma_1,\gamma_2 \in \GL_2(F)$ be such that $\pi_{\bfT}(\gamma_1) = \pi_{\bfT}(\gamma_2)$. By the work of Kempf \cite{Kem78} mentioned already at the beginning of \S\ref{sec:GIT} this implies that $\gamma_2\in (\bfT\gamma_1\bfT)(F)$. 
It remains to show that $\gamma_2\in \bfT(F)\gamma_1\bfT(F)$. 
If $\gamma_1 \in \bfN_{\bfT}(F)$, this is clear.
So suppose that $\gamma_1 \not\in \bfN_{\bfT}(F)$, so that the stabilizer for $\gamma_1$ under the $\bfT\times\bfT$-action is $\Delta\bfZ$ by Lemma~\ref{lem:stabilizers}.
Let $t_1,t_2 \in \bfT(\overline{F})$ be two points defined over the algebraic closure $\overline{F}$ so that $\gamma_2= t_1\gamma_1t_2^{-1}$.
Therefore, for any $\sigma \in \Gal(\overline{F}/F)$,
\begin{align*}
t_1\gamma_1t_2^{-1} = \gamma_2 = \sigma(\gamma_2)= \sigma(t_1)\gamma_1\sigma(t_2^{-1}),
\end{align*}
which implies that
\begin{align*}
(t_1/\sigma(t_1),t_2/\sigma(t_2)) \in \Delta\bfZ(\overline{F}).
\end{align*}
Write $s_\sigma = t_1/\sigma(t_1)\in \bfZ(\overline{F})$. As the Galois cohomology $H^1(\Gal(\overline{F}/F),\bfZ(\overline{F}))$ is trivial, there exists $s' \in \bfZ(\overline{F})$ so that $s_\sigma = s'/\sigma(s')$ for all $\sigma\in \Gal(\overline{F}/F)$.
In particular, $t_1' = t_1/s'$ and $t_2' = t_2/s'$ are $F$-rational points with
\begin{align*}
t_1'\gamma_1(t_2')^{-1} = t_1 \gamma_1 t_2^{-1} = \gamma_2,
\end{align*}
which proves the lemma.
\end{proof}

\subsection{Volume, discriminant, and local bounds on invariants}

In this subsection, we consider fixed a homogeneous toral set $Y = [\bfT g]$ and introduce certain local coordinates on $\GL_2(F_u)$ for any place $u$ of $F$ that will be used later on.
In particular, we obtain local bounds on the denominator of invariants.
The following proposition is technical, but very important in the rest of the argument for Theorem~\ref{thm:linnikbasiclemma}; we recommend reading it first in the non-Archimedean case only.
We also remark that this discussion already appears in \cite[\S5]{Khayutin-mixing} in a more general context.

We will use throughout the notations introduced in \S\ref{sec:discvol}.
In particular, we have associated to a homogeneous toral set $Y \subset [\GL_{2,F}]$ a field $K$ and an order (cf.~\S\ref{sec:nonarchdisc}).
Moreover, we have defined the local orders $\mathcal{O}_u \subset K_u$ for finite places $u$ of $F$ and write $\Delta_{\mathcal{O},u}$ for a choice of generator for the different ideal of $\mathcal{O}_{u}$ (cf.~\S\ref{sec:quadorders}). 
Lastly, recall for $u$ Archimedean the choice of open sets $B_u$ of `diameter' $r_u$ from \eqref{eq:defBu}.

\begin{prop}[Local coordinates]\label{prop:localconj}
Let $Y \subset [\GL_{2,F}]$ be homogeneous toral set and let $\mathcal{O} \subset K$ be the associated order resp.~field. 
Let $u$ be a place of $F$ and let $\sigma$ be the non-trivial automorphism of $K_u$ which fixes $F_u$ pointwise.
Then there exists $c_{Y,u} \in \GL_2(K_u)$ with the following properties:
\begin{enumerate}
\item If $u$ is a non-Archimedean place, then 
\begin{align*}
\Delta_{\mathcal{O},u}c_{Y,u},\, c_{Y,u}^{-1}  \in \bfM_2(\mathcal{O}_{u}).
\end{align*}
If $u$ is Archimedean, then for any place $w$ of $K$ above $u$
\begin{align*}
\norm{c_{Y,w}} \ll \disc_u(Y)^{-\frac{1}{2}} ,\
\norm{c_{Y,w}^{-1}} \ll 1.
\end{align*}
\item For any place $w$ of $K$ above $u$ we have $c_{Y,w}g_u^{-1}\bfT g_u c_{Y,w}^{-1} = \bfA$ as $K_w$-tori.
\item For any $\gamma \in \GL_2(F_u)$ the conjugate $c_{Y,u}\gamma c_{Y,u}^{-1}\in \GL_2(K_u)$ is of the form
\begin{align*}
\begin{pmatrix}
b_{1,u} & b_{2,u} \\ \sigma(b_{2,u}) & \sigma(b_{1,u})
\end{pmatrix}
\end{align*}
where $b_{1,u},b_{2,u} \in K_u$. We call the pair $(b_{1,u},b_{2,u})$ the \emph{local coordinates} of $\gamma$ (relative to $Y$).
\item If $u$ is non-Archimedean, the local coordinates $(b_{1,u},b_{2,u})$ of any $k \in \GL_2(\mathcal{O}_{F,u})$ satisfy 
\begin{itemize}
\item $b_{1,u},b_{2,u}\in \frac{1}{\Delta_{\mathcal{O},u}} \mathcal{O}_{u}$,
\item $b_{1,u}-b_{2,u} \in \mathcal{O}_{u}$,
\item $\Tr_{K_u/F_u}(b_{1,u}),\Tr_{K_u/F_u}(b_{2,u}) \in \mathcal{O}_{F,u}$, and
\item $\Nr_{K_u/F_u}(b_{1,u})-\Nr_{K_u/F_u}(b_{2,u}) \in \mathcal{O}_{F,u}^\times$.
\end{itemize}
If $u$ is Archimedean and $w \mid u$, then the local coordinates $(b_{1,u},b_{2,u})$ of any $k \in B_u$ satisfy 
\begin{align*}
\max\{\abs{b_{1,w}},\abs{b_{2,w}} \}\ll r_u\disc_u(Y)^{\frac{1}{2}}
\quad \text{and}\quad 
\abs{b_{1,w}-b_{2,w}}\ll r_u.
\end{align*}
\item If $u$ is non-Archimedean and the local coordinates $(b_{1,u},b_{2,u})$ of $\gamma \in \GL_2(F_u)$ satisfy the properties in (4), then $\gamma \in \GL_2(\mathcal{O}_{F,u})$.
\end{enumerate}
\end{prop}

We prove Proposition~\ref{prop:localconj} separately in the Archimedean and the non-Archimedean case beginning with the latter.
The proof in the non-Archimedean case is in essence contained in \cite[Proposition~7.4]{Kha15}; we include a proof here for the convenience of the reader and because it is arguably more concrete for quadratic extensions.

\begin{proof}[Proof in the non-Archimedean case]
We aim to imitate the situation in Example~\ref{exp:standardtori}.
Let $u$ be a non-Archimedean place of $F$.
There exists $v \in F_u^2$ such that
\begin{align*}
\theta: K_u \to F_u^2,\ x \mapsto vx
\end{align*}
is an isomorphism of $K_u$-modules. In fact, all points in $v \in F_u^2$ outside of (at most) two lines have this property.

We write $\mathfrak{a} = \theta^{-1}(\mathcal{O}_{F,u}^2g_u^{-1})$. 
Since $\mathcal{O}_{u} = K_u \cap g_u\bfM_2(\mathcal{O}_{F,u})g_u^{-1}$ (by definition in \eqref{eq:deflocalorder}) and $\theta$ is $K_u$-equivariant, $\mathfrak{a}$ is a proper $\mathcal{O}_u$-ideal. 
By \Cref{prop:locordersmono}, $\mathfrak{a}$ is principal and hence there exists $\lambda \in C_u$ with $\mathfrak{a} = \lambda \mathcal{O}_{u}$.

Let $\iota: K_u \hookrightarrow \bfM_2(F_u)$ be the embedding given by representing multiplication by elements in $K_u$ in the basis $b_i = \theta^{-1}(e_ig_u^{-1})$ of $\mathfrak{a}$ where $e_i$ is the standard basis.
One may verify $\iota(\lambda) = g_u^{-1}\lambda g_u$ for all $\lambda \in K_u$ noting that $K_u \subset \bfM_2(F_u)$ by definition.
Changing the basis $b_i$ of $\mathfrak{a}$ has the effect of conjugating $\iota(\cdot)$ by an element of $\GL_2(\mathcal{O}_{F,u})$ or equivalently multiplying $g_u$ by an element of $\GL_2(\mathcal{O}_{F,u})$ on the right. 
In view of the statement of the proposition we are proving, we may thus assume that the basis $b_i$ is a basis of our choosing.

By \Cref{prop:locordersmono}, $\mathcal{O}_{u}$ is monogenic over $\mathcal{O}_{F,u}$ i.e.~there exists $\alpha \in \mathcal{O}_{u}$ with $\mathcal{O}_{u} = \mathcal{O}_{F_u}[\alpha]$.
Consider the basis
\begin{align}\label{eq:choiceofnonarchbasis}
b_1 = \lambda, \; b_2 = \lambda \alpha.
\end{align}

Multiplication by $\alpha$ in the basis $b_1,b_2$ in \eqref{eq:choiceofnonarchbasis} is represented by
\begin{align*}
\begin{pmatrix}
0 & 1 \\ -\Nr_{C_u/F_u}(\alpha) & \Tr_{C_u/F_u}(\alpha)
\end{pmatrix}.
\end{align*}
Define
\begin{align*}
c_{Y,u} = \begin{pmatrix}
1 & 1\\ \alpha & \sigma(\alpha)
\end{pmatrix}^{-1}.
\end{align*}
By construction, $c_{Y,w}^{-1}\iota(\lambda)c_{Y,w}$ is diagonal for all $\lambda\in K_w$ as required in (2).

Note that both $c_{Y,u}^{-1}$ and $\Delta_{\mathcal{O},u}c_{Y,u}$ are in $\bfM_2(\mathcal{O}_{u})$, as $\det(c_{Y,u}^{-1})=\sigma(\alpha)-\alpha$ is a generator of the different ideal by Lemma~\ref{lem:different}.
Thus, (1) holds.
Without loss of generality we may assume that $\Delta_{\mathcal{O},u}=\sigma(\alpha)-\alpha$. We also note that the local absolute discriminant $\disc_u(Y)$ is related to $\Delta_{\mathcal{O},u}$ and $\alpha$ in the following way: 
\begin{equation*}
	\disc_u(Y)
	=\abs{\Nr_{K_u/F_u}(\Delta_{\mathcal{O},u})}_u^{-1}=\abs{2\Nr_{K_u/F_u}(\alpha)-\Tr_{K_u/F_u}(\alpha^2)}_u^{-1}.
\end{equation*}

Applying $\sigma$ to $c_{Y,u}$ switches its rows; this implies (3).
Noting that $c_{Y,u}^{-1}(1,-1)^t \in \Delta_{\mathcal{O},u} \mathcal{O}_u^2$
the remaining claims in (4) and (5) are verified by direct calculation.
\end{proof}

\begin{proof}[Proof in the Archimedean case]
The proof consists mostly of brute force calculations.
Suppose in the following that $u$ is Archimedean and let $w\mid u$ be a place of $K$.
Observe that $\overline{K_w} = \overline{F_u}$.
Let $\bfE_u \subset \bfM_2$ be the centralizer of $g_u^{-1}\bfT(F_u)g_u$ and let $f \in \bfE_u(F_u)$ be non-zero with $\Tr(f) = 0$ and $\norm{f} =1$.
Write
\begin{align*}
f = \begin{pmatrix}
a & b \\ c & -a
\end{pmatrix}
\end{align*}
and let $\pm \alpha \in K_w$ be the eigenvalues of $f$. 
In view of Example~\ref{exp:archdiscdim2} we have $\disc_u(Y) = \frac{1}{4|\alpha|^2}$.
Furthermore, $\norm{f} = 1$ implies $|\alpha|= \sqrt{|\det(f)|} \leq \frac{1}{\sqrt{2}}$.

In view of (2) we wish to diagonalize $f$.
From $\norm{f} = 1$ we know that one of $|a|,|b|,|c|$ is at least $\frac{1}{2}$. By conjugating $f$ with a unipotent matrix first if necessary, we may suppose that $|b| \geq \frac{1}{2}$ (the case $|c| \geq \frac{1}{2}$ is analogous).
Set 
\begin{align*}
c_{Y,w} = \begin{pmatrix}
b & b \\ \alpha-a & -\alpha-a
\end{pmatrix}^{-1}.
\end{align*}
We have $c_{Y,w}fc_{Y,w}^{-1}=\diag(\alpha,-\alpha)$, and thus (2) holds. The property in (1) is also readily verified.

For any $\gamma=\begin{pmatrix}
A & B \\
C & D
\end{pmatrix}\in \GL_2(F_u)$, we explicitly calculate
\begin{equation*}
	c_{Y,w}\gamma c_{Y,w}^{-1}=
	\begin{pmatrix}
	\frac{\alpha+a}{2\alpha}A+\frac{b}{2\alpha}C+\frac{c}{2\alpha}B+\frac{\alpha-a}{2\alpha}D & 
	\frac{\alpha+a}{2\alpha}A+\frac{b}{2\alpha}C-\frac{(\alpha+a)^2}{2\alpha b}B-\frac{\alpha+a}{2\alpha}D \\
	\frac{\alpha-a}{2\alpha}A-\frac{b}{2\alpha}C+\frac{(\alpha-a)^2}{2\alpha b}B-\frac{\alpha-a}{2\alpha}D &
	\frac{\alpha-a}{2\alpha}A-\frac{b}{2\alpha}C-\frac{c}{2\alpha}B+\frac{\alpha+a}{2\alpha}D
	\end{pmatrix}.
\end{equation*}
If $u$ is non-split i.e. $w$ is the unique place above $u$, then $\sigma(\alpha)=-\alpha$, and (3) can be checked easily. 
If $u$ is split i.e. $w_1, w_2$ are two places above $u$, then in the above construction for $c_{Y,w}$, we use $\alpha$ and $-\alpha$ for $c_{Y,w_1}$ and $c_{Y,w_2}$ respectively, and one can again check that (3) still holds (since $\sigma$ switches $w_1$ and $w_2$ in this case).

Finally, suppose $\gamma=k\in B_u$. We have $\max\{ \abs{A},\abs{B},\abs{C},\abs{D} \}\leq r_u$. By our assumption, we have $\frac{1}{\abs{2\alpha}}=\disc_u(Y)^{\frac{1}{2}}$, $\abs{a}\leq 1$, $\frac{1}{2}\leq\abs{b}\leq1$, $\abs{c}\leq1$ and $\abs{\alpha}\leq1$. Note that $b_{1,w}-b_{2,w}=\frac{a+\alpha}{b}B+D$. Combining the above, one gets the desired bounds for $\abs{b_{1,w}}$, $\abs{b_{2,w}}$ and $\abs{b_{1,w}-b_{2,w}}$.
\end{proof}

\begin{cor}[Bounds on invariants]\label{cor:boundsinvariants}
Let $Y \subset [\GL_{2,F}]$ be a homogeneous toral set and let $u$ be a place of $F$.
Then for any $k \in B_u$ we have 
\begin{align}\label{eq:locbounds}
|\inv{\bfT}(g_ukg_u^{-1})|_u \leq \kappa_u\disc_u(Y),\quad |1+\inv{\bfT}(g_ukg_u^{-1})|_u \leq \kappa_u\disc_u(Y).
\end{align}
where the constants satisfy $\kappa_u=1$ if $u$ is non-Archimedean, 
and $\kappa_u\ll_n r_u^4$ if $u$ is Archimedean (here, $n=[F:\Q]$).

Moreover, if $u\in S$ then for any $k \in B_{u,\tau}$
\begin{align}\label{eq:locbounds2}
|\inv{\bfT}(g_ukg_u^{-1})|_u \leq \kappa_u\disc_u(Y) \abs{\phi_u(a_u)}_u^{-2\tau}.
\end{align}
\end{cor}

Observe that assuming maximal type for the homogeneous toral set would imply that $\disc_u(Y) = 1$ for places $u\in \VF$ that are inert or split in the associated field and slightly simplify the proof below.

\begin{proof}
Let $k\in B_u$ and let $(b_{1,u},b_{2,u})$ be its local coordinates from Proposition~\ref{prop:localconj}.
The bound in \eqref{eq:locbounds} follows from 
\begin{align*}
\inv{\bfT}(g_ukg_u^{-1}) = \det(k)^{-1}\Nr_{K_u/F_u}(b_{2,u}),\ 
1+\inv{\bfT}(g_ukg_u^{-1}) = \det(k)^{-1}\Nr_{K_u/F_u}(b_{1,u})
\end{align*}
and the denominator bounds in Proposition~\ref{prop:localconj}(4).

So suppose from now on that $u\in S$. 
In this case, $K_u \simeq F_u \times F_u$ and $c_{Y,w} \in \GL_2(F_u)$ for some $w \in \VK$ with $w \mid u$ satisfies that $c_{Y,w}a_u c_{Y,w}^{-1} = \diag(\alpha_1,\alpha_2)$ where $((\alpha_1,\alpha_2),(0,0))$ is the local coordinate of $a_u$.
In particular, $\phi_u(a_u) \in \{\frac{\alpha_1}{\alpha_2},\frac{\alpha_2}{\alpha_1}\}$.
Furthermore, write $b_{2,u} = (x_1,x_2)$ so that by the above
\begin{align}\label{eq:splitboundloccoeff}
\inv{\bfT}(g_ukg_u^{-1}) = \det(k)^{-1}x_1x_2.
\end{align}

Assume that $u$ is Archimedean. Then $|x_1|_u,|x_2|_u \ll r_u\disc_u(Y)^{\frac{1}{2}}$ by Proposition~\ref{prop:localconj} and by \eqref{eq:splitboundloccoeff} it suffices to show that
\begin{align*}
|x_1|_u,|x_2|_u \ll r_u\disc_u(Y)^{\frac{1}{2}}|\phi_u(a_u)|_u^{-\tau}.
\end{align*}
To do so, note that for any $t \in \{-\tau,\ldots,\tau\}$ the requirement $a_u^t k a_u^{-t}\in B_u$ implies that  $|(\alpha_1/\alpha_2)^tx_1|_u \ll r_u\disc_u(Y)^{\frac{1}{2}}$.
Indeed, one may see from the definition that the second local coordinate of $a_u^t k a_u^{-t}\in B_u$ is the pair $((\alpha_1/\alpha_2)^tx_1,(\alpha_2/\alpha_1)^tx_2)$.
Choosing $t \in \{-\tau,\tau\}$ we obtain that the bound $|x_1|_u \ll r_u\disc_u(Y)^{\frac{1}{2}}|\phi_u(a_u)|_u^{-\tau}$ as desired (for $x_2$ one proceeds analogously).

Assume that $u$ is non-Archimedean.
Under the isomorphism $K_u \simeq F_u \times F_u$, the order $\mathcal{O}_{u}$ takes the form $\mathcal{O}_u = \mathcal{O}_{F,u}(1,1) + \mathcal{O}_{F,u} (c_1,c_2)$ where $c_1 \neq c_2$. Such an order is automatically Galois invariant.
Furthermore, set $\Delta = \Delta_{\mathcal{O},u} = (\Delta_1,\Delta_2)$ where $\Delta_1 = -\Delta_2 = c_2-c_1$.
As in the Archimedean case,  the requirement $a_u^t k a_u^{-t}\in B_u$ for any $t \in \{-\tau,\ldots,\tau\}$ implies that 
\begin{align*}
\big((\alpha_1/\alpha_2)^tx_1,(\alpha_2/\alpha_1)^t x_2 \big) \in \tfrac{1}{\Delta}\mathcal{O}_u.
\end{align*}
In particular, we have $|(\alpha_1/\alpha_2)^tx_i|_u \leq |\Delta_i|_u^{-1}$ for $i=1,2$ which implies that $|x_i|_u \leq |\Delta_i|_u^{-1}|\phi_u(a_u)|_u^{-\tau}$ when choosing $t \in \{-\tau,\tau\}$ appropriately.
Taking the product
\begin{align*}
|\inv{\bfT}(g_ukg_u^{-1})|_u 
\leq |\Delta_1\Delta_2|_u^{-1}|\phi_u(a_u)|_u^{-2\tau}
= |\Nr(\Delta)|_u^{-1} |\phi_u(a_u)|_u^{-2\tau}
\end{align*}
proving the corollary.
\end{proof}

\subsection{Geometric expansion for the Bowen kernel}\label{sec:geomexpansionLinnik}

Let $f$ be the characteristic function of $B$ (where $B$ is defined in \eqref{eq:defBu}) and for an integer $\tau \geq 0$ let $f_\tau$ be the characteristic function of $B_\tau$. Note that $f_0 = f$.

We define the kernel (henceforth sometimes called $\tau$-\emph{Bowen kernel})
\begin{align*}
K_\tau(x,y) = \sum_{\gamma \in \GL_2(F)} f_\tau (x^{-1}\gamma y)
\end{align*}
for all $x,y \in \GL_2(\bbA_F)^1$.
Note that $K_\tau$ defines a function on $[\GL_{2,F}]$ and that for any pair of points $x,y$ the above sum is finite (though the number of non-zero terms grows towards the cusp).

Throughout this section, we shall consider fixed a homogeneous toral set $Y=\GL_2(F)\bfT(\bbA_F)^1g$. We let $K$ be the associated field (cf.~\S\ref{sec:homogeneoustoralsets}) and $\mathcal{O} \subset K$ be the associated order (cf.~\S\ref{sec:nonarchdisc}). Note that we do not impose a maximal type assumption yet (i.e.~potentially $\mathcal{O} \subsetneq \mathcal{O}_{K}$).
Let $\de x$ etc denote the 'standard' Haar measure on $g^{-1}\bfT(\bbA_F)^1g$ or $\bfT(\bbA_F)^1$ i.e.~the one that gives measure one to $g^{-1}Bg$ or $B$ respectively.
Recall that $\mu_Y$ is the invariant probability measure on $Y$ and that  $\widetilde{\mu_Y}$ is its lift to $g^{-1}\bfT(\bbA_F)^1g$.

It is straightforward to verify that
\begin{align}\label{eq:bounded by kernel integral}
\mu_{Y}\times \mu_{Y}\big(\{ (x,y): y \in x B_\tau\}\big)  
\leq \int_{Y}\int_{Y} K_\tau \de \mu_{Y}(x)\de \mu_{Y}(y);
\end{align}
we shall estimate the latter expression.
To do so, expand
\begin{align}
&\int_{Y}\int_{Y} K_\tau(x,y) \de \mu_{Y}(x)\de \mu_{Y}(y)
= \int_{[\bfT]^2} K_\tau(xg,yg) \de\mu_{[\bfT]}(x) \de\mu_{[\bfT]}(y) \nonumber \\
&\quad = \sum_{\gamma \in \lrquot{\bfT(F)}{\GL_2(F)}{\bfT(F)}} \int_{[\bfT]^2} \sum_{\eta \in \bfT(F)\gamma\bfT(F)} f_\tau(g^{-1}x^{-1}\eta y g) \de\mu_{[\bfT]}(x) \de\mu_{[\bfT]}(y) \label{eq:geomexpansion}
\end{align}
We now analyze the above sum; for $\gamma \in \bfT(F)$, $\gamma \in \bfN_\bfT(F)\setminus\bfT(F)$ and $\gamma \not\in \bfN_\bfT(F)$ respectively where the last case is the most interesting one. 
We will use the bounds in \Cref{lem:counting lattice points} to control the set of $\gamma$'s for which the above integral is non-zero and then control each of the integrals.

\subsubsection{The identity contribution}\label{sec:contr1}
When $\gamma$ is the trivial coset of $ \lrquot{\bfT(F)}{\GL_2(F)}{\bfT(F)}$, the corresponding integral in \eqref{eq:geomexpansion} is by unfolding
\begin{align*}
\int_{[\bfT]^2} \sum_{\eta \in \bfT(F)} &f_\tau(g^{-1}x^{-1}\eta y g) \de\mu_{[\bfT]}(x) \de\mu_{[\bfT]}(y) \\
=& \int_{[\bfT]}\int_{\bfT(\bbA_F)^1}f_\tau(g^{-1}x^{-1} y g) \de\muTtilde(y)\de\mu_{[\bfT]}(x) \\
=& \int_{[\bfT]}\int_{\bfT(\bbA_F)^1}f_\tau(g^{-1}y g) \de\muTtilde(y)\de\mu_{[\bfT]}(x) \\
=& \int_{\bfT(\bbA_F)^1}f_\tau(g^{-1}y g) \de\muTtilde(y)
\end{align*}
Observe that $g B_\tau g^{-1}$ is also a Bowen ball, namely for time parameter $\tau$, transformation $gag^{-1}$, and open set $gBg^{-1}$. If $x \in gB g^{-1} \cap \bfT(\bbA_F)^1$, then $x \in ga^{-t}B a^t g^{-1}$ for all $t \in \Z$ as $x$ commutes with $gag^{-1}$ and so $x \in gB_\tau g^{-1}$ for all $\tau \geq 0$.
Therefore, 
\begin{align*}
\int_{\bfT(\bbA_F)^1}f_\tau(g^{-1}y g) \de\muTtilde(y)
&= \int_{\bfT(\bbA_F)^1}f(g^{-1}y g) \de\muTtilde(y)\\
&= \muTtilde(gBg^{-1})
= \frac{1}{\vol(Y)}.
\end{align*}

\subsubsection{Normalizer contribution}\label{sec:normalizercontribution}
We now consider the contribution of the normalizer $\bfN_\bfT(F)$ to \eqref{eq:geomexpansion}; the generic case (cf.~\Cref{lem:stabilizers}) will take up the rest of the section starting with \S\ref{sec:contr2}.
Observe that $\bfN_{\bfT}(F)$ consists of two $\bfT(F)$-cosets. Let $\gamma \in \bfN_{\bfT}(F)$ be a representative of the non-trivial coset.
Estimating as in \S\ref{sec:contr1} the contribution of the non-trivial coset in $\bfN_{\bfT}(F)$ to \eqref{eq:geomexpansion} is bounded by
\begin{align*}
&\int_{[\bfT]^2} \sum_{\eta \in \bfT(F)} f(g^{-1}x^{-1}\gamma\eta y g) \de\mu_{[\bfT]}(x) \de\mu_{[\bfT]}(y) \\
&= \int_{[\bfT]}\int_{\bfT(\bbA_F)^1}f(g^{-1}x^{-1} \gamma y g) \de\muTtilde(y)\de\mu_{[\bfT]}(x) \\
&= \int_{\bfT(\bbA_F)^1}f(g^{-1}\gamma y g) \de\muTtilde(y)
\end{align*}
by substitution.
Replacing the coset of $\gamma$ within $\bfN_{\bfT}(\bbA_F)^1/\bfT(\bbA_F)^1$ we may suppose that its local coordinates in the sense of Proposition~\ref{prop:localconj} are $(0,1)$ (see also the proof of Lemma~\ref{lem:stabilizers}).
Now fix $y$ in the support of the integrand and let $(t_{u},0)$ be its local coordinates.
As $\gamma y \in gBg^{-1}$ we know from Proposition~\ref{prop:localconj}(4) that $t_{u} \in \mathcal{O}_{u}$ for every finite place $u$ and $|t_{w}| \ll r_u$ for every Archimedean place $w\mid u$.
The measure of such $y$'s with respect to $\muTtilde$ is $\ll \frac{1}{\vol(Y)}$.

To summarize, \S\ref{sec:contr1} and \S\ref{sec:normalizercontribution} together show that
\begin{align}\label{eq:normalizer contribution}
\sum_{\gamma \in \lrquot{\bfT(F)}{\bfN_{\bfT}(F)}{\bfT(F)}} \int_{[\bfT]^2} \sum_{\eta \in \bfT(F)\gamma\bfT(F)} f_\tau(g^{-1}x^{-1}\eta y g) \de\mu_{[\bfT]}(x) \de\mu_{[\bfT]}(y)
\ll \frac{1}{\vol(Y)}.
\end{align}

\subsubsection{Generic contributions: Using estimates on invariants}\label{sec:contr2}

In view of the geometric expansion in \eqref{eq:geomexpansion} and the estimate of the non-generic contributions in \eqref{eq:normalizer contribution} above it remains to estimate
\begin{equation}\label{eq:geomexpansion2}
\begin{split}
&\sum_{\substack{\gamma \in \lrquot{\bfT(F)}{\GL_2(F)}{\bfT(F)}\\ \gamma \not\in \bfN_\bfT(F) }} \int_{[\bfT]^2} \sum_{\eta \in \bfT(F)\gamma\bfT(F)} f_\tau(g^{-1}x^{-1}\eta y g) \de \mu_{[\bfT]}(x)\de\mu_{[\bfT]}(y) \\
&\qquad\qquad  =\frac{1}{\vol(Y)^2}\sum_{\gamma}
\int_{\lquot{\Delta(\bfZ)(F)}{(\bfT(\bbA_F)^1)^2}} f_\tau(g^{-1}x^{-1}\gamma yg) \de x\de y\\
&\qquad\qquad =\frac{\vol\left([\bfZ]\right)}{\vol(Y)^2}\sum_{\gamma}
\int_{\lquot{\Delta(\bfZ)(\bbA_F)^1}{(\bfT(\bbA_F)^1)^2}} f_\tau(g^{-1}x^{-1}\gamma yg) \de x\de y\\
\end{split}
\end{equation}
Here, we used that the stabilizer for the $\bfT \times \bfT$-action of any $\gamma\not\in\bfN_\bfT(F)$ is $\Delta(\bfZ)$, where $\bfZ$ is the center of $\bfG$ and $\Delta\colon \bfZ\to\bfZ\times\bfZ$ is the diagonal embedding (cf.~Lemma~\ref{lem:stabilizers}).

We now estimate the number of $\gamma$'s for which the above integral does not vanish. Write
\begin{equation}\label{eq:definition of r}
r:=\max\{r_u:u \in \mathcal{V}_{F,\infty}\}
\end{equation} 
which is a constant depending only on $B_\infty$. We call $r$ the \emph{diameter} of $B_\infty$.

The following lemma is crucial in the proof of Theorem~\ref{thm:linnikbasiclemma}; it estimates the number of non-zero contributions to the geometric expansion in \eqref{eq:geomexpansion} using the counting results of \S\ref{sec:countinglemma}.

\begin{lem}\label{lem:counting lattice points}
Let $c$ be a positive real number. Suppose that $\disc(Y)|\phi(a)|_S^{-2\tau}\geq cD_F$. The number of $\gamma \in \lrquot{\bfT(F)}{\GL_2(F)}{\bfT(F)}$ for which the integral in \eqref{eq:geomexpansion2} does not vanish is  $\ll_{n,c} r^{4n}\frac{\disc(Y)}{\absdisc_F^{1/2}} |\phi(a)|_S^{-2\tau}$.
\end{lem}

\begin{proof}
The proof consists of putting together already proven estimates. Let $\gamma\in \GL_2(F)$ be so that \eqref{eq:geomexpansion2} is non-zero. 
In particular, there exist  $x,y \in \bfT(\bbA_f)$ with $k := x^{-1}\gamma y \in gB_\tau g^{-1}$. Note that by definition of the invariant $\inv{\bfT}$ we have $\inv{\bfT}(k_u) = \inv{\bfT}(\gamma)$ for any place $u$. In particular, \Cref{cor:boundsinvariants} implies that for any $u \notin S$
\begin{align}\label{eq:denomboundsinv}
|\inv{\bfT}(\gamma)|_u \leq \kappa_u{\disc_u(Y)},
\end{align}
where $\kappa_u$ is as in \Cref{prop:localconj}.
Moreover, for any place $u\in S$ (i.e.~a place with contraction) we have $|\inv{\bfT}(\gamma)|_u \leq \kappa_u\disc_u(Y)\abs{\phi_u(a_u)}_u^{-2\tau}$. 
Putting these estimates together and using \Cref{lem:counting lemma} we obtain that
\begin{equation*}
\begin{split}
\#\{\inv{\bfT}(\gamma): \gamma \text{ as in the lemma}\}&\ll_{n,c} \prod_{u \in \VF}\kappa_u\cdot\frac{\disc(Y)}{\absdisc_F^{1/2}} |\phi(a)|_S^{-2\tau}\\
&\ll_n r^{4n}\frac{\disc(Y)}{\absdisc_F^{1/2}} |\phi(a)|_S^{-2\tau},
\end{split}
\end{equation*}
As $\inv{\bfT}$ is injective on $\lrquot{\bfT(F)}{\GL_2(F)}{\bfT(F)}$ by Lemma~\ref{lem:projectiononFpoints}, the lemma follows.
\end{proof}

\subsubsection{Estimates for torus integrals}\label{sec:orbitintegrals}
Let $\gamma \in \GL_2(F)\setminus \bfN_{\bfT}(F)$ and 
recall that the stabilizer of $\gamma$ (under the $\bfT\times\bfT$-action) is the diagonal embedded copy $\Delta \bfZ < \bfT\times\bfT$ of the center $\bfZ < \GL_2$.
In view of Lemma~\ref{lem:counting lattice points}, it suffices to estimate
\begin{align*}
I_\gamma := \int_{\lquot{\Delta \bfZ(\bbA_F)^1}{\bfT(\bbA_F)^1 \times \bfT(\bbA_F)^1}} 1_{gBg^{-1}}(x^{-1}\gamma y) \de x \de y
\end{align*}
Here, $\de x$ resp.~$\de y$ denotes integration with respect to the Haar measure normalized for $gBg^{-1}$.
We will assume in the following that $I_\gamma\neq 0$.
In particular, $\gamma$ satisfies the local denominator bounds on its invariant in the proof of Lemma~\ref{lem:counting lattice points}.

\begin{prop}[Orbital integrals]\label{prop:torus integral estimate}
Assume that $Y$ has maximal type.
	For any $\varepsilon>0$, we have
	\begin{equation*}
		I_\gamma\ll_{n,\varepsilon}r^{n\varepsilon}\disc(Y)^\varepsilon,
	\end{equation*}
	where $r$ is defined in \eqref{eq:definition of r}.
\end{prop}

Note that the maximal type assumption will only be used in Lemma~\ref{lem:locint bound I} below.
We make some preparations before proving this proposition. By substitution
\begin{align*}
I_\gamma
= \int_{\bfT(\bbA_F)^1} \int_{\lquot{\bfZ(\bbA_F)^1}{\bfT(\bbA_F)^1}} 1_{gBg^{-1}}(x^{-1}\gamma x y) \de x \de y.
\end{align*}
We further disintegrate with respect to the variable $y$ as follows:
Setting
\begin{align*}
\bfT^1 = \mathrm{ker}(\det:\bfT \to \mathbb{G}_{m,F})
\end{align*}
note that the quotient $\lquot{\bfT^1(\bbA_F)}{\bfT(\bbA_F)^1}$ is naturally identified with
\begin{align*}
\Nr(\bbA_K^\times)^1 = \{s \in \Nr(\bbA_K^\times) = \det(\bfT(\bbA_F)): |s|_{\bbA_F} = 1\} \subset \bbA_F^1
\end{align*}
via the determinant.
Furthermore, we have $\lquot{\bfZ(\bbA_F)^1}{\bfT(\bbA_F)^1} \simeq \lquot{\bfZ(\bbA_F)}{\bfT(\bbA_F)}$.
To summarize
\begin{align*}
I_\gamma
= \int_{\lquot{\bfT^1(\bbA_F)}{\bfT(\bbA_F)^1}} \int_{\bfT^1(\bbA_F)}\int_{\lquot{\bfZ(\bbA_F)}{\bfT(\bbA_F)}} 1_{gBg^{-1}}(x^{-1}\gamma x y s) \de x \de y \de s.
\end{align*}
Note that the Haar measure on $\bfT^1(\bbA_F)$ and $\bfT(\bbA_F)^1$ is normalized so that the respective intersection with $gBg^{-1}$ has measure $1$. The measure on $\lquot{\bfT^1(\bbA_F)}{\bfT(\bbA_F)^1}$ is the induced measure.
Whenever the tuple $(x,y,s)$ is such that $x^{-1}\gamma x y s \in gBg^{-1}$, we must have 
\begin{align}\label{eq:restrictiontodet}
\det(s) \in \frac{1}{\det(\gamma)}\det(B),
\end{align}
where $\det(B):=\{ \det(b)\mid b\in B \}$.

We will deduce a bound for the inner double integral
\begin{align*}
I_{\gamma,s} = \int_{\bfT^1(\bbA_F)}\int_{\lquot{\bfZ(\bbA_F)}{\bfT(\bbA_F)}} 1_{gBg^{-1}}(x^{-1}\gamma x y s) \de x \de y
\end{align*}
for fixed $s$ which will imply the claim in \Cref{prop:torus integral estimate} when combined with the following lemma:

\begin{lem}[Norm measure]\label{lem:norm measure}
The measure of the set of points
\begin{align*}
s \in \Nr(\bbA_K^\times)^1 \cap \frac{1}{\det(\gamma)}\det(B)
\end{align*}
is $\ll_\varepsilon \disc(Y)^\varepsilon$. Here, the Haar measure on $\Nr(\bbA_K^\times)^1$ is normalized in accordance with $\lquot{\bfT^1(\bbA_F)}{\bfT(\bbA_F)^1}$.
\end{lem}

\begin{proof}
Using local coordinates
\begin{align}\label{eq:normalizationNorm}
U:= \bfT(\bbA_F)^1\cap gBg^{-1} \simeq (C_\infty \times \widehat{\mathcal{O}}^\times)^1 = C_\infty^1 \times \widehat{\mathcal{O}}^\times
\end{align}
where $C_\infty$ is a bounded neighborhood of the identity in $K_\infty^\times = \prod_{u\mid \infty}K_u^\times$.
Explicitly, the above isomorphism is given by mapping $t \in \bfT(\bbA_F)^1$ with local coordinates $(a,0)$ to $a$ (see \Cref{prop:localconj}) and the set $C_\infty$ is contained in the set of points $a \in K_\infty^\times $ with $(\kappa r_u)^{-1}\leq |a_w| \leq \kappa r_u$ for any $w\mid u$, $u$ an Archimedean place of $F$ and some constant $\kappa>0$.
Moreover, upon closer inspection of the proof of Proposition~\ref{prop:localconj} one may see that the converse inclusion holds up to adapting the constants:
there is a constant $\kappa'< \kappa$ so that 
\begin{align*}
C_\infty \supset \big\{a\in K_\infty^\times: (\kappa' r_u)^{-1}\leq|a_w| < \kappa' r_u \text{ for all } u \mid \infty \text{ and }w\mid u\big\}.
\end{align*}

Under the map in \eqref{eq:normalizationNorm}, the intersection $\bfT^1(\bbA_F)\cap gBg^{-1}$ corresponds to norm one elements in $C_\infty^1 \times \widehat{\mathcal{O}}^\times$.
In particular, the induced Haar measure $m$ on $\Nr(\bbA_K^\times)^1$ is normalized so that $m(\Nr(U))=1$ where for the purposes of this proof we make the identification in \eqref{eq:normalizationNorm}.

To bound the measure $m(\frac{1}{\det(\gamma)}\det(B))$ we first observe that for any two points $s_1,s_2 \in \frac{1}{\det(\gamma)}\det(B)$ we have
\begin{align*}
s_2 \in \frac{1}{\det(\gamma)}\det(B) \subset s_1 \det(BB^{-1}).
\end{align*}
Therefore, by invariance
\begin{align*}
m\Big(\frac{1}{\det(\gamma)}\det(B)\Big) \leq m(\det(BB^{-1}))
\end{align*}
and we bound the right-hand side instead, which is independent of $\gamma$.

Suppose that $u$ is a non-Archimedean non-dyadic place of $F$. 
If $\Nr(\Delta_{\mathcal{O},u})$ is a unit in $\mathcal{O}_{F,u}$, the norm is surjective as a map $\Nr: \mathcal{O}_{u}^\times \to \mathcal{O}_{F,u}^\times$ by Lemma~\ref{lem:normimage}.
By the choice of $B$ asserting that $B_u = \GL_2(\mathcal{O}_{F,u})$ is a group we see that 
\begin{align*}
\det(B_uB_u^{-1}) = \mathcal{O}_{F,u}^\times = \Nr(U_u).
\end{align*}
If $\Nr(\Delta_{\mathcal{O},u})$ is not a unit in $\mathcal{O}_{F,u}$,
the image of the norm map has index two in $\mathcal{O}_{F,u}^\times$ by Lemma~\ref{lem:normimage}.
Furthermore, Lemma~\ref{lem:normimage} implies that for dyadic places the image of the norm map has index $\ll_n 1$.
%
%
The analysis for the Archimedean places is largely analogous; both the set $\det(B_uB_u^{-1})$ and the norm image are comparable to balls of radius $r_u^2$.

Combining these statements yields that $\det(BB^{-1})\cap  \Nr(\bbA_K^\times)^1$ is covered by $\ll_B 2^{b}$ shifts of $\Nr(U)$ where $b$ is the number of non-Archimedean non-dyadic places $u$ of $F$ for which $\Nr(\Delta_{\mathcal{O},u})$ is a non-unit.
It is easy that $2^b$ is bounded by the divisor function of $\disc_{\fin}(Y)$ and in particular $2^b \ll_\varepsilon \disc(Y)^\varepsilon$.
This proves the lemma.
%
\end{proof}

The product decomposition states that $I_{\gamma,s } = \prod_{v \in \mathcal{V}_F} I_{\gamma,s,v}$ where
\begin{align*}
I_{\gamma,s,v} = \int_{\bfT^1(F_v)}\int_{\lquot{\bfZ(F_v)}{\bfT(F_v)}} 1_{g_vB_vg_v^{-1}}(x_v^{-1}\gamma x_v y_v s_v) \de x_v \de y_v;
\end{align*}
we estimate these local orbital integrals.

\begin{lem}[Non-Archimedean local integrals]\label{lem:locint bound I}
Suppose that $Y$ has maximal type.
Let $v$ be a non-Archimedean place of $F$ and let $s \in \lquot{\bfT^1(\bbA_F)}{\bfT(\bbA_F)^1}$ as in \eqref{eq:restrictiontodet}.
Whenever $I_{\gamma,s,v} \neq 0$ we have that $I_{\gamma,s,v} \leq 1$ if $K_v$ is a field and otherwise 
\begin{align*}
I_{\gamma,s,v} \leq (\log_{q_v}(|\inv{\bfT}(\gamma)|_v^{-1}) +1)(\log_{q_v}(|1+\inv{\bfT}(\gamma)|_v^{-1}) +1).
\end{align*}
where $q_v$ is the cardinality of the residue field of $F_v$.
\end{lem}

We remark here that $\inv{\bfT}(\gamma)\neq 0,-1$ as $\gamma \in \GL_2(F)\setminus \bfN_{\bfT}(F)$ so that the right-hand side in the lemma is well-defined (it is moreover positive).

\begin{proof}
Write $\bfH_v = g_v^{-1}\bfT_v g_v$ and $\bfH_v^1 = g_v^{-1}\bfT_v^1 g_v$.
By substitution
\begin{align*}
I_{\gamma,s,v} = \int_{\bfH_v^1(F_v)}\int_{\lquot{\bfZ(F_v)}{\bfH_v(F_v)}} 1_{B_v}(x_v^{-1}g_v^{-1}\gamma g_v x_v y_v s_v') \de x_v \de y_v;
\end{align*}
where $s_v' = g_v^{-1} s_v g_v$.

We rephrase this integral in terms of the local coordinates in Proposition~\ref{prop:localconj}.
Let $x_v,y_v$ as above and write $(t_1,0)$ and $(t_2,0)$ for the respective local coordinates.
Furthermore, write $(a,b)$ for the local coordinates of $g_v^{-1}\gamma g_v$ and $(c,0)$ for the local coordinates of $s_v'$.
In particular, the local coordinates of $x_v^{-1}g_v^{-1}\gamma g_v x_v y_v s_v'$ are $(a t_2 c,b \sigma(t_2c) \frac{\sigma(t_1)}{t_1})$.
By \Cref{prop:localconj}, the requirement $x_v^{-1}g_v^{-1}\gamma g_v x_v y_v s_v' \in B_v$ therefore implies that $at_2c,b \sigma(t_2c) \frac{\sigma(t_1)}{t_1} \in \frac{1}{\Delta_{K/F,v}}\mathcal{O}_{K,v} =: \mathfrak{d}_v^{-1}$ (the latter is the inverse different).
This shows that
\begin{align}\label{eq:Igammasv in loccoord}
I_{\gamma,s,v}
\leq \int_{\SL_1(K_v)}  1_{\mathfrak{d}_v^{-1}}(at_2c) 
\int_{\lquot{F_v^\times}{K_v^\times}} 1_{\mathfrak{d}_v^{-1}}\left(b \sigma(t_2c) \frac{\sigma(t_1)}{t_1}\right) \de t_1 \de t_2,
\end{align}
where $\SL_1(K_v)$ denote the group of norm 1 elements in $K_v$. 
By Hilbert's Theorem 90, the homomorphism 
\begin{align*}
t_1 \in \lquot{F_v^\times}{K_v^\times} \mapsto \frac{\sigma(t_1)}{t_1} \in \SL_1(K_v)
\end{align*}
is a bijection.
The Haar measure on $\SL_1(K_v)=\lquot{F_v^\times}{K_v^\times}$ is normalized so that $\mathcal{O}_{K,v}^\times$ has measure 1.
Therefore, the inner integral in the above expression is equal to 
\begin{align}\label{eq:intovernormone}
\int_{\SL_1(K_v)} 1_{\mathfrak{d}_v^{-1}}(b \sigma(c) t_1) \de t_1
\end{align}
after a substitution, and $I_{\gamma,s,v}$ is bounded by \eqref{eq:intovernormone} times the analogous expression with $a$ instead of $b$.

We thus estimate \eqref{eq:intovernormone}.
Suppose first that $K_v$ is a field. In this case, the group of norm one elements $\SL_1(K_v)$ of $K_v$ is exactly the group of norm one units in $\mathcal{O}_{K,v}^\times$.
By the choice of measure normalization, \eqref{eq:intovernormone} is thus $1$ if $b \sigma(c) \in \mathfrak{d}_v^{-1}$ and zero otherwise.
So suppose that $K_v \simeq F_v \times F_v$. 
Making this identification, $\SL_1(K_v) = \{(\rho,\rho^{-1}):\rho \in F_v^\times$ and $\mathcal{O}_{K,v} = \mathcal{O}_{F,v} \times \mathcal{O}_{F,v}$.
Also, $\Delta_{K/F,v} \in \mathcal{O}_{K,v}^\times$ so that \eqref{eq:intovernormone} is equal to
\begin{align*}
\int_{F_v^\times} 1_{\mathcal{O}_{F,v}}(b_1c_2\rho) 1_{\mathcal{O}_{F,v}}(b_2c_1\rho^{-1}) \de \rho
\end{align*}
where $b= (b_1,b_2)$ and $c = (c_1,c_2)$.
Note that any $\rho \in F_v^\times$ is in the support of the above integrand if and only if
\begin{align}\label{eq:bounds for r}
|b_2c_1|_v \leq |\rho|_v \leq |b_1c_2|_v^{-1}.
\end{align}
The measure of such $\rho$ is (when non-zero)
\begin{align*}
\log_{q_v}(|b_1c_2|_v^{-1})-\log_{q_v}(|b_2c_1|_v) + 1
= \log_{q_v}(|\Nr(bc)|_v^{-1}) + 1
= \log_{q_v}(|\inv{\bfT}(\gamma)|_v^{-1}) +1
\end{align*}
using \eqref{eq:restrictiontodet}. 
Proceeding analogously for $a$ instead of $b$ proves the lemma.
\end{proof}


\begin{lem}[Archimedean local integrals]\label{lem:locint bound II}
Let $v$ be any Archimedean place of $F$ and let $s \in \lquot{\bfT^1(\bbA_F)}{\bfT(\bbA_F)^1}$ as in \eqref{eq:restrictiontodet}.
Whenever $I_{\gamma,s,v}\neq0$, we have that $I_{\gamma,s,v} \leq 1$ if $K_v$ is a field (i.e.~if $F_v=\R$ and $K_v=\C$) and otherwise
\begin{align*}
I_{\gamma,s,v}\ll 
\big(\log(\abs{\inv{\bfT}(\gamma)}_v^{-1}) +\log(r_v\disc_v(Y))\big)
\big(\log(\abs{1+\inv{\bfT}(\gamma)}_v^{-1}) + \log(r_v\disc_v(Y))\big).
\end{align*}
\end{lem}

\begin{proof}
The proof is largely analogous to the proof of \Cref{lem:locint bound I} so we will be brief.
Using the local coordinates of $\gamma$ one sees from Proposition~\ref{prop:localconj} that it is sufficient to estimate an integral of the form
\begin{align}\label{eq:archintegral}
\int_{\SL_1(K_v)} 1_{\Omega_v}(a t_1) \de t_1
\end{align}
as in \eqref{eq:intovernormone} for some $a\in K_v$ where
\begin{equation*}
\Omega_v:=\{ x \in K_v: |x_w|_w \leq R_v \text{ for all } w\mid v \}
\end{equation*}
and $R_v = \kappa r_v \disc_v(Y)^{\frac{1}{2}}$ for some $\kappa>0$.
When $K_v$ is a field, \eqref{eq:archintegral} is $0$ or $1$ depending on the absolute value of $a$.
When $F_v=\R$ and $K_v=\R\times\R$, the set of $\rho\in\R^\times$ satisfying 
	\begin{equation*}
		R_v^{-1}|a_2|_v \leq |\rho|_v \leq R_v|a_1|_v^{-1}
	\end{equation*}
	has measure (if non-empty)
	\begin{equation*}
		2(\log(\abs{a}_v^{-1}) +2\log R_v).
	\end{equation*}
	When $F_v=\C$ and $K_v=\C\times\C$, the set of $\rho\in\C^\times$ satisfying 
	\begin{equation*}
	R_v^{-1}|b_2c_1|_v \leq |\rho|_v \leq R_v|b_1c_2|_v^{-1}
	\end{equation*}
	has measure (if non-empty)
	\begin{equation*}
	2\pi(\log(\abs{a}_v^{-1}) +2\log R_v).
	\end{equation*}
	The rest of the proof proceeds similarly.
\end{proof}

\bigskip

\begin{proof}[Proof of \Cref{prop:torus integral estimate}]
Consider first a fixed $s \in \Nr(\bbA_K^\times)^1$.
Define the set of places
\begin{equation*}
	\calA:=\{ v\in\mathcal{V}_{F,\mathrm{f}}: K_v \text{ is not a field} \}.
\end{equation*}
By Lemma~\ref{lem:locint bound I} we have $I_{\gamma,s,v} \leq 1$ for any finite place $v \not\in \calA$ so that
\begin{align*}
I_{\gamma,s } =  \prod_{v \in \mathcal{V}_F} I_{\gamma,s,v} = \prod_{v\in \mathcal{V}_{F,\mathrm{f}}}I_{\gamma,s,v}\prod_{v\mid\infty}I_{\gamma,s,v}
\leq \prod_{v \in\calA}I_{\gamma,s,v}\prod_{v\mid\infty}I_{\gamma,s,v}
\end{align*}
To estimate $\prod_{v \in\calA}I_{\gamma,s,v}$, we apply the local integral bounds in Lemma~\ref{lem:locint bound I}
\begin{align*}
\prod_{v \in\calA}I_{\gamma,s,v}
&\leq \prod_{v \in\calA}(\log_{q_v}(|\inv{\bfT}(\gamma)|_v^{-1}) +1)(\log_{q_v}(|1+\inv{\bfT}(\gamma)|_v^{-1}) +1)\\
&\ll_\varepsilon\prod_{v \in\calA}(|\inv{\bfT}(\gamma)|_v^{-1})^{\varepsilon/4}(|1+\inv{\bfT}(\gamma)|_v^{-1})^{\varepsilon/4}\\
&=\prod_{v\in\mathcal{V}_F\setminus\calA}(|\inv{\bfT}(\gamma)|_v)^{\varepsilon/4}(|1+\inv{\bfT}(\gamma)|_v)^{\varepsilon/4}
\ll_{n}r^{n\varepsilon}\,\disc({Y})^{\varepsilon/2}
\end{align*}
where we also used the local bounds \eqref{eq:denomboundsinv} for the value $\inv{\bfT}(\gamma)$ and the analogous bounds for $1+\inv{\bfT}(\gamma)$ in the last step.

For the product $\prod_{v\mid\infty}I_{\gamma,s,v}$ over the Archimedean places, we proceed analogously using Lemma~\ref{lem:locint bound II} to obtain 
\begin{align*}
\prod_{v\mid\infty}I_{\gamma,s,v}
\ll_\varepsilon 
r^{n\varepsilon}\,\disc({Y})^{\varepsilon/2}.
\end{align*}
Overall, we have shown that
\begin{align*}
I_{\gamma,s } \ll_{\varepsilon} r^{2n\varepsilon}\disc(Y)^\varepsilon.
\end{align*}
Combined with \Cref{lem:norm measure}, this implies that
\begin{equation*}
\begin{split}
	I_\gamma&=\int_{s\in\Nr(\bbA_K)^1\cap \frac{1}{\det(\gamma)}\det(B)}I_{\gamma,s}\,\de s 
	\ll_{n,\varepsilon} r^{2n\varepsilon}\disc(Y)^{2\varepsilon}.
\end{split}
\end{equation*}
This completes the proof of the proposition.
%
\end{proof}

We are now ready to prove Linnik's basic lemma as in Theorem~\ref{thm:linnikbasiclemma}.

\begin{proof}[Proof of \Cref{thm:linnikbasiclemma}]
The proof merely consists of putting together already proven statements.
	By \eqref{eq:bounded by kernel integral} and \eqref{eq:geomexpansion}, we have
	\begin{equation*}
		\begin{split}
		&\mu_Y \times \mu_Y\big(\{(x,y) \in [\GL_{2,F}]^2:y \in x B_\tau\}\big)\\
		&\leq \sum_{\gamma \in \lrquot{\bfT(F)}{\GL_2(F)}{\bfT(F)}} \int_{[\bfT]^2} \sum_{\eta \in \bfT(F)\gamma\bfT(F)} f_\tau(g^{-1}x^{-1}\eta y g) \de\mu_{[\bfT]}(x) \de\mu_{[\bfT]}(y).
		\end{split}
	\end{equation*}
	Recall that \eqref{eq:normalizer contribution} gives
	\begin{equation}\label{eq:temp_normalizer contribution}
		\sum_{\gamma \in \lrquot{\bfT(F)}{\bfN_{\bfT}(F)}{\bfT(F)}} \int_{[\bfT]^2} \sum_{\eta \in \bfT(F)\gamma\bfT(F)} f_\tau(g^{-1}x^{-1}\eta y g) \de\mu_{[\bfT]}(x) \de\mu_{[\bfT]}(y)
		\ll \frac{1}{\vol(Y)}.
	\end{equation}
	Combining \eqref{eq:geomexpansion2}, \Cref{lem:counting lattice points} and \Cref{prop:torus integral estimate}, we get
	\begin{equation}\label{eq:temp_generic contribution}
	\begin{split}
	&\sum_{\substack{\gamma \in \lrquot{\bfT(F)}{\GL_2(F)}{\bfT(F)}\\ \gamma \not\in \bfN_\bfT(F) }} \int_{[\bfT]^2} \sum_{\eta \in \bfT(F)\gamma\bfT(F)} f_\tau(g^{-1}x^{-1}\eta y g) \de \mu_{[\bfT]}(x)\de\mu_{[\bfT]}(y) \\
	&\qquad=\frac{\vol\left([\bfZ]\right)}{\vol(Y)^2}\sum_{\gamma}
	\int_{\lquot{\Delta(\bfZ)(\bbA_F)^1}{(\bfT(\bbA_F)^1)^2}} f_\tau(g^{-1}x^{-1}\gamma yg) \de x\de y\\
	&\qquad\ll_{n,\varepsilon}\frac{\vol\left([\bfZ]\right)}{\vol(Y)^2}r^{n(1+\varepsilon)}\frac{\disc(Y)^{1+\varepsilon}}{\absdisc_F^{1/2}} |\phi(a)|_S^{-2\tau}\\
	&\qquad\overset{\eqref{eq:volume of Gm over F}}{\ll_{\varepsilon}} \frac{\absdisc_F^{1/2+\varepsilon}}{\vol(Y)^2}r^{n(1+\varepsilon)}\frac{\disc(Y)^{1+\varepsilon}}{\absdisc_F^{1/2}} |\phi(a)|_S^{-2\tau}\\
	&\qquad\leq r^{n(1+\varepsilon)}\frac{\disc(Y)^{1+2\varepsilon}}{\vol(Y)^2} |\phi(a)|_S^{-2\tau}.
	\end{split}
	\end{equation}
To complete the proof, combine \eqref{eq:temp_normalizer contribution} and \eqref{eq:temp_generic contribution} and observe that the entropy $h_{[\GL_{2,F}]}(a)$ is exactly $\log(|\phi(a)|_S)$.
\end{proof}

\newpage

\section{Reduction to type $A_1$}
In this section we proceed to prove Theorem \ref{thm:main}. Quite generally, a lower bound on the entropy can be obtained by showing sufficient decay of the measure of Bowen balls at typical points.
We do so in the following two steps:
\begin{itemize}
\item[(A)] Find a time at which all points in a Bowen ball lie on the orbit of an intermediate group of the form $\Res_{F/\Q}(\GL_2)$.
\item[(B)] Apply Linnik's uniform basic lemma \Cref{thm:linnikbasiclemma} at this time scale.
\end{itemize}

Let us begin by recalling how to obtain an entropy bound from separation.
For any neighborhood of the identity $B \subset \GL_4(\bbA)^1$, any semisimple element $a\in \bfG(\Q_u)$, and any $\tau \geq 0$ we define the Bowen ball
\begin{align*}
B_\tau = \bigcap_{-\tau \leq t \leq \tau} a^{\tau}B a^{-\tau}.
\end{align*}

\begin{prop}[{cf.~\cite[Prop.~3.2]{ELMV09Duke} and \cite[Prop.~8.2]{Kha15}}]\label{prop:Bowentoentropy}
	Fix a semisimple element $a\in \bfG(\Q_u)$ for some place $u$ of $\Q$. Suppose that $\{\mu_i \}$ is a sequence of $a$-invariant probability measures on $\lquot{\bfG(\Q)}{\bfG(\bbA)^1}$ converging to a probability measure $\mu$ in the weak-* topology. Assume that for some fixed $\eta>0$ we have a sequence of integers $\tau_i\to\infty$ such that for any compact subset $\calF\subset\lquot{\bfG(\Q)}{\bfG(\bbA)^1}$ there exists a bounded identity neighborhood $B\subset \bfG(\bbA)$ such that
	\begin{align*}
	\mu_i\times\mu_i\left\{ (x,y)\in \calF\times \calF \mid y\in xB_\tau \right\}
	&=\int_\calF\mu_i\left(xB_\tau\cap \calF \right) d\mu_i(x) \\
	&\ll_{\calF,\varepsilon} \exp(-2(\eta-\varepsilon)\tau_i),
	\end{align*}
	then the Kolmogorov-Sinai entropy of the $a$-action with respect to the measure $\mu$ satisfies $h_\mu(a)\geq \eta$.
\end{prop}

Following the notation of Theorem~\ref{thm:main} we fix a place $u$ of $\Q$ and a split $\Q_u$-torus $\bfA < \GL_4$.
Without loss of generality, we may assume that $\bfA$ is the diagonal subgroup.
Furthermore, we set $A = \bfA(\Q_u)$ and fix some element $a \in A$ (more restrictions will be imposed later on). The Bowen balls $B_\tau$ are defined with respect to $a$.
Moreover, $B$ is taken to be of the form $B = \prod_v B_{v}$ where $B_v = \GL_4(\Z_v)$ if $v$ is finite and 
\begin{align*}
B_\infty = \{g \in \GL_4(\R): \norm{g},\norm{g^{-1}}\leq 2\}.
\end{align*}

In view of Proposition~\ref{prop:Bowentoentropy} and Theorem~\ref{thm:main} we shall fix in the following a toral packet $Y=[\bfT g]$ of either biquadratic, cyclic or dihedral type and of maximal type satisfying the properties of the main theorem and exhibit a time $\tau = \tau(\disc(Y))$ at which
\begin{align*}
\mu_{Y}\times\mu_{Y}\left\{ (x,y)\in \calF\times \calF \mid y\in xB_\tau \right\}
\ll \exp(-2\eta\tau)
\end{align*}
for some specific $\eta>0$.
Let $K$ be the number field associated to the torus $\bfT$.
Explicitly, $K$ is the centralizer of $\bfT(\Q)$ in $\bfM_4(\Q)$.
Let $L$ be the Galois closure of $K$ (in our case $[L:K] \leq 2$) and set $\mathcal{G} = \Gal(L/\Q)$.
Fix a quadratic subfield $F$ of $K$.
We write $\intgroup \simeq \Res_{F/\Q}(\GL_2)$ for the intermediate group that $F$ defines; it is the centralizer of the subgroup $\bfS<\bfT$ isomorphic to $\Res_{F/\Q}(\mathbb{G}_{m,F})$ (cf.~\S\ref{sec:inthomog}).
In view of Theorem~\ref{thm:main} we also assume that the points in $g_u^{-1} \bfS(\Q_u) g_u$ are of the form $\diag(\lambda_1,\lambda_1,\lambda_2,\lambda_2)$ for $\lambda_1,\lambda_2 \in \Q_u^\times$.
We write $\intgroup_{\mathrm{std}} < \GL_4$ for the block-diagonal subgroup so that in particular,  $g_u^{-1} \intgroup(\Q_u) g_u = \intgroup_{\mathrm{std}}(\Q_u)$.

\subsection{Geometric invariant theory}
We briefly recall the way geometric invariant theory (GIT) was used in \cite{Kha15} and refer to Sections~3,4 therein for more details.
Define the affine variety
\begin{align*}
\lrdoublequot{\bfT}{\GL_4}{\bfT} =
\Spec \Q[\GL_4]^{\bfT\times\bfT}
\end{align*}
which is a universal categorical quotient (see \cite{MFK94}).
For $\sigma \in S_4$ consider the rational function $\Psi_\sigma^0$ on $\GL_4$ given by 
\[
\Psi_\sigma^0(g)=(\det g)^{-1}\mathrm{sign}(\sigma)\prod_{1\leq i\leq n}g_{\sigma(i),i}.
\]
These rational functions form a generating set of the $\Q$-algebra $\Spec \Q[\GL_4]^{\bfA\times\bfA}$ when $\bfA < \GL_4$ is the diagonal torus (cf.~\cite[Prop.~4.1]{Kha15}).
Let $\cspl \in \GL_4(L)$ be so that $\cspl^{-1}\bfT\cspl = \bfA$.
By universality, the $\bfT\times\bfT$-invariant regular functions
\begin{align*}
\Psi_\sigma(g) = \Psi_\sigma^0(\cspl^{-1} g \cspl)
\end{align*}
for $\sigma \in S_4$ form a generating set for $\Spec L[\GL_4]^{\bfT\times\bfT}$. We shall call these the \emph{canonical generators}. They are not typically defined over $\Q$ even when $\mathcal{G}$ is abelian as will be apparent from the discussion to follow.

Let $W$ be the Weyl group $\bfN_{\bfA}/\bfA$. 
For any $w\in W$, $w\,\diag(t_i)w^{-1}=\diag(t_{\sigma^{-1}(i)})$ for some $\sigma\in S_4$. Then $W$ is identified with $S_4$ via $w\mapsto\sigma$. We have an injective homomorphism from the Galois group $\mathcal{G}$ to the Weyl group $W\cong S_4$ via the 1-cocycle $\sigma\mapsto\sigma(\cspl)^{-1}\cspl$.
We identify $\mathcal{G}$ with its image under this homomorphism.
By \cite[Proposition 6.6]{Kha15}, for any $g\in\bfG(\Q)$, $\sigma \in S_4$ and $\tau\in\mathcal{G}$,
\begin{align}\label{eq:Galoisrel}
\tau.\Psi_{\sigma}(g)=\Psi_{\tau\sigma\tau^{-1}}(g).
\end{align}
In particular, $\Psi_{\sigma}$ is defined over $\Q$ if and only if $\sigma$ is centralized by $\mathcal{G}$.

\subsection{Galois relations}
We shall need a more precise description of this Galois relations in \eqref{eq:Galoisrel} and in particular of the image of $\mathcal{G}$ in $S_4$.
Let us fix an ordering on the Galois embeddings
\begin{align}\label{eq:comp wrt F}
\sigma_i: K \hookrightarrow L,\ i=1,\ldots,4.
\end{align}
compatible with $F$ in the sense that $\sigma_1|_F=\sigma_2|_F$ and $\sigma_3|_F=\sigma_4|_F$.
The Galois action on these embeddings determines $\mathcal{G} \hookrightarrow S_4$.

\begin{lem}[Explicit conjugation]\label{lem:expl conj}
There exists a basis $x_1,\ldots,x_4$ of $K$ so that the matrix $\cspl = (\sigma_j(x_i))_{ij}$ satisfies  $\cspl^{-1}\bfT\cspl = \bfA$ and for any $x \in K \subset \bfM_4(\Q)$
\begin{align*}
\cspl^{-1}x\cspl = \diag(\sigma_1(x),\ldots,\sigma_4(x)).
\end{align*}
\end{lem}

\begin{proof}
This is completely standard; we give a proof merely for completeness.
To construct the basis, consider the $\Q$-linear map
\begin{align*}
\theta: x \in K \mapsto e_1 x \in \Q^4
\end{align*}
where we view $K \subset \bfM_4(\Q)$. This is an (right-)equivariant isomorphism. By definition, the matrix representation of multiplication by $x \in K$ in the basis $x_i = \theta^{-1}(e_i)$ is exactly $x$ when viewed as an element of $\bfM_4(\Q)$.

The matrix $\cspl = (\sigma_j(x_i))_{ij}$ is the representation in the standard basis of the $\Q$-linear map $\varphi \circ\theta^{-1}$ where
\begin{align*}
\varphi: x \in K \mapsto (\sigma_1(x),\ldots,\sigma_4(x)).
\end{align*}
From here, the lemma is straightforward to verify.
\end{proof}

\begin{rem}
Together with the compatibility assumption in \eqref{eq:comp wrt F}, Lemma~\ref{lem:expl conj} also asserts that
\begin{align*}
\cspl^{-1}\bfS(\Q)\cspl
= \{\diag(\sigma_1(x),\sigma_1(x),\sigma_3(x),\sigma_3(x)):x \in F^\times\}
\end{align*}
and similarly for $\bfS(R)$ where $R$ is any $\Q$-algebra. In particular, $\cspl^{-1}\intgroup(\Q)\cspl$ is block-diagonal.
\end{rem}

We shall now use the constructed element $\cspl$ in Lemma~\ref{lem:expl conj} to explicitly determine the induced homomorphism $\mathcal{G} \to S_4$.
To this end, observe that the $1$-cocycle introduced earlier is equal to (under the identification $W \simeq S_4$)
\begin{align}\label{eq:embed into S4}
\tau \in \mathcal{G} \mapsto (j \mapsto k \text{ where } \sigma_k = \tau \circ \sigma_j) \in S_4. 
\end{align}
In particular, it is injective.

In the following we determine the image of $\mathcal{G}$ according to each Galois type.
It is independent of the choice in \eqref{eq:comp wrt F} while of course the image of an individual element of $\mathcal{G}$ can depend on the ordering.
\begin{itemize}
\item $K$ is \emph{biquadratic}. The image of \eqref{eq:embed into S4} is
\begin{align*}
\{\id, (1\,2)(3\,4),(1\,3)(2\,4),(1\,4)(2\,3)\}.
\end{align*}
The image of the subgroup $\Gal(K/F)$ is $\{\id, (1\,2)(3\,4)\}$.
\item $K$ is \emph{cyclic}.
 The image of \eqref{eq:embed into S4} is
\begin{align*}
\{\id, (1\,3\,2\,4),(1\,2)(3\,4),(1\,4\,2\,3)\} \simeq \Z/4\Z
\end{align*}
and the image of $\Gal(K/F)$ is $\{\id, (1\,2)(3\,4)\}$.
\item $K$ is \emph{dihedral}. The image of \eqref{eq:embed into S4} is
\begin{align*}
\langle (1\,3\,2\,4), (3\,4)\rangle \simeq D_4.
\end{align*}
Under identification with the image, $K$ is the fixed field of the transposition $(3\,4)$ and $F$ is the fixed field of the normal subgroup $\langle (1\,2),(3\,4)\rangle$.
\end{itemize}
From now on we shall identify $\mathcal{G}$ with its image under \eqref{eq:embed into S4}.

\begin{lem}\label{lem:Galoisrel entries}
For any $\tau \in \mathcal{G}$, any $g\in \GL_4(\Q)$, and any index $(i,j)$ the conjugated element $g' =\cspl^{-1}g\cspl$ satisfies
\begin{align*}
\tau(g'_{ij}) = g'_{\tau(i)\tau(j)}.
\end{align*}
\end{lem}

\begin{proof}
Let $w_\tau$ denote the image of $\tau$ in the Weyl group $W$. Then $w_\tau$ has a representative in the form of a permutation matrix $\widetilde{w}_\tau$, such that $\tau(\cspl)=\cspl\widetilde{w}_\tau$. Then we have
\begin{equation*}
	\tau(g')=\tau(\cspl^{-1}g\cspl)=\widetilde{w}_\tau^{-1}\cspl^{-1}g\cspl\widetilde{w}_\tau=\widetilde{w}_\tau^{-1}g'\widetilde{w}_\tau.
\end{equation*}
The lemma follows.
\end{proof}

Let us illustrate Lemma~\ref{lem:Galoisrel entries} for each Galois type. Let $g\in \GL_4(\Q)$ and $g' =\cspl^{-1}g\cspl$.
\begin{itemize}
\item For $K$ biquadratic
\begin{align*}
g' = \begin{pmatrix}
\ast_1 & \ast_2 & \ast_3 & \ast_4\\
\ast_2 &\ast_1 &\ast_4 & \ast_3\\
\ast_3 &\ast_4 &\ast_1 &\ast_2\\
\ast_4 &\ast_3 &\ast_2 & \ast_1
\end{pmatrix}
\end{align*}
where for each $i$ the entries with $\ast_i$ are all Galois conjugate to each other.
\item For $K$ cyclic
\begin{align*}
g' = \begin{pmatrix}
\ast_1 & \ast_2 & \ast_3 & \ast_4\\
\ast_2 &\ast_1 &\ast_4 & \ast_3\\
\ast_4 &\ast_3 &\ast_1 &\ast_2\\
\ast_3 &\ast_4 &\ast_2 & \ast_1
\end{pmatrix}.
\end{align*}
\item For $K$ dihedral
\begin{align*}
g' = \begin{pmatrix}
\ast_1 & \ast_2 & \ast_3 & \ast_3\\
\ast_2 &\ast_1 &\ast_3 & \ast_3\\
\ast_3 &\ast_3 &\ast_1 &\ast_2\\
\ast_3 &\ast_3 &\ast_2 & \ast_1
\end{pmatrix}.
\end{align*}
\end{itemize}
In view of the above we define the following special set of permutations
\begin{align*}
\perm = 
\begin{cases}
\{(1\,4)(2\,3),(1\,3)(2\,4)\} & \text{if $K$ is biquadratic},\\
\{(1\,3\,2\,4),(1\,4\,2\,3)\} & \text{if $K$ is cyclic or dihedral}.
\end{cases}
\end{align*}

\begin{lem}\label{lem:partialvanishing}
Let $g\in \GL_4(\Q)$ be such that $\Psi_\sigma(g) = 0$ for every $\sigma \in \perm$. Then $g \in \intgroup(\Q)$.
\end{lem}

\begin{proof}
Let $g'= \cspl^{-1}g\cspl$. By the Galois conjugacy relations $\Psi_\sigma(g) = 0$ for $\sigma$ as in the lemma asserts that all entries $g'_{ij}$ of $g'$ away from the block-diagonal (i.e.~for $(i,j)$ such that $i\neq j$ and $4\leq i+j\leq 6$) vanish.
In particular, $g'$ commutes with $\cspl^{-1}\bfS(\Q)\cspl$ and hence $g$ commutes with $\bfS(\Q)$. This implies $g \in \intgroup(\Q)$.
\end{proof}

\subsection{Proof of Theorem~\ref{thm:main}}
Following the outline, we now show that given enough contraction, any return to a Bowen ball around a point in the homogeneous toral set must occur within the intermediate orbit.
As before, we shall continue using the notation introduced at the beginning of this section.
Let $a \in A$ and define the Bowen ball $B_\tau$ with respect to $a$.
We denote by
\begin{align*}
\phi_{ij}(a) = \frac{a_{ii}}{a_{jj}}
\end{align*}
the value of the roots at $a$. We identify the tuple $(i,j)$ with its associated root $\phi_{ij}$. To simplify the notation, we set for $\sigma\in S_4$
\begin{align*}
R_\sigma &= \{(\sigma(i),i): 1\leq i \leq 4,\ i \neq \sigma(i)\},\\
\eta_\sigma(a) &= \sum_{\phi \in R_{\sigma}} |\log|\phi(a)|_p|.
\end{align*}

\begin{prop}[Returns within the intermediate orbit]\label{prop:firstcontraction}
There exists $\kappa>0$ with the following property.
Suppose that
\begin{align*}
\tau \decayone(a) > \tfrac{1}{2}\log\disc(Y) + \kappa
\end{align*}
where
\begin{align*}
\decayone(a) = \min_{\sigma \in \perm} \eta_\sigma(a).
\end{align*}
Then any $\gamma \in \GL_4(\Q) \cap \bfT(\bbA) gB_\tau g^{-1} \bfT(\bbA)$ lies in $\intgroup(\Q)$.
\end{prop}

\begin{proof}
We apply the proof of \cite[Thm.~8.9]{Kha15} based on geometric invariants and their local bounds. 

Suppose first that $K/\Q$ is abelian and let $\sigma \in \perm$.
Hence, $\Psi_\sigma$ is defined over $\Q$. The proof of \cite[Thm.~8.9]{Kha15} shows that
\begin{align*}
|\Psi_\sigma(\gamma)|_{\bbA} \leq C \exp(-2\tau \eta_\sigma(a))\disc(Y)
\end{align*}
for some absolute constant $C$.
Therefore, if
\begin{align*}
\tau \eta_\sigma(a) > \tfrac{1}{2} \log\disc(Y)+ \tfrac{1}{2}\log(C)
\end{align*}
we have the content bound $|\Psi_\sigma(\gamma)|_{\bbA} < 1$ so that $\Psi_\sigma(\gamma) = 0$ by the product formula for $\Q$. 
In particular, if $\tau \min_{\sigma \in \perm} \eta_\sigma(a) > \tfrac{1}{2} \log\disc(Y)+ \tfrac{1}{2}\log(C)$, Lemma~\ref{lem:partialvanishing} implies that $\gamma \in \intgroup(\Q)$ as claimed.

Suppose now that $K/\Q$ is dihedral. Then $\perm$ is a single Galois conjugacy class (under the identification of the Galois group $\mathcal{G} \subset S_4$). 
Applying again the proof of \cite[Thm.~8.9]{Kha15} we obtain
\begin{align*}
\prod_{\sigma \in \perm}|\Psi_\sigma(\gamma)|_{\bbA} \leq C^2 \exp\Big(-2\tau \sum_{\sigma \in \perm}\eta_\sigma(a)\Big)\disc(Y)^2.
\end{align*}
By a similar argument as in the abelian case, the condition
\begin{align*}
\tau \tfrac{1}{2}\sum_{\sigma \in \perm}\eta_\sigma(a) > \tfrac{1}{2} \log\disc(Y)+ \tfrac{1}{2}\log(C)
\end{align*}
asserts that $\Psi_\sigma(\gamma) = 0$ for some $\sigma \in \perm$.
As $\perm$ is a single conjugacy class, this implies $\Psi_\sigma(\gamma) = 0$ for both elements $\sigma \in \perm$ and by Lemma~\ref{lem:partialvanishing} we have $\gamma \in \intgroup(\Q)$.
Note that $\eta_{\sigma}(a)$ is constant for $\sigma\in\perm$ so that the above average over $\perm$ is equal to the minimum.
This proves the proposition.
\end{proof}

The following should be considered to be one of the main theorems of this article.
In particular, it encompasses Theorem~\ref{thm:main}.
For the readers' convenience, we shall repeat all standing assumptions.

\begin{thm}\label{thm:base}
Let $u$ be a place of $\Q$, let $Y = [\bfT g] \subset [\GL_4]$ be a homogeneous toral sets satisfying the following assumptions.
\begin{enumerate}[(i)]
\item The homogeneous toral set $Y$ is of maximal type and invariant under the $\Q_u$-points $A$ of the diagonal subgroup $\bfA$ of $\GL_4$.
\item The quartic field $K$ associated to $Y$ is either biquadratic, cyclic or dihedral. Fix a quadratic subfield $F \subset K$ and the associated subgroup $\inttorus \subset \bfT$ with $\inttorus \simeq \Res_{F/\Q}(\mathbb{G}_{m,F})$.
\item The $\Q_u$-torus $g_{u}^{-1} \bfS g_{u}$ commutes with the block-diagonal subgroup $\intgroup_{\mathrm{std}}$.
\end{enumerate}
Choose $a \in \bfA(\Q_u)$ defining the Bowen ball $B_\tau$ for any $\tau >0$.
There exists a constant $\kappa >0$ depending only on $A_{\infty} = g_{\infty}^{-1}\bfT(\R) g_{\infty}$ with the following property.
Whenever $\tau>0$ satisfies
\begin{align}\label{eq:firstcontraction}
\tau \decayone(a) > \tfrac{1}{2}\log\absdisc_K + \kappa
\end{align}
and
\begin{align}\label{eq:tau not too large large to apply counting lemma}
2\tau \hint(a) \leq \log\absdisc_K - 3\log\absdisc_F - \log c,
\end{align}
we have
\begin{align*}
\mu_{Y}\times \mu_{Y}\big(\{ (x,y): y \in x B_\tau\}\big) 
\ll_{A_\infty,\varepsilon,c} 
\absdisc_K^{-\frac{1}{2}+\varepsilon} 
+ \frac{1}{\absdisc_F^2}\absdisc_K^{\varepsilon}\mathrm{e}^{-2\tau \hint(a)}.
\end{align*}
\end{thm}

\begin{proof}
Suppose that $(x_1,x_2) \in Y \times Y$ is a pair of points with $x_2 \in x_1 B_\tau$. Write $x_1 = \GL_4(\Q)t_1g$ and $x_2 = \GL_4(\Q)t_2g$ for $t_1,t_2 \in \bfT(\bbA)^1$. By assumption on the pair, there is $\gamma \in\GL_4(\Q)$ with $\gamma t_1 g \in t_2g B_\tau$.
By Proposition~\ref{prop:firstcontraction} and Assumption \eqref{eq:firstcontraction}, we have $\gamma \in \intgroup(\Q)$. In particular, $t_2 \gamma t_1^{-1} \in \intgroup(\bbA)^1$ and
\begin{align*}
x_2 \in x_1 (B_\tau \cap g^{-1}\intgroup(\bbA)^1g).
\end{align*}
At this point, one has
\begin{align*}
\mu_{Y}\times \mu_{Y}\big(\{ (x,y): y \in x B_\tau\}\big)
= \mu_{Y}\times \mu_{Y}\big(\{ (x,y): y \in x (B_\tau \cap g^{-1}\intgroup(\bbA)^1g)\}\big)
\end{align*}
and would like to apply Theorem~\ref{thm:linnikbasiclemma}. 
If $g \in \intgroup(\bbA)^1$, this is possible; otherwise, one needs to extend Theorem~\ref{thm:linnikbasiclemma} and its proof.
We shall do so in the following proposition.
\end{proof}

\begin{prop}[An extension of Theorem~\ref{thm:linnikbasiclemma}]\label{prop:extensionLinnik}
For $[\bfT g]$ and $[\intgroup g]$ as above we have
\begin{align*}
\mu_{Y}\times \mu_{Y}\big(\{ (x,y): y \in x (B_\tau \cap g^{-1}\intgroup(\bbA)^1g)\}\big)\\
\ll_{A_\infty,\varepsilon} \absdisc_K^{-\frac{1}{2}+\varepsilon} 
+ \frac{1}{\absdisc_F^2}\absdisc_K^{\varepsilon}\mathrm{e}^{-2\tau \hint(a)}.
\end{align*}
\end{prop}

As mentioned, the proof here is largely analogous to the proof of Theorem~\ref{thm:linnikbasiclemma}.
In fact, upon closer inspection the proof of Theorem~\ref{thm:linnikbasiclemma} relies chiefly on local computations; we exhibit here the necessary local coordinates in two steps.
The following proposition is phrased quite generally (more generally than needed). We recall that for any rational prime $p$, $F_p=F\otimes_\Q\Q_p=\prod_{v\mid p}F_v$.

\begin{prop}[Non-Archimedean Block-coordinates]\label{prop:non-Arch block-coordinates}
Let $[\bfT g]$ be a homogeneous toral set in $[\GL_4]$ of maximal type and let $\bfS$ be the subtorus for a subfield $F$ of the associated quartic field $K$.
Let $p$ be a rational prime.
Then there exists $c_{1,Y,p} \in \GL_4(F_p)$ with the following properties
\begin{enumerate}
\item  
$c_{1,Y,p}^{-1}\in \bfM_2(\mathcal{O}_{F,p})$ and $\Delta_{F,u}c_{1,Y,p} \in \bfM_2(\mathcal{O}_{F,p}).$
\item For any place $v$ of $F$ above $p$ we have $c_{1,Y,v}g_p^{-1}\intgroup g_p c_{1,Y,v}^{-1} = \intgroup_{\mathrm{std}}$.
\item  For any $\gamma \in \GL_4(\Q_p)$ the conjugate $c_{1,Y,p}\gamma c_{1,Y,p}^{-1}\in \GL_4(F_p)$ is of the form
\begin{align*}
\begin{pmatrix}
A_{1,p} & A_{2,p} \\ \sigma_1(A_{2,p}) & \sigma_1(A_{1,p})
\end{pmatrix}
\end{align*}
where $A_{1,p},A_{2,p}\in \bfM_{2}(F_p)$ and where $\sigma_1$ is the non-trivial Galois automorphism of $F_p/\Q_p$. We call the pairs $(A_{1,p},A_{2,p})$ the local block coordinates (relative to $F$) of $\gamma$. 
\item  The block coordinates $(A_{1,p},A_{2,p})$ of any $k \in \GL_4(\Z_p)$ satisfy 
\begin{itemize}
\item $A_{1,p},A_{2,p}\in \frac{1}{\Delta_{F,p}} \bfM_2(\mathcal{O}_{F,p})$,
\item $A_{1,p}-A_{2,p} \in \bfM_2(\mathcal{O}_{F,p})$.
\end{itemize}
In particular, any $k \in \GL_4(\Z_p) \cap g_p^{-1}\intgroup(\Q_p) g_p$ has local block coordinates of the form $(A_p,0)$ for $A_p \in \GL_2(\mathcal{O}_{F,p})$.
\end{enumerate}
\end{prop}

\begin{proof}
The proof bears many similarities to the proof of \Cref{prop:localconj} so we shall be brief.
Observe that instead of block-diagonalizing $g_p^{-1}\intgroup g_p$ we may as well diagonalize $g_p^{-1}\bfS g_p$ (where the eigenvalues need to be ordered so that the first two agree).

It is sufficient to find a basis of any $\mathcal{O}_{K,p}$-ideal $\mathfrak{a}$ where multiplication by the generator $\alpha$ with $\mathcal{O}_{F,p} = \Z_p[\alpha]$ takes a desirable form. Let $\beta\in K_p$ with $\mathcal{O}_{F,p}[\beta] = \mathcal{O}_{K_p}$ and consider the $\Z_p$-basis
\begin{align*}
\lambda, \lambda \alpha, \lambda \beta, \lambda \beta \alpha
\end{align*}
of $\mathfrak{a}$ where $\lambda \in K_p$ is such that $\mathfrak{a} = \lambda \mathcal{O}_{K,p}$. In this basis, multiplication by $\alpha$ is given by the block-diagonal matrix $M = \diag(M_1,M_1)$ where
\begin{align*}
M_1 = \begin{pmatrix}
0 & 1 \\ -\Nr_{F_p/Q_p}(\alpha) & \Tr_{F_p/Q_p}(\alpha)
\end{pmatrix}.
\end{align*}
Define $c_{1,Y,p} = P_{(2\,3)}\diag(c,c)$ where $P_{(2\,3)}$ is the permutation matrix for the transposition $(2\,3)$ and where
\begin{align*}
c = \begin{pmatrix}
1 & 1 \\ \alpha & \sigma(\alpha)
\end{pmatrix}^{-1}.
\end{align*}

The rest of the proof is completely analogous to the proof of Proposition~\ref{prop:localconj}.
\end{proof}

\subsubsection{Proof of Proposition~\ref{prop:extensionLinnik} and Theorem~\ref{thm:main}}
We now use Proposition~\ref{prop:non-Arch block-coordinates} to define the local coordinates.
For $p$ a rational prime and $\gamma\in g_p^{-1}\intgroup(\Q_p) g_p$ we let $(A_p,0)$ be the block coordinates so that $ A_p \in \GL_2(\mathcal{O}_{F,p})$ if and only if $\gamma \in \GL_4(\Z_p)$.
Note that the group of points $s\in \GL_2(F_p)$ for which $(s,0)$ is the block coordinate of some element of $g_p^{-1}\bfT(\Q_p) g_p$ can be written as  $\prod_{v\in\VF, v\mid p}\bfT'_v(F_v)$ where $\bfT'_v < \GL_2$ is an $F_v$-torus isomorphic to $\Res_{K/F}(\mathbb{G}_{m,K})$.
Next, we apply Proposition~\ref{prop:localconj} for each place $v\in\VF$ with $v\mid p$ to the $F_v$-torus $\bfT'_v$ to obtain the local coordinates $(b_{1,v},b_{2,v})$ of $A_v$.
We call the tuple $(b_{1,p},b_{2,p}) = (b_{1,v},b_{2,v})_{v \mid p}$ the \emph{local coordinates} of $\gamma$ (at $p$).
The local coordinates have properties much like in Proposition~\ref{prop:localconj}; we do not list them here.

\begin{proof}[Sketch of proof of Proposition~\ref{prop:extensionLinnik}]
The proof of Proposition~\ref{prop:localconj} is largely analogous to the proof of Theorem~\ref{thm:linnikbasiclemma}.
In words, one expands the integral over $Y^2$ of the Bowen kernel
\begin{align*}
K_\tau(x,y) = \sum_{\gamma \in \intgroup(\Q)} f_\tau (x^{-1}\gamma y)
\end{align*}
with respect to 
$
\lrquot{\bfT(\Q)}{\intgroup(\Q)}{\bfT(\Q)}
$
and estimates the individual contribution for each point in the above double quotient. These contribution were analysed in \S\ref{sec:orbitintegrals} using local coordinates only. Given the above construction of the local coordinates in the current situation one can simply proceed in analogous fashion.
\end{proof}

\begin{proof}[Proof of Theorem~\ref{thm:main}]
We use notations in Theorem~\ref{thm:main}.
We combine Theorem~\ref{thm:base} with Proposition~\ref{prop:Bowentoentropy}.
Recall that $A' \subset A$ is the set of $a \in A$ with $\hint(a) \leq \frac{1}{3}h_{[\GL_4]}(a)$. 

We claim that $\hint(a) < \frac{1}{3}h_{[\GL_4]}(a)$ if and only if $\eta(a)>2\hint(a)$.
Indeed, we have
\begin{align*}
    h_{[\GL_4]}(a) = \sum_{1\leq i<j\leq 4}|\log|\phi_{ij}(a)|_p|,\\
    \hint(a) = |\log|\phi_{12}(a)|_p| + |\log|\phi_{34}(a)|_p|.
\end{align*}
When $K$ is cyclic or dihedral we have
\begin{align*}
    \eta(a) = \sum_{i=1}^{2}\sum_{j=3}^4 |\log|\phi_{ij}(a)|_p|.
\end{align*}
Hence in these two cases we have $h_{[\GL_4]}(a) = \hint(a) + \eta(a)$, and the equivalence follows. 

When $K$ is biquadratic we have
\begin{align*}
    \eta(a) = 2\min \{|\log|\phi_{14}(a)|_p|+|\log|\phi_{23}(a)|_p|, |\log|\phi_{13}(a)|_p| + |\log|\phi_{24}(a)|_p|\}.    
\end{align*}
Similar to the other two cases, we can easily show that $\eta(a)>2\hint(a)$ implies $\hint(a) < \frac{1}{3}h_{[\GL_4]}(a)$. Conversely, if $\hint(a) < \frac{1}{3}h_{[\GL_4]}(a)$, then it follows that $x_{12}+x_{34}$ is the smallest among $x_{12}+x_{34},x_{13}+x_{24},x_{14}+x_{23}$, where we denote $x_{ij}=|\log|\phi_{ij}(a)|_p|$. This gives a restriction on the configuration of the absolute values of the 4 entries of $a$ on the real line, and it follows that $x_{14}+x_{23}=x_{13}+x_{24}$. Hence $\eta(a)=2(x_{14}+x_{23})=2(x_{13}+x_{24})>2(x_{12}+x_{34})=2\hint(a).$ We have thus verified the claim in all cases.

Let $a  \in A$ with $\hint(a) < \frac{1}{3}h_{[\GL_4]}(a)$.
In particular, $\eta(a)>0$.
For any $\tau_i>0$ such that
\begin{align}\label{eq:timerestr1}
\tau_i > \tfrac{1}{\decayone(a)}(\tfrac{1}{2}\log(D_{K_i}) + \kappa)
\end{align}
and
\begin{align}\label{eq:timerestr5}
\tau_i \leq \tfrac{1}{2\hint(a)}(\log\absdisc_{K_i} - 3\log\absdisc_{F_i} - \log c)
\end{align}
we have by Theorem~\ref{thm:base}
\begin{align*}
\mu_{Y_i}\times \mu_{Y_i}\big(\{ (x,y): y \in x B_{\tau_i}\}\big) &\ll_{\varepsilon} 
\absdisc_{K_i}^{-\frac{1}{2}+\varepsilon} 
+ \frac{1}{\absdisc_{F_i}^2}\absdisc_{K_i}^{\varepsilon}\mathrm{e}^{-2\tau_i \hint(a)} \\
&\leq \absdisc_{K_i}^{-\frac{1}{2}+\varepsilon} 
+ \absdisc_{K_i}^{\varepsilon}\mathrm{e}^{-2\tau_i \hint(a)}.
\end{align*}
For the second term here to dominate, we need in addition that 
\begin{align}\label{eq:timerestr2}
\tau_i \leq \tfrac{1}{2\hint(a)}(\tfrac{1}{2}-2\varepsilon)\log(D_{K_i}).
\end{align}
By assumption on $a \in A$, we have $\eta(a)>2\hint(a)$; by our assumption we also have $\absdisc_K\gg\absdisc_F^6$. Hence \eqref{eq:timerestr1}, \eqref{eq:timerestr5} and \eqref{eq:timerestr2} are compatible for sufficiently small $\varepsilon$ and sufficiently large $i$.
Let  $\tau_i$ be any time with \eqref{eq:timerestr1} and \eqref{eq:timerestr2}.
In particular, $\tau_i \to \infty$ as $i \to \infty$ and 
 \begin{align*}
 \mu_{Y_i}\times \mu_{Y_i}\big(\{ (x,y): y \in x B_{\tau_i}\}\big) \ll_\varepsilon
 \mathrm{e}^{-2\tau_i (\hint(a)-\varepsilon)}.
 \end{align*}
 By Proposition~\ref{prop:Bowentoentropy} this proves the theorem for all $a$ in the interior of $A'$.
The inequality $h_\mu(a) \geq \hint(a)$ passes to the boundary of $A'$ by properties of entropy in homogeneous dynamics, specifically the product formula \cite[Cor.~9.10]{pisaEL} (see also \cite{Hu-entropycomm}).
\end{proof}

\subsection{Further results}
The proof of Theorem~\ref{thm:main} is somewhat wasteful. Indeed, the additional decay given by the discriminants of the intermediate fields is simply discarded.
Of course, in general the discriminant of the intermediate field can grow arbitrarily slowly in comparison to the discriminant of the quartic field.
In the following we prove further results under additional restrictions on the relative growth rates of the discriminants.

\begin{cor}\label{cor:alternativetheorem}
Assume the notations and conditions in Theorem~\ref{thm:main} except for the condition $\absdisc_{K_i} \gg \absdisc_{F_i}^6$. Furthermore, assume that there exists $0 < \alpha< \frac{1}{3}$ with
\begin{align*}
   \absdisc_{K_i}^\alpha \leq \absdisc_{F_i}
\end{align*}
for every $i$.
Let $\beta \leq 2$ be a non-negative number. If there is some $\delta>0$ with 
\begin{align*}
\absdisc_{F_i} \leq \absdisc_{K_i}^{\frac{1}{\max\{2\beta,3\}}-\delta},
\end{align*}
then there is a closed subset $A'(\beta,\delta) \subset A'$ with non-empty interior and with
\begin{align*}
h_\mu(a) \geq \hint(a) + \decayone(a)\alpha\beta
\end{align*}
for every $a \in A'(\beta,\delta)$.
\end{cor}

Note that Corollary~\ref{cor:alternativetheorem} implies Theorem~\ref{thm:refinement of main} as $\decayone(a) \geq 2 \hint(a)$ for every $a \in A'$ as established at the beginning of the proof of Theorem~\ref{thm:main}.
As expected, the bounds in the corollary turn into the bounds from Theorem~\ref{thm:main} when $\alpha\to 0$.
In both cases of Corollary~\ref{cor:alternativetheorem} the assumption implies in particular that the discriminant of the intermediate field is at the same power scale as the discriminant of the quartic field.

\begin{proof}
The corollary merely consists of an adaptation of the arguments in the proof of Theorem~\ref{thm:main}.
We estimate for $\tau_i$ as in  \eqref{eq:timerestr1} and \eqref{eq:timerestr5}:
\begin{align*}
\mu_{Y_i}\times \mu_{Y_i}\big(\{ (x,y): y \in x B_{\tau_i}\}\big) &\ll_{\varepsilon} 
\absdisc_{K_i}^{-\frac{1}{2}+\varepsilon} 
+ \frac{1}{\absdisc_{F_i}^2}\absdisc_{K_i}^{\varepsilon}\mathrm{e}^{-2\tau_i \hint(a)} \\
&\leq \absdisc_{K_i}^{-\frac{1}{2}+\varepsilon} 
+ \absdisc_{F_i}^{-\beta}\absdisc_{K_i}^{\varepsilon}\mathrm{e}^{-2\tau_i \hint(a)}.
\end{align*}
For the second term to dominate, we need
\begin{align}\label{eq:timerestr3}
\tau_i \leq \frac{1}{2\hint(a)}\log\big(\absdisc_{K_i}^{1/2}\absdisc_{F_i}^{-\beta}\big).
\end{align}
Note that \eqref{eq:timerestr1} and \eqref{eq:timerestr3} are compatible if and only if
\begin{align}\label{eq:restrcomp1}
\absdisc_{F_i} \leq \absdisc_{K_i}^{\frac{1}{2\beta}-\frac{\hint(a)}{\beta\decayone(a)}}
\end{align}
Moreover, \eqref{eq:timerestr1} and \eqref{eq:timerestr5} are compatible if and only if
\begin{align}\label{eq:restrcomp2}
\absdisc_{F_i} \leq C(a) \absdisc_{K_i}^{\frac{1}{3}-\frac{\hint(a)}{3\decayone(a)}}
\end{align}
where $C(a)>0$ is a constant depending on $a$.
Under the assumptions in the corollary, \eqref{eq:restrcomp1} and \eqref{eq:restrcomp2} both hold if $\decayone(a)\beta\delta > \hint(a)$.
This defines an open subset $A'(\beta,\delta)$ of $A'$.
Therefore, we have for any $a \in A'(\beta,\delta)$
\begin{align*}
\mu_{Y_i}\times \mu_{Y_i}\big(\{ (x,y): y \in x B_{\tau_i}\}\big)
&\ll_\varepsilon \absdisc_{K_i}^{\varepsilon}\absdisc_{F_i}^{-\beta}\mathrm{e}^{-2\tau_i \hint(a)}\\
&\leq \absdisc_{K_i}^{-\alpha\beta+\varepsilon}\mathrm{e}^{-2\tau_i \hint(a)}
\end{align*}
Choosing $\tau_i$ so that \eqref{eq:timerestr1}, \eqref{eq:timerestr5}, and \eqref{eq:timerestr3} hold, we obtain $\tau_i \to \infty$ as well as
\begin{align*}
\mu_{Y_i}\times \mu_{Y_i}\big(\{ (x,y): y \in x B_{\tau_i}\}\big)
&\ll_\varepsilon \mathrm{e}^{-2\tau_i\decayone(a)\alpha\beta}\mathrm{e}^{-2\tau_i (\hint(a)-\varepsilon)}
\end{align*}
whenever $\tau_i$ is sufficiently close to its lower bound in \eqref{eq:timerestr1}.
Using Proposition~\ref{prop:Bowentoentropy} this proves the corollary for $A'(\beta,\delta)$ and hence also for its closure.
\end{proof}

\bibliographystyle{alpha}
\bibliography{references}

\begin{thebibliography}{ALMW22}

\bibitem[ALMW22]{ALMW}
Menny Aka, Manuel Luethi, Philippe Michel, and Andreas Wieser.
\newblock Simultaneous supersingular reductions of {CM} elliptic curves.
\newblock {\em J. Reine Angew. Math.}, 786:1--43, 2022.

\bibitem[Bai80]{Bai80}
Andrew~Marc Baily.
\newblock On the density of discriminants of quartic fields.
\newblock {\em J. Reine Angew. Math.}, 315:190--210, 1980.

\bibitem[BGS94]{BGS94}
J.-B. Bost, H.~Gillet, and C.~Soul\'{e}.
\newblock Heights of projective varieties and positive {G}reen forms.
\newblock {\em J. Amer. Math. Soc.}, 7(4):903--1027, 1994.

\bibitem[Bha05]{BhargavaQuartic}
Manjul Bhargava.
\newblock The density of discriminants of quartic rings and fields.
\newblock {\em Ann. of Math. (2)}, 162(2):1031--1063, 2005.

\bibitem[BK10]{BK10}
Jean-Beno\^{\i}t Bost and Klaus K\"{u}nnemann.
\newblock Hermitian vector bundles and extension groups on arithmetic schemes. {I}. {G}eometry of numbers.
\newblock {\em Adv. Math.}, 223(3):987--1106, 2010.

\bibitem[Bos20]{Bos20}
Jean-Beno\^{\i}t Bost.
\newblock {\em Theta invariants of {E}uclidean lattices and infinite-dimensional {H}ermitian vector bundles over arithmetic curves}, volume 334 of {\em Progress in Mathematics}.
\newblock Birkh\"{a}user/Springer, Cham, [2020] \copyright 2020.

\bibitem[Duk88]{Duk88}
William Duke.
\newblock Hyperbolic distribution problems and half-integral weight {M}aass forms.
\newblock {\em Invent. Math.}, 92(1):73--90, 1988.

\bibitem[Duk07]{DukeSurvey}
W.~Duke.
\newblock An introduction to the {L}innik problems.
\newblock In {\em Equidistribution in number theory, an introduction}, volume 237 of {\em NATO Sci. Ser. II Math. Phys. Chem.}, pages 197--216. Springer, Dordrecht, 2007.

\bibitem[EKL06]{EKL06}
Manfred Einsiedler, Anatole Katok, and Elon Lindenstrauss.
\newblock Invariant measures and the set of exceptions to {L}ittlewood's conjecture.
\newblock {\em Ann. of Math. (2)}, 164(2):513--560, 2006.

\bibitem[EL08]{EL1}
Manfred Einsiedler and Elon Lindenstrauss.
\newblock On measures invariant under diagonalizable actions: the rank-one case and the general low-entropy method.
\newblock {\em J. Mod. Dyn.}, 2(1):83--128, 2008.

\bibitem[EL10]{pisaEL}
M.~Einsiedler and E.~Lindenstrauss.
\newblock Diagonal actions on locally homogeneous spaces.
\newblock In {\em Homogeneous flows, moduli spaces and arithmetic}, volume~10 of {\em Clay Math. Proc.}, pages 155--241. Amer. Math. Soc., Providence, RI, 2010.

\bibitem[EL15]{EL2}
Manfred Einsiedler and Elon Lindenstrauss.
\newblock On measures invariant under tori on quotients of semisimple groups.
\newblock {\em Ann. of Math. (2)}, 181(3):993--1031, 2015.

\bibitem[ELMV09]{ELMV09Duke}
Manfred Einsiedler, Elon Lindenstrauss, Philippe Michel, and Akshay Venkatesh.
\newblock Distribution of periodic torus orbits on homogeneous spaces.
\newblock {\em Duke Math. J.}, 148(1):119--174, 2009.

\bibitem[ELMV11]{ELMV11Annals}
Manfred Einsiedler, Elon Lindenstrauss, Philippe Michel, and Akshay Venkatesh.
\newblock Distribution of periodic torus orbits and {D}uke's theorem for cubic fields.
\newblock {\em Ann. of Math. (2)}, 173(2):815--885, 2011.

\bibitem[ELMV12]{ELMV12}
Manfred Einsiedler, Elon Lindenstrauss, Philippe Michel, and Akshay Venkatesh.
\newblock The distribution of closed geodesics on the modular surface, and {D}uke's theorem.
\newblock {\em Enseign. Math. (2)}, 58(3-4):249--313, 2012.

\bibitem[EP80]{EP80}
Hugh Edgar and Brian Peterson.
\newblock Some contributions to the theory of cyclic quartic extensions of the rationals.
\newblock {\em J. Number Theory}, 12(1):77--83, 1980.

\bibitem[GS91]{GS91}
H.~Gillet and C.~Soul\'{e}.
\newblock On the number of lattice points in convex symmetric bodies and their duals.
\newblock {\em Israel J. Math.}, 74(2-3):347--357, 1991.

\bibitem[Hu93]{Hu-entropycomm}
Hu~Yi Hu.
\newblock Some ergodic properties of commuting diffeomorphisms.
\newblock {\em Ergodic Theory Dynam. Systems}, 13(1):73--100, 1993.

\bibitem[Iwa87]{Iwa87}
Henryk Iwaniec.
\newblock Fourier coefficients of modular forms of half-integral weight.
\newblock {\em Invent. Math.}, 87(2):385--401, 1987.

\bibitem[Kem78]{Kem78}
George~R. Kempf.
\newblock Instability in invariant theory.
\newblock {\em Ann. of Math. (2)}, 108(2):299--316, 1978.

\bibitem[Kha19a]{Kha15}
Ilya Khayutin.
\newblock Arithmetic of double torus quotients and the distribution of periodic torus orbits.
\newblock {\em Duke Math. J.}, 168(12):2365--2432, 2019.

\bibitem[Kha19b]{Khayutin-mixing}
Ilya Khayutin.
\newblock Joint equidistribution of {CM} points.
\newblock {\em Ann. of Math. (2)}, 189(1):145--276, 2019.

\bibitem[Lin68]{Lin68}
Yu.~V. Linnik.
\newblock {\em Ergodic properties of algebraic fields}.
\newblock Translated from the Russian by M. S. Keane. Ergebnisse der Mathematik und ihrer Grenzgebiete, Band 45. Springer-Verlag New York Inc., New York, 1968.

\bibitem[MFK94]{MFK94}
D.~Mumford, J.~Fogarty, and F.~Kirwan.
\newblock {\em Geometric invariant theory}, volume~34 of {\em Ergebnisse der Mathematik und ihrer Grenzgebiete (2) [Results in Mathematics and Related Areas (2)]}.
\newblock Springer-Verlag, Berlin, third edition, 1994.

\bibitem[Mic04]{Michel-subconvex}
P.~Michel.
\newblock The subconvexity problem for {R}ankin-{S}elberg {$L$}-functions and equidistribution of {H}eegner points.
\newblock {\em Ann. of Math. (2)}, 160(1):185--236, 2004.

\bibitem[MV06]{MV-ICM}
Philippe Michel and Akshay Venkatesh.
\newblock Equidistribution, {$L$}-functions and ergodic theory: on some problems of {Y}u. {L}innik.
\newblock In {\em International {C}ongress of {M}athematicians. {V}ol. {II}}, pages 421--457. Eur. Math. Soc., Z\"{u}rich, 2006.

\bibitem[PV94]{PV94}
V.~L. Popov and E.~B. Vinberg.
\newblock {\em Invariant Theory}, pages 123--278.
\newblock Springer Berlin Heidelberg, Berlin, Heidelberg, 1994.

\bibitem[Shi10]{Shi10}
Goro Shimura.
\newblock {\em Arithmetic of Quadratic Forms}.
\newblock Springer New York, NY, 2010.

\bibitem[Sku62]{skubenko}
B.~F. Skubenko.
\newblock The asymptotic distribution of integers on a hyperboloid of one sheet and ergodic theorems.
\newblock {\em Izv. Akad. Nauk SSSR Ser. Mat.}, 26:721--752, 1962.

\bibitem[Voi21]{voight}
John Voight.
\newblock {\em Quaternion algebras}, volume 288 of {\em Graduate Texts in Mathematics}.
\newblock Springer, 2021.

\bibitem[Wie19]{W-linnik}
Andreas Wieser.
\newblock Linnik's problems and maximal entropy methods.
\newblock {\em Monatsh. Math.}, 190(1):153--208, 2019.

\end{thebibliography}

\end{document}